\documentclass[english,12pt]{amsart}

\usepackage{amsmath,amssymb,amsfonts,amscd}

\usepackage[utf8]{inputenc}
\usepackage[T1]{fontenc}
\usepackage[english]{babel}

\usepackage{euscript}
\usepackage[normalem]{ulem}

\usepackage{fouriernc}
\usepackage[scaled=0.875]{helvet}
\usepackage{courier}

\usepackage{MnSymbol}
\usepackage{wasysym}
\usepackage{marvosym}

\usepackage[babel=true]{csquotes}

\usepackage[marginparwidth=2cm]{geometry}
\usepackage[usenames,dvipsnames,svgnames,table]{xcolor}
\usepackage[bookmarks=true,colorlinks=true,urlcolor=DarkBlue,linkcolor=DarkRed,citecolor=DarkGreen,pagebackref=true,pdfstartview=FitV]{hyperref}

\usepackage[normalem]{ulem}
\usepackage[marginpar]{todo}

\usepackage{tikz}
\usetikzlibrary{shapes,arrows}
\usetikzlibrary{fit,positioning}
\usetikzlibrary{patterns,decorations.pathreplacing}
\usepackage{xkeyval}
\usepackage{moreverb}
\usepackage{epic}
\usepgfmodule{shapes,plot,decorations}
\usepackage{accents}

\usepackage{epsfig}
\usepackage{pinlabel}
\usepackage{subfig}

\usepackage{lineno}

\newcommand{\Z}{\mathbb{Z}}

\newcommand{\R}{\mathbb{R}}
\newcommand{\C}{\mathcal{C}} 

\newcommand{\G}{\Gamma}

\renewcommand{\C}{\mathcal{C}}
\newcommand{\U}{\mathcal{U}}
\newcommand{\F}{\sigma}

\newcommand{\V}{\mathcal{V}}

\newcommand{\GG}{\mathcal{G}}

\newcommand{\T}{\mathcal{T}}

\renewcommand{\O}{\Omega}

\newcommand{\dO}{\partial \Omega}

\renewcommand{\S}{\mathbb{S}}

\newcommand*{\Quotient}[2]{\ensuremath{#1/\!\raisebox{-.90ex}{\ensuremath{#2}}}}

\newcommand{\Aut}{\textrm{Aut}}

\newcommand{\Rank}{\textrm{rank}}

\newcommand{\Pyr}{\textrm{Pyr}}

\newcommand{\B}{\beta}

\newcommand{\Cm}{\mathfrak{C}}

\newcommand{\ie}{i.e.\ }

\newcommand{\Vp}{\mathcal{V}_{\mathcal{P}}}

\renewcommand{\leq}{\leqslant}
\renewcommand{\geq}{\geqslant}

\usepackage{aurical}
\renewcommand{\B}{\textrm{\Fontauri C\normalfont}}

\theoremstyle{plain}
\newtheorem{theorem}{Theorem}[section]
\newtheorem{qu}[theorem]{Question}
\newtheorem{propo}[theorem]{Proposition}
\newtheorem{cor}[theorem]{Corollary}
\newtheorem{lemma}[theorem]{Lemma}

\newtheorem{theointro}{Theorem}[section]

\theoremstyle{definition}

\newtheorem{de}[theorem]{Definition}

\theoremstyle{remark}
\newtheorem{rem}[theorem]{Remark}

\title{Convex projective generalized Dehn filling}

\author{
Suhyoung Choi
}
\address{Department of Mathematical Sciences, KAIST,  Republic of Korea}
\email{schoi@math.kaist.ac.kr}

\author{
Gye-Seon Lee
}
\address{Mathematisches Institut, Universit\"{a}t Heidelberg, Germany}
\email{lee@mathi.uni-heidelberg.de}

\author{
Ludovic Marquis
}
\address{Univ Rennes, CNRS, IRMAR - UMR 6625, F-35000 Rennes, France}
\email{ludovic.marquis@univ-rennes1.fr}

\begin{document}

\newcommand\point{\textbullet}
\newcommand*\points[1]{%
 \ifcase\value{#1}\or
   $\cdot$ \or $\udotdot$
    \or $\therefore$ \or
  $\diamonddots$
   \or  $\fivedots$
   \or $\davidsstar$
   \or 7
 \fi}
\renewcommand\theenumi{%
 \points{enumi}}
\renewcommand\labelenumi{%
 \points{enumi}}

\let\oldmarginpar\marginpar
\renewcommand\marginpar[1]{\-\oldmarginpar[\raggedleft\tiny #1]%
{\raggedright\tiny #1}}

\begin{abstract}
For $d=4, 5, 6$, we exhibit the first examples of complete finite volume hyperbolic $d$-manifolds $M$ with cusps such that infinitely many $d$-orbifolds $M_{m}$ obtained from $M$ by generalized Dehn filling admit properly convex real projective structures. The orbifold fundamental groups of $M_m$ are Gromov-hyperbolic relative to a collection of subgroups virtually isomorphic to $\mathbb{Z}^{d-2}$, hence the images of the developing maps of the projective structures on $M_m$ are new examples of divisible properly convex domains of the projective $d$-space which are not strictly convex, in contrast to the previous examples of Benoist. 
\end{abstract}



\maketitle


\tableofcontents

\section{Introduction}

\subsection{Motivation}

Hyperbolic Dehn filling theorem proven by Thurston \cite{Thurston:2002} is a fundamental theorem of hyperbolic $3$-manifold theory. It states that if the interior of a compact $3$-manifold $M$ with toral boundary admits a complete hyperbolic structure of finite volume, then except for finitely many Dehn fillings on each boundary component, all other Dehn fillings of $M$ admit hyperbolic structures. 

\medskip

Within the realm of hyperbolic geometry, this phenomenon happens only for \emph{three}-manifolds even though there is a notion of topological Dehn filling for every compact $d$-manifold $M$ with boundary homeomorphic to the $(d-1)$-dimensional torus $\mathbb{T}^{d-1}$. Let $\mathbb{D}^{n}$ be the closed $n$-ball. Since the compact $d$-manifold $\mathbb{D}^2 \times \mathbb{T}^{d-2}$ has also a torus boundary component, if we glue $M$ and $\mathbb{D}^2 \times \mathbb{T}^{d-2}$ together along their boundaries, then we obtain closed $d$-manifolds, called \emph{Dehn fillings of $M$}. In a similar way we can deal with compact $d$-manifolds with multiple torus boundary components. Now even if we assume that the interior of $M$ admits a finite volume hyperbolic structure, no Dehn fillings of $M$ admit a hyperbolic structure when $d \geqslant 4$ (see Theorem \ref{garland_raghunathan}).

\medskip

Although Dehn fillings in dimension bigger than $3$ admit no hyperbolic structure, they can admit a geometric structure which is “larger” than a hyperbolic structure. A main aim of this article is to show that a real projective structure can be a good candidate for this purpose. In a similar spirit, Anderson \cite{MR2225518} and Bamler \cite{MR2911887} proved that many aspects of Dehn filling theory for hyperbolic $3$-manifolds can be generalized to Einstein $d$-manifolds in every dimension $d \geqslant 3$. More precisely, if the interior of a compact $d$-manifold $M$ with toral boundary admits a finite volume hyperbolic structure, then except for finitely many Dehn fillings on each boundary component, all other Dehn fillings of $M$ admit Einstein metrics.

\medskip

Let $\mathbb{S}^d$ be the $d$-dimensional projective sphere\footnote{The \emph{projective $d$-sphere} $\mathbb{S}^d$ is the space of half-lines of $\R^{d+1}$.} and note that the group $\mathrm{SL}_{d+1}^{\pm}(\mathbb{R})$ of linear automorphisms of $\mathbb{R}^{d+1}$ with determinant $\pm 1$ is the group of projective automorphisms of $\mathbb{S}^d$. We say that a $d$-dimensional manifold $N$ \emph{admits a properly convex real projective structure }if $N$ is homeomorphic to $\Omega/\Gamma$ with $\Omega$ a properly convex subset of $\mathbb{S}^d$ and $\Gamma$ a discrete subgroup of $\mathrm{SL}_{d+1}^{\pm}(\mathbb{R})$ acting properly discontinuously on $\Omega$. Now we can ask the following question:

\begin{qu}\label{Main_question}
Is there a compact manifold $M$ of dimension $d \geqslant 4$ with toral boundary such that the interior of $M$ admits a finite volume hyperbolic structure, and except for finitely many Dehn fillings on each boundary component, all other Dehn fillings of $M$ admit properly convex real projective structures? 
\end{qu}

\begin{rem}\label{question:generalize}
We can generalize Question \ref{Main_question} only requiring the interior of $M$ to admit a properly convex real projective structure of finite volume.
\end{rem}

\subsection{Evidence}

In this paper, we give an evidence towards a positive answer to Question \ref{Main_question}. It is difficult for us to find such a manifold directly, and hence we begin with Coxeter orbifold, also called reflection orbifold or reflectofold (see \cite{Thurston:2002,ALR,msjbook}  for the definition of orbifold). The definition of Coxeter orbifold is somewhat more complicated than the definition of manifold, however it turns out that convex projective Coxeter orbifolds are easier to build than convex projective manifolds. The first examples towards Question \ref{Main_question} are hyperbolic Coxeter $d$-orbifolds for $d=4, 5, 6$.

\medskip

The projective model of the hyperbolic $d$-space $\mathbb{H}^d$ is a round open ball in the projective $d$-sphere $\mathbb{S}^d$. Consider a (convex) polytope $P$ in an affine chart of $\mathbb{S}^d$ such that every facet\footnote{A \emph{facet} of a polytope is a face of codimension $1$.} of $P$ has a non-empty intersection with $\mathbb{H}^d$ and all dihedral angles of $P$ are submultiples of $\pi$, \ie each dihedral angle is $\frac{\pi}{m}$ with an integer $m \geqslant 2$ or $m=\infty$. We call such a polytope a \emph{hyperbolic Coxeter polytope}. The group $\Gamma$  generated by the reflections about the facets of $P$ is a discrete subgroup of the group $\mathrm{Isom}(\mathbb{H}^d)$ of isometries of $\mathbb{H}^d$, which is isomorphic to $\mathrm{O}^+_{d,1}(\R)$, and the quotient $\mathbb{H}^d/\Gamma$ is a hyperbolic Coxeter $d$-orbifold. Hyperbolic Coxeter orbifolds are useful objects that we can concretely construct. For example, the first example of a closed orientable hyperbolic $3$-manifold is an eight-fold cover of a right-angled hyperbolic Coxeter $3$-orbifold with 14 faces, which was constructed by L\"{o}bell \cite{lobell} in 1931.

\begin{figure}[ht!]
\begin{tabular}{>{\centering\arraybackslash}m{.4\textwidth} c >{\centering\arraybackslash}m{.4\textwidth}}
\definecolor{ffqqqq}{rgb}{1,0,0}
\begin{tikzpicture}[line cap=round,line join=round,>=triangle 45,x=1.0cm,y=1.0cm]
\clip(-3.42,-4.5) rectangle (3.18,3.06);
\draw [dotted] (-2.5,2)-- (-2,2.5);
\draw [dotted] (-2,2.5)-- (-2,2);
\draw [dotted] (-2.5,2)-- (-2,2);
\draw [dotted] (-2.5,-2)-- (-2,-2);
\draw [dotted] (-2,-2)-- (-2,-2.5);
\draw [dotted] (-2,-2.5)-- (-2.5,-2);
\draw [line width=1.6pt] (-2,-2.5)-- (2,-2.5);
\draw [dotted] (2,-2.5)-- (2.5,-2);
\draw [dotted] (2.5,-2)-- (2,-2);
\draw [dotted] (2,-2)-- (2,-2.5);
\draw [dotted] (2,2.5)-- (2.5,2);
\draw [dotted] (2.5,2)-- (2,2);
\draw [dotted] (2,2.5)-- (2,2);
\draw [line width=1.6pt] (2.5,-2)-- (2.5,2);
\draw [line width=1.6pt] (2,2.5)-- (-2,2.5);
\draw [line width=1.6pt] (-2.5,2)-- (-2.5,-2);
\draw [line width=1.6pt,color=ffqqqq] (-2,2)-- (0,0);
\draw [line width=1.6pt,color=ffqqqq] (0,0)-- (-2,-2);
\draw [line width=1.6pt,color=ffqqqq] (0,0)-- (2,-2);
\draw [line width=1.6pt,color=ffqqqq] (0,0)-- (2,2);
\draw (-0.2,3) node[anchor=north west] {$y$};
\draw (-3,0) node[anchor=north west] {$x$};
\draw (2.5,0) node[anchor=north west] {$z$};
\draw (-0.2,-2.5) node[anchor=north west] {$w$};

\draw (-1,1.4) node[anchor=north west] {$\frac\pi 2$};
\draw (0.5,1.4) node[anchor=north west] {$\frac\pi 2$};
\draw (0.5,-0.8) node[anchor=north west] {$\frac\pi 2$};
\draw (-1,-0.8) node[anchor=north west] {$\frac\pi 2$};

\draw (-0.2,0.5) node[anchor=north west] {$v$};
\fill [color=black] (0,0) circle (3pt);

\draw[decorate,decoration={brace,raise=0.1cm}] (3,-3.5) -- (-3,-3.5);
\draw (0.2,-4) node[] {$\mathcal{O}_{\infty}$};
\end{tikzpicture}
&
$\quad \rightsquigarrow \quad$
&
\definecolor{ffqqqq}{rgb}{1,0,0}
\begin{tikzpicture}[line cap=round,line join=round,>=triangle 45,x=1.0cm,y=1.0cm]
\clip(-4.23,-4.5) rectangle (4.38,3.23);
\draw [dotted] (-3.5,2)-- (-3,2);
\draw [dotted] (-3,2.5)-- (-3,2);
\draw [dotted] (-3,2.5)-- (-3.5,2);
\draw [dotted] (-3.5,-2)-- (-3,-2);
\draw [dotted] (-3,-2)-- (-3,-2.5);
\draw [dotted] (-3.5,-2)-- (-3,-2.5);
\draw [dotted] (3,-2.5)-- (3,-2);
\draw [dotted] (3,-2)-- (3.5,-2);
\draw [dotted] (3.5,-2)-- (3,-2.5);
\draw [dotted] (3,2.5)-- (3.5,2);
\draw [dotted] (3.5,2)-- (3,2);
\draw [line width=1.6pt] (3,2.5)-- (3,2);
\draw [line width=1.6pt] (-3,2.5)-- (3,2.5);
\draw [line width=1.6pt] (3.5,2)-- (3.5,-2);
\draw [line width=1.6pt] (3,-2.5)-- (-3,-2.5);
\draw [line width=1.6pt] (-3.5,2)-- (-3.5,-2);
\draw [line width=1.6pt] (-1,0)-- (1,0);
\draw [line width=1.6pt,color=ffqqqq] (-3,2)-- (-1,0);
\draw [line width=1.6pt,color=ffqqqq] (-1,0)-- (-3,-2);
\draw [line width=1.6pt,color=ffqqqq] (1,0)-- (3,-2);
\draw [line width=1.6pt,color=ffqqqq] (1,0)-- (3,2);
\draw (-4,0) node[anchor=north west] {$x$};
\draw (-0.2,3) node[anchor=north west] {$y$};
\draw (3.5,0) node[anchor=north west] {$z$};
\draw (-0.2,-2.5) node[anchor=north west] {$w$};
\draw (-0.2,0.8) node[anchor=north west] {$\frac\pi m$};

\draw (-2,1.4) node[anchor=north west] {$\frac\pi 2$};
\draw (1.5,1.4) node[anchor=north west] {$\frac\pi 2$};
\draw (1.5,-0.8) node[anchor=north west] {$\frac\pi 2$};
\draw (-2,-0.8) node[anchor=north west] {$\frac\pi 2$};

\fill [color=black] (-1,0) circle (3pt);
\fill [color=black] (1,0) circle (3pt);

\draw[decorate,decoration={brace,raise=0.1cm}] (3,-3.5) -- (-3,-3.5);
\draw (0.2,-4) node[] {$\mathcal{O}_m$};
\end{tikzpicture}
\end{tabular}
\caption{Hyperbolic Dehn fillings in dimension $3$.}
\label{fig:dehnfilling}
\end{figure}
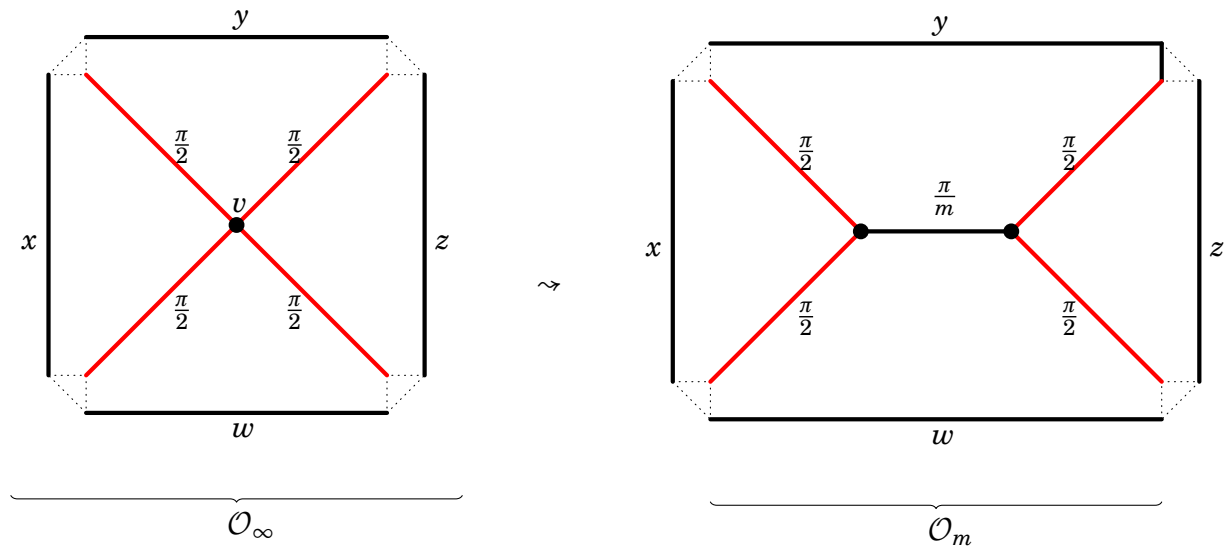

By Andreev's theorems \cite{MR0259734,MR0273510}, there is a version of Dehn filling theorem for hyperbolic Coxeter $3$-orbifolds (see Chapter 7 of  Vinberg and Shvartsman \cite{MR1254933} and Proposition 2 of Kolpakov \cite{MR2950475}): Let $\mathcal{O}_{\infty}$ be a compact Coxeter $3$-orbifold with boundary $\partial \mathcal{O}_\infty$ that is a closed Coxeter $2$-orbifold admitting a Euclidean structure. An \emph{$m$-generalized Dehn filling} $\mathcal{O}_m$ of $\mathcal{O}_{\infty}$, or simply an {\em $m$-Dehn filling}, is a closed Coxeter $3$-orbifold $\mathcal{O}_m$ such that $\mathcal{O}_{\infty}$ is orbifold diffeomorphic\footnote{See Davis \cite{davis_when} or Wiemeler \cite{wiemeler}.} to the complement of an open neighborhood of an edge $r$ of order\footnote{An edge $r$ of $\mathcal{O}_m$ is said to be {\em of order $m$} if each interior point of $r$ has a neighborhood modeled on $(\mathbb{R}^{2}/D_{m}) \times \mathbb{R}$, where $D_m$ is the dihedral group of order $2m$ generated by reflections in two lines meeting at angle $\tfrac{\pi}{m}$.} $m$ of $\mathcal{O}_m$ (see Figure \ref{fig:dehnfilling} and Definition \ref{def:dehn_filling}). A corollary of Andreev's theorems says that if the interior of $\mathcal{O}_{\infty}$ admits a hyperbolic structure of finite volume, then there exists a natural number $N$ such that for each $m > N$, the $3$-orbifold $\mathcal{O}_m$ admits a hyperbolic structure. Note that the existence of a hyperbolic $m$-Dehn filling implies that the boundary $\partial \mathcal{O}_\infty$ of $\mathcal{O}_\infty$ must be a quadrilateral with four right angles as one can see on Figure \ref{fig:dehnfilling} (see Proposition \ref{prop:condition_dehn_filling} for a higher dimensional and projective version of this statement).

\medskip

Now we can state the main theorem of the paper: Let $\mathcal{O}_{\infty}$ be a compact Coxeter $d$-orbifold with boundary $\partial \mathcal{O}_\infty$ that is a closed Coxeter $(d-1)$-orbifold admitting a Euclidean structure. An \emph{$m$-Dehn filling} $\mathcal{O}_m$ of $\mathcal{O}_{\infty}$ is a Coxeter $d$-orbifold such that $\mathcal{O}_{\infty}$ is orbifold diffeomorphic to the complement of an open neighborhood of a ridge\footnote{A \emph{ridge} of a polytope is a face of codimension 2.} $r$ of $\mathcal{O}_m$, and each interior point of $r$ has a neighborhood modeled on $(\mathbb{R}^{2}/D_{m}) \times \mathbb{R}^{d-2}$. We abbreviate a connected open set to a \emph{domain}. 

\begin{theointro}\label{MainThm1}
In dimension $d=4, 5, 6$ (resp. $d=7$), there exists a complete finite volume hyperbolic Coxeter $d$-orbifold $\mathcal{O}_{\infty}$ with orbifold fundamental group $W_{\infty}$ and holonomy representation $\rho_\infty \,:\, W_\infty \rightarrow \mathrm{O}^{+}_{d,1}(\R) \subset \mathrm{SL}^{\pm}_{d+1}(\mathbb{R})$ such that there are a natural number $N$ and a sequence of representations
$$(\,\, \rho_m \,:\, W_\infty \rightarrow \mathrm{SL}^{\pm}_{d+1}(\mathbb{R}) \,\,)_{m > N}$$
satisfying the following:
\begin{itemize}
\item The image $\rho_m(W_\infty)$ is a discrete subgroup of $\mathrm{SL}^{\pm}_{d+1}(\mathbb{R})$ acting properly discontinuously and cocompactly (resp. with finite covolume) on the unique $\rho_m(W_\infty)$-invariant properly convex domain $\O_m \subset \mathbb{S}^d$.

\item The induced representation $W_\infty / \ker(\rho_m) \rightarrow \mathrm{SL}^{\pm}_{d+1}(\mathbb{R})$ is the holonomy representation of a properly convex real projective structure on a Dehn filling $\mathcal{O}_m$ of $\mathcal{O}_{\infty}$.\footnote{More precisely, the Coxeter orbifold $\mathcal{O}_m$ is an $m$-Dehn filling of a compact Coxeter $d$-orbifold $\overline{\mathcal{O}_{\infty}}$ with boundary such that the interior of $\overline{\mathcal{O}_{\infty}}$ is orbifold diffeomorphic to $\mathcal{O}_{\infty}$.}

\item The representations $(\rho_m)_{m > N}$ converge algebraically to $\rho_\infty$.

\item The convex sets $(\overline{\O_m})_{m > N}$ converge to $\overline{\O_\infty} = \overline{\mathbb{H}^d} \subset \mathbb{S}^d$ in the Hausdorff topology.
\end{itemize}
\end{theointro}



\begin{table}[ht!]
\centering
\begin{tabular}{cccc}
\begin{tikzpicture}[thick,scale=0.52, every node/.style={transform shape}]
\node[draw,circle] (3) at (0,0) {};
\node[draw,circle] (2) at (-1.5,0.866) {};
\node[draw,circle] (1) at (-1.5,-0.866) {};

\node[draw,circle] (4) at (1.732,0) {};
\node[draw,circle] (5) at (1.732+1.5,0.866) {};
\node[draw,circle] (6) at (1.732+1.5,-0.866) {};

\draw (1) -- (2)  node[midway,left] {};
\draw (2) -- (3)  node[above,midway] {};
\draw (3)--(1) node[above,midway] {};
\draw (4) -- (5) node[above,near start] {$k$};
\draw (5) -- (6) node[right,midway] {$m$};
\draw (3) -- (4) node[above,midway] {};
\draw (0.8,-1.5) node[]{$k=3,4,5.$} ;
\end{tikzpicture}
&
\begin{tikzpicture}[thick,scale=0.52, every node/.style={transform shape}]
\node[draw,circle] (3) at (0,0) {};
\node[draw,circle] (2) at (-1.5,0.866) {};
\node[draw,circle] (1) at (-1.5,-0.866) {};

\node[draw,circle] (4) at (1.732,0) {};
\node[draw,circle] (5) at (1.732+1.5,0.866) {};
\node[draw,circle] (6) at (1.732+1.5,-0.866) {};

\draw (1) -- (2)  node[midway,left] {};
\draw (2) -- (3)  node[above,midway] {};
\draw (3)--(1) node[above,midway] {};
\draw (4) -- (5) node[above,near start] {$k$};
\draw (5) -- (6) node[right,midway] {$m$};
\draw (3) -- (4) node[above,midway] {};
\draw (4) -- (6) node[below,near start] {$l$};
\draw (0.7,-1.5) node[]{$k,l=3,4,5$ and $ k \geqslant l$.} ;
\end{tikzpicture}
&
\begin{tikzpicture}[thick,scale=0.52, every node/.style={transform shape}]
\node[draw,circle] (3) at (0,0) {};
\node[draw,circle] (2) at (-1.5,0.866) {};
\node[draw,circle] (1) at (-1.5,-0.866) {};

\node[draw,circle] (4) at (1.732,0) {};
\node[draw,circle] (5) at (1.732+1.5,0.866) {};
\node[draw,circle] (6) at (1.732+1.5,-0.866) {};

\draw (1) -- (2)  node[midway,left] {};
\draw (2) -- (3)  node[above,midway] {};
\draw (3)--(1) node[above,midway] {};
\draw (4) -- (5) node[above,midway] {};
\draw (5) -- (6) node[right,midway] {$m$};
\draw (3) -- (4) node[above,midway] {$j$};
\draw (0.7,-1.5) node[]{$j=4,5.$} ;
\end{tikzpicture}
&
\begin{tikzpicture}[thick,scale=0.52, every node/.style={transform shape}]
\node[draw,circle] (3) at (0,0) {};
\node[draw,circle] (2) at (-1.5,0.866) {};
\node[draw,circle] (1) at (-1.5,-0.866) {};

\node[draw,circle] (4) at (1.732,0) {};
\node[draw,circle] (5) at (1.732+1.5,0.866) {};
\node[draw,circle] (6) at (1.732+1.5,-0.866) {};

\draw (1) -- (2)  node[midway,left] {};
\draw (2) -- (3)  node[above,midway] {};
\draw (3)--(1) node[above,midway] {};
\draw (4) -- (5) node[above,midway] {};
\draw (5) -- (6) node[right,midway] {$m$};
\draw (3) -- (4) node[above,midway] {$j$};
\draw (6) -- (4) node[above,midway] {};
\draw (0.7,-1.5) node[]{$j=4,5.$} ;
\end{tikzpicture}
\end{tabular}
\subfloat[\label{Cox_list_dim4} Thirteen examples in dimension $4$]{\hspace{.8\linewidth}}

\begin{tabular}{cc cc cc}
\begin{tikzpicture}[thick,scale=0.6, every node/.style={transform shape}]
\node[draw,circle] (3) at (0,0) {};
\node[draw,circle,below left=0.8cm of 3] (1) {};
\node[draw,circle,above left=0.8cm of 3] (2) {};
\node[draw,circle,above left=0.8cm of 1] (-1) {};

\node[draw,circle,right=0.8cm of 3] (4) {};
\node[draw,circle,above right=0.8cm of 4] (5) {};
\node[draw,circle,below right=0.8cm of 4] (6) {};

\draw (1) -- (-1);
\draw (2) -- (-1);

\draw (2) -- (3);
\draw (3)--(1);
\draw (4) -- (5) node[above,near start] {$p$};
\draw (5) -- (6) node[right,midway] {$m$};
\draw (3) -- (4) node[above,midway] {};
\draw (0.5,-1.5) node[]{$p=3,4,5.$} ;
\end{tikzpicture}
&
\begin{tikzpicture}[thick,scale=0.6, every node/.style={transform shape}]
\node[draw,circle] (3) at (0,0) {};
\node[draw,circle,below left=0.8cm of 3] (1) {};
\node[draw,circle,above left=0.8cm of 3] (2) {};
\node[draw,circle,above left=0.8cm of 1] (-1) {};
\node[draw,circle,right=0.8cm of 3] (4) {};
\node[draw,circle,above right=0.8cm of 4] (5) {};
\node[draw,circle,below right=0.8cm of 4] (6) {};

\draw (1) -- (-1);
\draw (2) -- (-1);

\draw (2) -- (3);
\draw (3)--(1);
\draw (4)--(6) node[below,near start] {$q$};
\draw (4) -- (5) node[above,near start] {$p$};
\draw (5) -- (6) node[right,midway] {$m$};
\draw (3) -- (4) ;
\draw (0.5,-1.5) node[]{$p,q=3,4,5$ and $p \geqslant q$.} ;
\end{tikzpicture}
&
&&
\begin{tikzpicture}[thick,scale=0.6, every node/.style={transform shape}]
\node[draw,circle] (3) at (1,0) {};

\node[draw,circle] (1) at (0.309,-0.951){};
\node[draw,circle] (2) at (0.309,0.951){};
\node[draw,circle] (-1) at (-0.809,-0.587){};
\node[draw,circle] (-2) at (-0.809,0.587){};

\node[draw,circle,right=0.8cm of 3] (4) {};
\node[draw,circle,above right=0.8cm of 4] (5) {};

\node[draw,circle,below right=0.8cm of 4] (6) {};

\draw (3)--(1);
\draw (1) -- (-1);
\draw (-1) -- (-2);
\draw (-2) -- (2);
\draw (2) -- (3);

\draw (4) -- (5);
\draw (5) -- (6) node[right,midway] {$m$};
\draw (3) -- (4) node[above,midway] {};
\draw (0.8,-1.5) node[]{} ;

\end{tikzpicture}
&
\begin{tikzpicture}[thick,scale=0.6, every node/.style={transform shape}]
\node[draw,circle] (3) at (1,0) {};

\node[draw,circle] (1) at (0.309,-0.951){};
\node[draw,circle] (2) at (0.309,0.951){};
\node[draw,circle] (-1) at (-0.809,-0.587){};
\node[draw,circle] (-2) at (-0.809,0.587){};

\node[draw,circle,right=0.8cm of 3] (4) {};
\draw (0.8,-1.5) node[]{} ;
\node[draw,circle,above right=0.8cm of 4] (5) {};

\node[draw,circle,below right=0.8cm of 4] (6) {};

\draw (3)--(1);
\draw (1) -- (-1);
\draw (-1) -- (-2);
\draw (-2) -- (2);
\draw (2) -- (3);

\draw (4) -- (5);
\draw (4) -- (6);
\draw (5) -- (6) node[right,midway] {$m$};
\draw (3) -- (4) node[above,midway] {};
\end{tikzpicture}
\end{tabular}
\\
\subfloat[\label{ex:dim5} Nine examples in dimension $5$]{\hspace{.45\linewidth}}
\subfloat[\label{ex:dim6} Two examples in dimension $6$]{\hspace{.45\linewidth}}

\begin{tabular}{cc}
\begin{tikzpicture}[thick,scale=0.6, every node/.style={transform shape}]
\node[draw,circle] (3) at (1,0) {};

\node[draw,circle] (2) at (0.5,0.866) {};
\node[draw,circle] (1) at (0.5,-0.866) {};
\node[draw,circle] (-3) at (-1,0) {};
\node[draw,circle] (-2) at (-0.5,0.866) {};
\node[draw,circle] (-1) at (-0.5,-0.866) {};

\node[draw,circle,right=0.8cm of 3] (4) {};
\node[draw,circle,above right=0.8cm of 4] (5) {};
\node[draw,circle,below right=0.8cm of 4] (6) {};

\draw (3)--(1);
\draw (1) -- (-1);
\draw (-1) -- (-3);
\draw (-3) -- (-2);
\draw (-2) -- (2);
\draw (2) -- (3);

\draw (4) -- (5);
\draw (5) -- (6) node[right,midway] {$m$};
\draw (3) -- (4) node[above,midway] {};
\end{tikzpicture}
&
\begin{tikzpicture}[thick,scale=0.6, every node/.style={transform shape}]
\node[draw,circle] (3) at (1,0) {};

\node[draw,circle] (2) at (0.5,0.866) {};
\node[draw,circle] (1) at (0.5,-0.866) {};
\node[draw,circle] (-3) at (-1,0) {};
\node[draw,circle] (-2) at (-0.5,0.866) {};
\node[draw,circle] (-1) at (-0.5,-0.866) {};

\node[draw,circle,right=0.8cm of 3] (4) {};
\node[draw,circle,above right=0.8cm of 4] (5) {};
\node[draw,circle,below right=0.8cm of 4] (6) {};

\draw (3)--(1);
\draw (1) -- (-1);
\draw (-1) -- (-3);
\draw (-3) -- (-2);
\draw (-2) -- (2);
\draw (2) -- (3);

\draw (4) -- (5);
\draw (4) -- (6);
\draw (5) -- (6) node[right,midway] {$m$};
\draw (3) -- (4) node[above,midway] {};
\end{tikzpicture}
\end{tabular}
\\
\subfloat[\label{ex:dim7} Two examples in dimension $7$]{\hspace{.5\linewidth}}
\caption[Examples]{Examples for Theorem \ref{MainThm1}}
\label{table:ex1}
\end{table}

\begin{rem}
The authors conjecture that the sequence $(\mathcal{O}_m)_m$ of convex projective orbifolds with Hilbert metric converges to the convex projective orbifold $\mathcal{O}_\infty$ in the Gromov-Hausdorff topology.
\end{rem}

\begin{rem}\label{rem:compa_geo_hyp}
In hyperbolic geometry, a standard assumption to study hyperbolic Coxeter polytopes (in particular, to build compact or finite volume hyperbolic Coxeter polytopes) is that all the open edges of the polytope $P$ of $\mathbb{S}^d$ intersect the hyperbolic space $\mathbb{H}^d \subset \mathbb{S}^d$ (see \cite{survey_hyp_reflec_vinberg} or \cite{MR1254933}). Under this assumption, the geometry of $P$ is governed by the position of its vertices in $\mathbb{S}^d$.

Namely, a vertex $v$ of $P$ is either (\emph{i}) inside the hyperbolic space $\mathbb{H}^d$ i.e. \emph{elliptic}, (\emph{ii}) on the boundary of the hyperbolic space $\partial\mathbb{H}^d$, i.e. \emph{ideal}, or (\emph{iii}) outside the compactification of the hyperbolic space $\overline{\mathbb{H}^d} := \mathbb{H}^d \cup \partial\mathbb{H}^d$, i.e. \emph{hyperideal}. An important observation is that if $v$ is ideal, then $v$ is \emph{cusped}, which means that the intersection of $P$ with a neighborhood of $v$ has finite volume, and that if $v$ is hyperideal, then $v$ is \emph{truncatable}, which means that the dual hyperplane of $v$ with respect to the quadratic form defining $\mathbb{H}^d$ intersects perpendicularly all the open edges emanating from $v$. The main benefit of $v$ being hyperideal is to guarantee the possibility of truncating\footnote{This procedure “truncation” is called an \emph{excision} by Vinberg \cite{survey_hyp_reflec_vinberg}.} the vertex $v$ of $P$ via the dual hyperplane (see Proposition 4.4 of \cite{survey_hyp_reflec_vinberg}). If all the hyperideal vertices of $P$ are truncated, then the resulting polytope $P^{\dagger}$ is of finite volume, and if, in addition, $P$ contains no ideal vertices, then the truncated polytope $P^{\dagger}$ is compact. Moreover, if we begin with two hyperbolic Coxeter polytopes each of which has a hyperideal vertex, then after truncating the hyperideal vertices, we may glue the pair of truncated polytopes along their new facets in place of the hyperideal vertices in order to obtain a new hyperbolic Coxeter polytope, assuming that the new facets match each other.

By the work of Vinberg \cite{MR0302779} and the third author \cite{Marquis:2014aa}, one can give a sense to those definitions in the context of projective geometry. The standard assumption for (projective) Coxeter polytope $P$ is the “2-perfectness” of $P$ (see Section \ref{subsec:def_perfect}). In this new context, the notion of ideal (resp. hyperideal) vertex is replaced by the notion of parabolic (resp. loxodromic) vertex, however the terminology for an elliptic vertex is unchanged. A Coxeter polytope with only elliptic vertices is said to be \emph{perfect}, and a Coxeter polytope with only elliptic or parabolic vertices is \emph{quasi-perfect}.
\end{rem}

\begin{rem}\label{remark:examples}
The examples in Table \ref{table:ex1} are not of the same nature. Namely, all of them are 2-perfect but not all are perfect or quasi-perfect.

Firstly, assuming $j\neq 5$ and $p=q=3$ in Tables \ref{table:ex1}(A), \ref{table:ex1}(B) or \ref{table:ex1}(C), the corresponding Coxeter polytopes $P_m$ are perfect for $N < m < \infty$ and quasi-perfect for $m = \infty$. Those are “prime” examples satisfying Theorem \ref{MainThm1}.

Secondly, assuming $j=5$, $p=5$ and $q=3,5$ in Tables \ref{table:ex1}(A) or \ref{table:ex1}(B), the corresponding Coxeter polytopes $P_m$ are 2-perfect with no parabolic vertices for $N < m < \infty$ and with one parabolic vertex for $m = \infty$, but not perfect for any $m$. We use those examples in Section \ref{sec:gluing} to build infinitely many Coxeter orbifolds $\mathcal{O}_{\infty}$ satisfying Theorem \ref{MainThm1} via a gluing procedure. Note that this gluing trick can be applied only in dimensions $4$ and $5$ (see Remark \ref{rem:up_to_five}).

In any of the two previous cases, the Dehn filling operation happens at the unique parabolic vertex of $P_\infty$.

Thirdly, the assumption that $p=4$ or $q=4$ in Tables \ref{table:ex1}(B) or \ref{table:ex1}(D) provides examples of Coxeter polytopes $P_{\infty}$ with multiple parabolic vertices. The Dehn filling operation happens at a specific parabolic vertex. See Section \ref{subsec:case_p=4} for more details about those examples.
\end{rem}

\begin{rem}
The reader might wonder what “prime” means in this paper. As an ad-hoc definition, a Coxeter $d$-orbifold is {\em prime} if it arises from a $d$-polytope with $d+2$ facets, or as a bit more interesting definition, if it cannot be obtained by the gluing procedure described in Section \ref{sec:gluing}.
\end{rem}

In Theorem \ref{MainThm1}, changing the target Lie group of representations from $\mathrm{O}^{+}_{d,1}(\R)$ to $\mathrm{SL}^{\pm}_{d+1}(\mathbb{R})$ is essential because of the following theorem:

\begin{theorem}[Garland and Raghunathan \cite{MR0267041}]\label{garland_raghunathan}
Let $d$ be an integer bigger than $3$. Assume that $\Gamma$ is a lattice in $\mathrm{O}^{+}_{d,1}(\R)$ and $\rho_{\infty} : \Gamma \to \mathrm{O}^{+}_{d,1}(\R)$ is the natural inclusion. Then there is a neighborhood $U$ of $\rho_{\infty}$ in $\mathrm{Hom}(\Gamma, \mathrm{O}^{+}_{d,1}(\R))$ such that every $\rho \in U$ is conjugate to $\rho_{\infty}$.
\end{theorem}

In other words, every sufficiently small deformation of $\rho_{\infty}$ in $\mathrm{Hom}(\Gamma, \mathrm{O}^{+}_{d,1}(\R))$ is trivial (see also Bergeron and Gelander \cite{MR2110758} for an alternative proof).

\medskip

By (a refined version of) Selberg's lemma, Theorem \ref{MainThm1} also allows us to find properly convex projective structures on Dehn fillings of a {\em manifold}:

\begin{theointro}\label{MainThm2}
In dimension $d=4, 5, 6$, there exists a finite volume hyperbolic $d$-manifold $M_{\infty}$ such that infinitely many closed $d$-orbifolds $M_m$ obtained by generalized Dehn fillings of $M_{\infty}$ admit properly convex real projective structures.
Moreover, the singular locus of $M_{m}$ is a disjoint union of $(d-2)$-dimensional tori, and hence the orbifold fundamental group of $M_{m}$ is not Gromov-hyperbolic. 
\end{theointro}

\begin{rem}
Recently, Martelli and Riolo \cite{martelli_riolo} introduced an interesting concept of “hyperbolic Dehn filling” in dimension $4$, not violating Theorem \ref{garland_raghunathan}. In order to do so, they change the \emph{domain group} of the holonomy representation instead of the \emph{target group} $\mathrm{O}_{4,1}^{+}(\R)$. More precisely, they found a group $\tilde\G$ and a one-parameter family of representations $(\tilde\rho_t : \tilde\G \to \mathrm{O}_{4,1}^{+}(\R))_{t \in [0,1]}$ such that for every $t \in (0,1)$, $\tilde\rho_t$ is the holonomy representation of a complete finite volume hyperbolic \emph{cone}-manifold, and for each $t \in \{0,1\}$, the induced representation ${\tilde\G}/\mathrm{ker}(\tilde\rho_t) \to \mathrm{O}_{4,1}^{+}(\R)$ is the holonomy of a complete finite volume hyperbolic $4$-manifold. A key observation is that the group $\tilde\G$ is the fundamental group of an \emph{incomplete} hyperbolic manifold.

\medskip

There is another difference between the Dehn fillings presented in this paper and the ones presented by Martelli and Riolo. In \cite{martelli_riolo}, the authors manage to hyperbolize some Dehn fillings obtained by closing a cusp of type $\mathbb{S}^1 \times \mathbb{S}^1 \times \mathbb{S}^1$ via collapsing a $\mathbb{S}^1 \times \mathbb{S}^1$ factor. In this paper, we geometrize only Dehn fillings obtained by collapsing a “$\mathbb{S}^1$ factor” of the cusp.\footnote{We work only on Coxeter orbifolds, so the precise statement is that we collapse one ridge to a vertex but do not collapse a face of codimension bigger than 2 to a vertex.}
\end{rem}

\subsection{Another novelty of the examples in Theorem \ref{MainThm1}}

A properly convex domain $\O$ in $\mathbb{S}^d$ is \emph{divisible} if there exists a discrete subgroup $\G$ of $\mathrm{SL}^{\pm}_{d+1}(\mathbb{R})$ such that $\G$ preserves $\O$ and the quotient $\O/\G$ is compact. In that case, we say that $\G$ \emph{divides} $\O$. We call $\O$ {\em decomposable} if there exist non-empty  closed convex subsets $A, B \subset \mathbb{S}^d$ such that the spans of $A$ and $B$ are disjoint and the convex hull of $A \cup B$ is the closure $\overline{\O}$ of $\O$, \ie $\overline{\O}$ is the join of $A$ and $B$. Otherwise, $\O$ is called \emph{indecomposable}. Note that a strictly convex domain must be indecomposable.

\medskip

The real hyperbolic space $\mathbb{H}^d$ thanks to the projective model is the simplest example of divisible properly convex domain. Note that the existence of a group $\G$ dividing $\mathbb{H}^d$, i.e. a uniform lattice of $\mathrm{O}^+_{d,1}(\R)$, is already a non-trivial problem solved by Poincaré \cite{poincare} in dimension $2$, Löbell \cite{lobell} in dimension $3$, Lannér \cite{MR0042129} in dimensions $4$ and $5$, and Borel \cite{borel} in any dimension. 

\medskip

Second, the existence of a divisible strictly convex domain not projectively equivalent to the real hyperbolic space, i.e. not an ellipsoid, was first discovered in dimension $2$ by Kac and Vinberg in \cite{KacVin} and then extended to any dimension $d \geqslant 2$ by the works of Koszul \cite{kos_open}, Johnson and Millson \cite{johnson_millson} and Benoist \cite{divI}. All those examples are obtained by deforming a uniform lattice of $\mathrm{O}^+_{d,1}(\R)$ in $\mathrm{SL}^{\pm}_{d+1}(\mathbb{R})$, so in particular the divisible convex domains of dimension $d$ obtained by those techniques are all quasi-isometric to the real hyperbolic space $\mathbb{H}^d$, by \u{S}varc--Milnor lemma.

\medskip

Third, the existence of a divisible strictly convex domain of dimension $d$ not quasi-isometric to the hyperbolic space $\mathbb{H}^d$ was first discovered in dimension $4$ by Benoist \cite{Benoist_quasi} and then extended to any dimension $d\geqslant 4$ by M. Kapovich \cite{MR2350468}.\footnote{Note that a divisible strictly convex domain of dimension $2$ (resp. $3$) must be quasi-isometric to $\mathbb{H}^2$ (resp. $\mathbb{H}^3$), for trivial reason in dimension $2$, and because any closed $3$-manifold whose fundamental group is Gromov-hyperbolic must carry a hyperbolic metric by Perelman's solution of Thurston's geometrization conjecture.}

\medskip

Now, let us turn to the problem of the existence of an indecomposable non-strictly convex divisible properly convex domain. In \cite{benzecri}, Benzécri showed that an indecomposable divisible properly convex domain of dimension $2$ must be strictly convex. So, the quest starts in dimension $3$.

\medskip

At the time of writing this paper, the only known examples of inhomogeneous\footnote{The classification of homogeneous divisible properly convex domains follows from work of Koecher \cite{MR1718170}, Borel \cite{borel} and Vinberg \cite{MR0201575}.} indecomposable divisible properly convex domains which are not strictly convex are the examples introduced by Benoist \cite{CD4} in dimension $d=3, \dotsc ,7$.\footnote{The existence of indecomposable divisible properly convex domains of every dimension $d$ which are not strictly convex is an open question.} The construction of Benoist has been explored by the third author in \cite{ecima_ludo} and by the second author with Ballas and Danciger in \cite{BDL_3d_convex}, both in dimension $d=3$. Afterwards, the authors \cite{CLM_ecima} study further these kinds of examples in dimension $d=4, \dotsc, 7$. Moreover, the construction in this paper is inspired by the work of  Benoist \cite{Benoist_quasi}. To find more details on the history of divisible convex domains, see \cite{survey_benoist,survey_ludo,survey_CLM}.

\medskip

The difference between the previously known examples and the examples in Theorem \ref{MainThm1} can be visually seen from the notion of properly embedded simplex of dimension $e$ (shorten to $e$-simplex). An open simplex $\Delta$ is \emph{properly embedded} in a properly convex domain $\O$ if $\Delta  \subset \O$ and $\partial \Delta \subset \partial \Omega$. A properly embedded $e$-simplex $\Delta$ of a convex domain $\O$ is \emph{tight} if $\Delta$ is not included in an $e'$-properly embedded simplex $\Delta'$ of $\O$ with $e'>e$.

All the pairs $(\G,\O)$ built in \cite{CD4,ecima_ludo,BDL_3d_convex,CLM_ecima}, assuming that $\G$ divides $\O$, satisfy the three following properties $(\davidsstar)$:

%
\begin{itemize}
\item The $d$-dimensional convex domain $\O$ contains a properly embedded $(d-1)$-simplex.

\item The group $\G$ is Gromov-hyperbolic relative to a finite collection of subgroups of $\G$ virtually isomorphic to $\Z^{d-1}$.

\item The quotient orbifold $\O/\G$ \emph{decomposes into hyperbolic pieces}, i.e. it is homeomorphic to the union along the boundaries of finitely many (but more than one) $d$-orbifolds each of which admits a finite volume hyperbolic structure on its interior.
\end{itemize}

\begin{rem}
If $\O$ is a divisible properly convex domain of dimension $3$ which is not strictly convex and $\G$ is a group dividing $\O$, then by Theorem 1.1 of Benoist \cite{CD4}, the pair $(\G,\O)$ satisfies the three properties $(\davidsstar)$ mentioned above. So the quest for more exotic examples starts in dimension $4$.
\end{rem}

\begin{theointro}\label{whynew}
Assume that $d=4, 5, 6$, $\G = \rho_m (W_{\infty})$, $\O = \O_m$ and $N <  m < \infty$ in Theorem \ref{MainThm1}. Then $(\G,\O)$ satisfies the following:
\begin{itemize}
\item The $d$-dimensional convex domain $\O$ contains a properly embedded tight $(d-2)$-simplex.

\item The group $\G$ is Gromov-hyperbolic relative to a finite collection of subgroups virtually isomorphic to $\Z^{d-2}$.

\item The quotient orbifold $\O/\G$ does {\em not} decompose into hyperbolic pieces.
\end{itemize}
\end{theointro}

The computations involved in the proof of Theorem \ref{MainThm1} in fact easily extend to give two (small) new families of pairs $(\G,\O)$ (see the Coxeter graphs in Table \ref{table:ex2} for Theorem \ref{fewnew} and in Table \ref{ex:mix_examples} for Theorem \ref{thm:mixed}).

\begin{theointro}\label{fewnew}
For each $(d,e) \in \{(5,2), (6,2), (7,3), (8,4)\}$, there exists a discrete group $\G \subset \mathrm{SL}^{\pm}_{d+1}(\mathbb{R})$ which divides $\O \subset \mathbb{S}^d$ such that:
\begin{itemize}
\item The $d$-dimensional convex domain $\O$ contains a properly embedded tight $e$-simplex.

\item The group $\G$ is Gromov-hyperbolic relative to a finite collection of subgroups virtually isomorphic to $\Z^e$.

\item The quotient orbifold $\O/\G$ does {\em not} decompose into hyperbolic pieces.
\end{itemize}
\end{theointro}

\begin{table}[!htbp]
\centering
\begin{tabular}{ccccc}
\begin{tikzpicture}[thick,scale=0.6, every node/.style={transform shape}]
\node[draw,circle] (3) at (0,0) {};
\node[draw,circle,below right=0.8cm of 3] (1) {};
\node[draw,circle,above right=0.8cm of 3] (2) {};
\node[draw,circle,above right=0.8cm of 1] (-1) {};
\node[draw,circle,left=0.8cm of 3] (4) {};
\node[draw,circle,above left=0.8cm of 4] (5) {};
\node[draw,circle,below left=0.8cm of 4] (6) {};

\draw (1) -- (-1) node[below, midway]{$k$};
\draw (2) -- (-1);
\draw (2) -- (3);
\draw (3)--(1);
\draw (1)--(3) ;
\draw (4) -- (5);
\draw (5) -- (6);
\draw (4) -- (6);
\draw (3) -- (4) ;
\draw (0.8,-1.5) node[]{$k=4,5.$} ;
\end{tikzpicture}
&&
\begin{tikzpicture}[thick,scale=0.6, every node/.style={transform shape}]
\node[draw,circle] (3) at (0,0) {};
\node[draw,circle,right=0.8cm of 3] (1) {};
\node[draw,circle,right=0.8cm of 1] (2) {};
\node[draw,circle,right=0.8cm of 2] (-1) {};
\node[draw,circle,left=0.8cm of 3] (4) {};
\node[draw,circle,above left=0.8cm of 4] (5) {};
\node[draw,circle,below left=0.8cm of 4] (6) {};

\draw (1) -- (2) node[above,midway]{$5$};
\draw (2) -- (-1);

\draw (3)--(1);
\draw (1)--(3) ;
\draw (4) -- (5);
\draw (5) -- (6);
\draw (4) -- (6);
\draw (3) -- (4) ;
\draw (0.8,-1.5) node[]{} ;
\end{tikzpicture}
&&
\begin{tikzpicture}[thick,scale=0.6, every node/.style={transform shape}]
\node[draw,circle] (3) at (0,0) {};
\node[draw,circle,right=0.8cm of 3] (1) {};
\node[draw,circle,above right=0.8cm of 1] (2) {};
\node[draw,circle,below right=0.8cm of 1] (-1) {};
\node[draw,circle,left=0.8cm of 3] (4) {};
\node[draw,circle,above left=0.8cm of 4] (5) {};
\node[draw,circle,below left=0.8cm of 4] (6) {};

\draw (1) -- (2) ;
\draw (1) -- (-1) node[below,midway]{$5$};

\draw (3)--(1);
\draw (1)--(3) ;
\draw (4) -- (5);
\draw (5) -- (6);
\draw (4) -- (6);
\draw (3) -- (4) ;
\draw (0.8,-1.5) node[]{} ;
\end{tikzpicture}
\end{tabular}
\\
\subfloat[\label{ex2:dim5} Four examples in dimension $5$]{\hspace{.5\linewidth}}

\begin{tabular}{cc cc}
\begin{tikzpicture}[thick,scale=0.6, every node/.style={transform shape}]
\node[draw,circle] (3) at (1,0) {};
\node[draw,circle] (2) at (2-0.309,0.951){};
\node[draw,circle] (-2) at (2--0.809,0.587){};
\node[draw,circle] (-1) at (2--0.809,-0.587){};
\node[draw,circle] (1) at (2-0.309,-0.951){};
\node[draw,circle,left=0.8cm of 3] (4) {};
\node[draw,circle,above left=0.8cm of 4] (5) {};
\node[draw,circle,below left=0.8cm of 4] (6) {};

\draw (3)--(1);
\draw (1) -- (-1);
\draw (-1) -- (-2)  node[right,midway] {$4$};
\draw (-2) -- (2);
\draw (2) -- (3);

\draw (4) -- (5);
\draw (4) -- (6);
\draw (5) -- (6) node[right,midway] {};
\draw (3) -- (4) node[above,midway] {};
\end{tikzpicture}
&

&&
\begin{tikzpicture}[thick,scale=0.6, every node/.style={transform shape}]
\node[draw,circle] (3) at (1,0) {};
\node[draw,circle, right=0.8cm of 3] (2) {};
\node[draw,circle, right=0.8cm of 2] (-2) {};
\node[draw,circle, right=0.8cm of -2] (-1) {};
\node[draw,circle, right=0.8cm of -1] (1) {};
\node[draw,circle,left=0.8cm of 3] (4) {};
\node[draw,circle,above left=0.8cm of 4] (5) {};
\node[draw,circle,below left=0.8cm of 4] (6) {};

\draw (3)--(2);
\draw (2) -- (-2);
\draw (-2) -- (-1);
\draw (-1) -- (1) node[above,midway] {$5$};
\draw (4) -- (5);
\draw (4) -- (6);
\draw (5) -- (6) node[right,midway] {};
\draw (3) -- (4) node[above,midway] {};
\end{tikzpicture}
\end{tabular}
\\
\subfloat[\label{ex2:dim6} Two examples in dimension $6$]{\hspace{.5\linewidth}}

\begin{tabular}{cc cc}
\begin{tikzpicture}[thick,scale=0.6, every node/.style={transform shape}]
\node[draw,circle] (3) at (0,0) {};
\node[draw,circle,below left=0.8cm of 3] (1) {};
\node[draw,circle,above left=0.8cm of 3] (2) {};
\node[draw,circle,above left=0.8cm of 1] (-1) {};
\node[draw,circle,right=0.8cm of 3] (4) {};
\node[draw,circle,right=0.8cm of 4] (5) {};
\node[draw,circle,right=0.8cm of 5] (6) {};
\node[draw,circle,right=0.8cm of 6] (7) {};
\node[draw,circle,right=0.8cm of 7] (8) {};

\draw (1) -- (-1);
\draw (2) -- (-1);
\draw (2) -- (3);
\draw (3)--(1);
\draw (4) -- (5);
\draw (5) -- (6);
\draw (3) -- (4);
\draw (7) -- (6);
\draw (7) -- (8) node[above,midway] {$5$};
\end{tikzpicture}
&&&
\begin{tikzpicture}[thick,scale=0.6, every node/.style={transform shape}]
\node[draw,circle] (3) at (1,0) {};
\node[draw,circle] (1) at (0.309,-0.951){};
\node[draw,circle] (2) at (0.309,0.951){};
\node[draw,circle] (-1) at (-0.809,-0.587){};
\node[draw,circle] (-2) at (-0.809,0.587){};
\node[draw,circle,right=0.8cm of 3] (4) {};
\node[draw,circle,right=0.8cm of 4] (5) {};
\node[draw,circle,right=0.8cm of 5] (6) {};
\node[draw,circle,right=0.8cm of 6] (7) {};
\node[draw,circle,right=0.8cm of 7] (8) {};

\draw (3)--(1);
\draw (1) -- (-1);
\draw (-1) -- (-2);
\draw (-2) -- (2);
\draw (2) -- (3);
\draw (4) -- (5);
\draw (5) -- (6) node[right,midway] {};
\draw (3) -- (4) node[above,midway] {};
\draw (7) -- (6);
\draw (7) -- (8) node[above,midway] {$5$};
\end{tikzpicture}
\end{tabular}
\subfloat[\label{ex2:dim78} One example in dimension $7$ and one example in dimension $8$]{\hspace{.8\linewidth}}
\caption{Coxeter groups for Theorem \ref{fewnew}}\label{table:ex2}
\end{table}

\begin{theointro}\label{thm:mixed}
For each $d=5,6$, there exists a discrete group $\G \subset \mathrm{SL}^{\pm}_{d+1}(\mathbb{R})$ which divides $\O \subset \mathbb{S}^d$ such that:
\begin{itemize}
\item The $d$-dimensional convex domain $\O$ contains a properly embedded tight $(d-2)$-simplex and a properly embedded tight triangle.

\item The group $\G$ is Gromov-hyperbolic relative to a finite collection of subgroups virtually isomorphic to $\Z^2$ or $\Z^{d-2}$.

\item The quotient orbifold $\O/\G$ does {\em not} decompose into hyperbolic pieces.
\end{itemize} 
\end{theointro}

\begin{table}[ht]
\centering
\begin{tabular}{ccc}
\begin{tikzpicture}[thick,scale=0.6, every node/.style={transform shape}]
\node[draw,circle] (3) at (1,0) {};
\node[draw,circle] (1) at (0,1) {};
\node[draw,circle] (2) at (0,-1) {};
\node[draw,circle] (-1) at (-1,0) {};
\node[draw,circle,right=0.8cm of 3] (4) {};
\node[draw,circle,above right=0.8cm of 4] (5) {};
\node[draw,circle,below right=0.8cm of 4] (6) {};

\draw (1) -- (-1);
\draw (2) -- (-1);

\draw (2) -- (3);
\draw (3)--(1);
\draw (4)--(6) node[below,near start] {};
\draw (4) -- (5) node[above,near start] {};
\draw (5) -- (6) node[right,midway] {};
\draw (3) -- (4) ;
\end{tikzpicture}
&
\begin{tikzpicture}[thick,scale=0.6, every node/.style={transform shape}]
\node[draw,circle] (3) at (1,0) {};

\node[draw,circle] (1) at (0.309,-0.951){};
\node[draw,circle] (2) at (0.309,0.951){};
\node[draw,circle] (-1) at (-0.809,-0.587){};
\node[draw,circle] (-2) at (-0.809,0.587){};

\node[draw,circle,right=0.8cm of 3] (4) {};
\node[draw,circle,above right=0.8cm of 4] (5) {};

\node[draw,circle,below right=0.8cm of 4] (6) {};

\draw (3)--(1);
\draw (1) -- (-1);
\draw (-1) -- (-2);
\draw (-2) -- (2);
\draw (2) -- (3);

\draw (4) -- (5);
\draw (4) -- (6);
\draw (5) -- (6) node[right,midway] {};
\draw (3) -- (4) node[above,midway] {};
\end{tikzpicture}
\end{tabular}
\\
\caption{\label{ex:mix_examples} Coxeter groups for Theorem \ref{thm:mixed}}{\hspace{.8\linewidth}}
\end{table}

Finally, we would like to propose the following question:

\begin{qu}
For which integers $(d,m,e_1, \dotsc, e_m)$ with $m \geqslant 1$ and $2 \leqslant e_1 < e_2 < \dotsm < e_m  \leqslant d-1$, is there a $d$-dimensional divisible convex domain $\O \subset \mathbb{S}^d$ which contains a properly embedded tight $e_i$-simplex for all $i = 1, \dotsc, m$?
\end{qu}

\begin{rem}\label{rem:referee}
Theorems \ref{whynew}, \ref{fewnew} and \ref{thm:mixed} are reminiscent of $\mathrm{CAT}(0)$ \textit{spaces with isolated flats} studied by Hruska and Kleiner \cite{hruska_kleiner}, in which the authors proved that if $X$ is a $\mathrm{CAT}(0)$ space and $\Gamma$ is a group acting geometrically on $X$, then the following are equivalent: (\emph{i}) $X$ has isolated flats, and (\emph{ii}) $\Gamma$ is a relatively Gromov hyperbolic group with respect to a collection of virtually abelian subgroups of rank at least two. By Kelly and Straus \cite{kelly_straus}, any properly convex domain $\O$ with the simplest distance, namely the Hilbert distance, is $\mathrm{CAT}(0)$ if and only if $\O$ is an ellipsoid. Nevertheless, Hilbert geometries and $\mathrm{CAT}(0)$ spaces have a similar flavor in spirit, and so one can ask whether the groups dividing inhomogeneous indecomposable properly convex domains are always relatively Gromov-hyperbolic with respect to a (possibly empty) collection of virtually abelian subgroups of rank at least two. The examples “à la Benoist” and our new examples provide an evidence towards the positive answer to the question. 

Note that Corollary D of Caprace \cite{caprace_cox_rel-hyp} (or Theorem \ref{moussong_caprace}) gives a characterization of the Coxeter groups $W$ which are relatively Gromov-hyperbolic with respect to a collection of virtually abelian subgroups of rank at least 2. As a consequence, the Davis complex (i.e. a $\mathrm{CAT}(0)$ space on which $W$ acts geometrically) of such a Coxeter group $W$ has isolated flats.
\end{rem}

\subsection{Strategy to find the examples in Theorem \ref{MainThm1}}

The simplest $d$-polytope is the $d$-simplex $\Delta_d$, however we will see in Proposition  \ref{prop:condition_dehn_filling} that a hyperbolic Coxeter $d$-polytope which allows us to get Dehn fillings must have a vertex whose link is a Coxeter $(d-1)$-prism with Coxeter group $\tilde{A}_1 \times \tilde{A}_{d-2},$\footnote{We refer to Appendix \ref{classi_diagram} for the classical names of the spherical and affine Coxeter groups.} but the link of every vertex of a $d$-simplex is a $(d-1)$-simplex.

\medskip

The next simplest polytopes have $d+2$ facets, and Tumarkin \cite{MR2086616} classified all finite volume hyperbolic Coxeter $d$-polytopes with $d+2$ facets: They exist only in dimension $d \leqslant 17$, and among them, hyperbolic Coxeter $d$-polytopes $P_\infty$ with a $\tilde{A}_1 \times \tilde{A}_{d-2}$ cusp exist only in dimension $d \leqslant 7$. In this paper (for prime examples), we restrict ourself to $d$-polytopes with $d+2$ facets, and thus Proposition \ref{prop:condition_dehn_filling} (or the work of Tumarkin) implies that the underlying polytope of $P_\infty$ is the pyramid over $\Delta_{1} \times \Delta_{d-2}$. We will see that the underlying polytopes of the Dehn fillings $P_m$ of $P_\infty$ are $\Delta_{2} \times \Delta_{d-2}$ (see Theorem \ref{thm:high_dim_m_infinite}).

\medskip

In the paper \cite{Benoist_quasi}, Benoist uses Coxeter polytopes whose underlying polytopes are $\Delta_{2} \times \Delta_2$ in order to obtain closed $4$-manifolds which admit strictly convex real projective structures but no hyperbolic structure. The computation in Section \ref{sec:the_computation} is inspired by his paper.

\subsection{Organization of the paper}

In Section \ref{sec:prel}, we recall definitions and results essentially due to Tits and Vinberg around the notion of projective Coxeter groups \cite{MR0240238,MR0302779}. In Section \ref{sec:defi_dehn_filling}, we give the definition of projective generalized Dehn filling. In Section \ref{section:labeled}, we explain on three examples (see Table \ref{Cox_gp}) how to build the labeled polytope $\GG_m$ (see Section \ref{subsec:def_labeled} for the definition of labeled polytope) that will support the Coxeter polytopes $P_m$. 

\medskip

In Sections \ref{sec:more_prime} and \ref{sec:geometric_description}, we give the list of all the prime examples we are able to get in dimension $4$, state Theorems \ref{thm:general_4.1} and \ref{thm:general_4.2}, which describe the deformation space $\B(\GG_m)$ of $\GG_m$ and the limit of $P_m$, respectively, and explain the geometry of the different examples. In particular, we explain that the examples fall into three classes. One type (superscript $i=1$ in Table \ref{Cox_gp}) will give examples with deformation space of $\GG_m$ consisting of two points and with deformation space of $\GG_\infty$ consisting of one point. In that case, the interior $\O_\infty$ of the $\Gamma_{\infty}$-orbit of $P_{\infty}$ is an ellipsoid, where $\Gamma_{\infty}$ is the subgroup of $\mathrm{SL}^{\pm}_{d+1}(\mathbb{R})$ generated by the reflections in the facets of $P_\infty$, and the action of $\G_m$ on $\O_m$ is cocompact, \ie $\GG_m$ is perfect (see Section \ref{subsec:def_perfect}). Second type (superscript $i=2$ in Table \ref{Cox_gp}) will give examples with deformation space of $\GG_m$ being two points and with deformation space of $\GG_\infty$ being one point, but this time the action of $\G_m$ on $\O_m$ is not cocompact, but we will be able to truncate some vertices of $\GG_m$ and to glue copies of $\GG_m$ in order to get infinitely many solutions to Theorems \ref{MainThm1} and \ref{whynew}. In that case, $\GG_m$ will be 2-perfect (see Section \ref{subsec:def_perfect}). Third type (superscript $i=3$ in Table \ref{Cox_gp}) will give examples with deformation space of $\GG_m$ being two lines and with deformation space of $\GG_\infty$ being one line. In that case, the action of $\G_m$ on $\O_m$ is cocompact, and $\O_\infty$ is an ellipsoid if and only if the specific parameter $\mu > 0$ of the deformation space is $1$.

\medskip

In Section \ref{sec:the_computation}, we do the actual computation proving the existence of non-empty deformation spaces of $\GG_m$ and of $\GG_{\infty}$. In Section \ref{section: Limit}, we understand the behavior of $P_m$ as the parameter $m\to \infty$. These sections together will show Theorems \ref{thm:general_4.1} and \ref{thm:general_4.2}. Note that we have chosen to present the details of this computation only for the Coxeter groups in Table \ref{Cox_gp}. The computations for the Coxeter groups in Tables \ref{table:ex1}, \ref{table:ex2} and \ref{ex:mix_examples} are similar, but nevertheless we explain how to transfer the four-dimensional argument to the higher dimensional ones in the proof of Theorems \ref{thm:high_dim_m_finite} and \ref{thm:high_dim_m_infinite}.

\medskip

In Section \ref{general_remark}, we show that in order to find Dehn fillings of a Coxeter polytope at a vertex $v$, we need a condition on the link of $v$. This section is logically independent from the rest of the paper. Proposition \ref{prop:condition_dehn_filling} together with Tumarkin's list \cite{MR2086616} is the starting point of our inspiration to find the list of examples presented in this article.

\medskip

In Section \ref{sec:gluing}, we explain how from examples with $i=2$ in Table \ref{Cox_gp}, we can build infinitely many examples by a general procedure: basically cutting two examples around vertices and then gluing them together. In this part, on the way, we find an example of Coxeter polytope whose deformation space is a circle (see Theorem \ref{thm:circle_guy}). First, this is surprising since the deformation space of every known example is non-compact if its dimension is bigger than $0$. Second, this example will allow us to build even more “exotic” solutions to Theorem \ref{whynew} (see Remark \ref{question:generalize}).

\medskip

In Section \ref{section:Higher dimensional}, we study the higher dimensional cases. In Section \ref{sec:main_theo}, we give the proof of the main theorems.

\medskip

An experienced reader can skim through the preliminary, read Sections \ref{sec:defi_dehn_filling}, \ref{section:labeled}, \ref{sec:more_prime}, \ref{sec:geometric_description} and then jump to Section \ref{sec:main_theo}, to have a good overview of the paper.

\subsection*{Acknowledgements}

We are thankful for helpful conversations with Yves Benoist, Seonhee Lim and Anna Wienhard. We also would like to thank the referee for carefully reading this paper, suggesting several improvements and providing us an interesting point of view, which is Remark \ref{rem:referee}. This work benefited from the ICERM 2013 semester workshop on Exotic Geometric Structures, attended by the all three authors, from the MSRI semester on Dynamics on Moduli Spaces of Geometric Structures during Spring 2015, during which the first and second authors were in residence, and from many visits of the second author to the third author, which were supported by the Centre Henri Lebesgue ANR-11-LABX-0020-01.

S. Choi was supported by the National Research Foundation of Korea (NRF) grant funded by the Korea government (MEST) (No.2013R1A1A2056698). G.-S. Lee was supported by the European Research Council under ERC-Consolidator Grant 614733, by the DFG research grant \lq\lq Higher Teichm\"{u}ller Theory\rq\rq\, and by the National Research Foundation of Korea (NRF) grant funded by the Korea government (MEST) (No. 2011-0027952), and he acknowledges support from U.S. National Science Foundation grants DMS 1107452, 1107263, 1107367 “RNMS: Geometric structures And Representation varieties” (the GEAR Network). L. Marquis was supported by the ANR Facets grant and the ANR Finsler grant during the preparation of this paper. 

\section{Preliminary}\label{sec:prel}
\subsection{Coxeter groups}

A \emph{Coxeter system} is a pair $(S,M)$ of a finite set $S$ and a symmetric matrix $M=(M_{st})_{s,t \in S}$ with diagonal entries $M_{ss}=1$ and off-diagonal entries $M_{st} \in \{2,3,\dotsc, m, \dotsc,\infty \}$. To a Coxeter system $(S,M)$ we can associate a \emph{Coxeter group}, denoted by $W_S$ (or simply $W$): the group generated by $S$ with the relations $(st)^{M_{st}}=1$ for every $(s,t) \in S \times S$ such that $M_{st} \neq \infty$. The \emph{rank} of $W$ is the cardinality $|S|$ of $S$.

\medskip

The \emph{Coxeter diagram} (or \emph{Coxeter graph}) of $W$ is a labeled graph, also denoted by $W$, such that the set of vertices (i.e. nodes) of $W$ is $S$, an edge $\overline{st}$ connects two distinct vertices $s,t \in S$ if and only if $M_{st} \neq 2$, and the label of the edge $\overline{st}$ is $M_{st} \in \{3,\dotsc, m, \dotsc,\infty \}$. It is customary to omit the label of the edge $\overline{st}$ when $M_{st} = 3$. A Coxeter group $W$ is \emph{irreducible} if the Coxeter graph $W$ is connected. The \emph{Gram matrix} of $W$, denoted by $Cos(W)$, is an $S \times S$ symmetric matrix whose entries are given by $$(Cos(W))_{st} = -2\cos \left( \frac{\pi}{M_{st}} \right) \quad \,\,\textrm{for every }  s,t \in S.$$

\medskip

An irreducible Coxeter group $W$ (or a connected Coxeter graph $W$) is \emph{spherical} (resp. \emph{affine} resp. \emph{Lann{\'e}r}) if for every $s \in S$, the submatrix of $Cos(W)$ formed by deleting the $s$-th row and $s$-th column is positive definite and the determinant of $Cos(W)$ is positive (resp. zero, resp. negative). If an irreducible Coxeter group $W$ is not spherical nor affine, then $W$ is \emph{large}, \ie there is a homomorphism of a finite index subgroup of $W$ onto a non-abelian free group (see Margulis and Vinberg \cite{MR1748082}). For example, any Lann{\'e}r Coxeter group is large. Remark that an irreducible affine Coxeter group is virtually isomorphic to $\Z^{|S|-1}$. A Coxeter graph $W$ (or a Coxeter group $W$) is called \emph{spherical} (resp. \emph{affine}) if each connected component of $W$ is spherical (resp. affine). Note that a Coxeter group $W$ is spherical if and only if $W$ is a finite group.

\medskip

We refer to Appendix \ref{classi_diagram} for the list of all the irreducible spherical, irreducible affine and Lann{\'e}r Coxeter groups, which we will use frequently.

\subsection{Coxeter polytopes}

Let $V$ be a $(d+1)$-dimensional vector space over $\mathbb{R}$, and let $\S(V)$ be the projective sphere, \ie the space of half-lines in $V$.
  In order to indicate the dimension of $\S(V)$, we use the notation $\S^d$ instead of $\S(V)$. Denote by $\mathrm{SL}^{\pm}(V)$ (resp. $\mathrm{SL}(V)$) the group of linear automorphisms of $V$ with determinant $\pm 1$ (resp. $1$). The projective sphere $\S(V)$ and $\mathrm{SL}^{\pm}(V)$ are two-fold covers of the real projective space $\mathbb{P}(V)$ and  the group $\mathrm{PGL}(V)$ of projective transformations of $\mathbb{P}(V)$, respectively.

\medskip

The natural projection of $V \setminus \{0\}$ onto $\S(V)$ is denoted by $\S$, and for every subset $W$ of $V$, $\S(W)$ denotes $\S(W \setminus \{0\})$ for the simplicity of the notation. A subset $P$ of $\S(V)$ is \emph{convex} if there exists a convex cone\footnote{By a \emph{cone} of $V$, we mean a subset of $V$ which is invariant under multiplication by positive scalars.} $W$ of $V$ such that $P = \mathbb{S}(W)$, and moreover a convex subset $P$ is \emph{properly convex} if the closure $\overline{P}$ of $P$ does not contain a pair of antipodal points. In other words, $P$ is properly convex if and only if there is an affine chart $\mathbb{A}$ of $\S(V)$ such that $\overline{P} \subset \mathbb{A}$ and $P$ is convex in $\mathbb{A}$ in the usual sense. A properly convex set $P$ is \emph{strictly convex} if there is no non-trivial segment in the boundary $\partial P$ of $P$. We note that if $\Omega$ is a properly convex subset of $\mathbb{S}(V)$, then the subgroup $\mathrm{SL}^{\pm}(\Omega)$ of $\mathrm{SL}^{\pm}(V)$ preserving $\Omega$ is naturally isomorphic to a subgroup $\mathrm{PGL}(\Omega)$ of $\mathrm{PGL}(V)$ preserving $p(\Omega)$, where $p$ is the natural projection of $\mathbb{S}(V)$ onto $\mathbb{P}(V)$. We may therefore regard the action of $\mathrm{SL}^{\pm}(\Omega)$ on $\Omega$ as the action of $\mathrm{PGL}(\Omega)$ on $p(\Omega)$ when convenient.

\medskip

A \emph{projective polytope} is a properly convex subset $P$ of $\S(V)$ such that $P$ has a non-empty interior and
$$P = \bigcap_{i=1}^{r} \S(\{ v \in V \mid \alpha_i(v) \leqslant 0 \})$$
where $\alpha_i$, $i=1, \dotsc, r$, are linear forms on $V$. We always assume that the set of linear forms is \emph{minimal}, \ie none of the half space $\S(\{ v \in V \mid \alpha_i(v) \leqslant 0 \})$ contains the intersection of all the others. A \emph{facet} (resp. \emph{ridge}) of a polytope is a face of codimension $1$ (resp. $2$). Two facets $s, t$ of a polytope $P$ are \emph{adjacent} if $s \cap t$ is a ridge of $P$.

\medskip

A \emph{projective reflection} is an element of $\mathrm{SL}^{\pm}(V)$ of order 2 which is the identity on a hyperplane. Each projective reflection $\sigma$ can be written as:
$$\sigma=\mathrm{Id}-\alpha\otimes b$$
where $\alpha$ is a linear form on $V$ and $b$ is a vector in $V$ such that $\alpha(b)=2$. The projective hyperplane $\S(\ker(\alpha))$ is the \emph{support} of $\sigma$ and the projective point $\S(b)$ the \emph{pole} of $\sigma$. A \emph{projective rotation} is an element of $\mathrm{SL}(V)$ which is the identity on a subspace $H$ of codimension 2 and which is conjugate to a matrix
$\left(\begin{smallmatrix}
\cos\theta & -\sin\theta \\
\sin\theta & \hspace{6pt} \cos\theta
\end{smallmatrix}\right)$ on $V/H$ with $0 < \theta < 2 \pi$. The real number $\theta$ is called the \emph{angle} of rotation.

\medskip

A \emph{mirror polytope} is a pair of a projective polytope $P$ of $\S(V)$ and a finite number of projective reflections $(\sigma_s = \mathrm{Id} - \alpha_s \otimes b_s )_{s \in S}$ with $\alpha_s(b_s)=2$ such that: 
\begin{itemize}
\item The index set $S$ consists of all the facets of $P$.
\item For each facet $s \in S$, the support of $\sigma_s$ is the supporting hyperplane of $s$.
\item For every pair of distinct facets $s,t$ of $P$,
$\alpha_s(b_t)$ and $\alpha_t(b_s)$ are either both negative or both zero,
and 
$\alpha_s(b_t) \alpha_t(b_s) \geqslant 4$  or equals $4 \cos^2 \theta$ with $0 < \theta \leqslant \frac{\pi}{2}$.
\end{itemize}
In fact, $\alpha_s(b_t) \alpha_t(b_s) = 4 \cos^2 \theta$  if and only if the product $\sigma_s \sigma_t$ is a rotation of angle $2\theta$. The definition of mirror polytope does not come out of the blue: It is motivated by Theorem \ref{theo_vinberg} of Tits and Vinberg (see Proposition 6 of Vinberg \cite{MR0302779} or Lemma 1.2 of Benoist \cite{fivelectures} for more details).

\medskip

The \emph{dihedral angle} of a ridge $s \cap t$ of a mirror polytope $P$ is $\theta$ if $\sigma_s \sigma_t$ is a rotation of angle $2\theta$, and $0$ otherwise. A \emph{Coxeter polytope} is a mirror polytope whose dihedral angles are submultiples of $\pi$, \ie each dihedral angle is $\frac{\pi}{m}$ with an integer $m \geqslant 2$ or $m=\infty$. If $P$ is a Coxeter  polytope, then the \emph{Coxeter system of $P$} is the Coxeter system $(S,M)$ such that $S$ is the set of facets of $P$ and for every pair of distinct facets $s,t$ of $S$, $M_{st}=m_{st}$ if $\sigma_s \sigma_t$ is a rotation of angle $\frac{2\pi}{m_{st}}$, and $M_{st} =\infty$ otherwise. We denote by $W_P$ (or simply $W$) the \emph{Coxeter group of $P$}, and we call $P$ \emph{irreducible} if $W_P$ is irreducible.

\subsection{Tits-Vinberg's Theorem}

Let $(S,M)$ be a Coxeter system. For each subset $S'$ of $S$, we consider the Coxeter group $W_{S'}$ of the Coxeter subsystem $(S',M')$, where $M'$ is the restriction of $M$ to $S' \times S'$. Theorem \ref{theo_vinberg} shows that the natural homomorphism $W_{S'} \rightarrow W_S$ is injective, and therefore $W_{S'}$ can be identified with the subgroup of $W_S$ generated by $S'$. These subgroups $W_{S'}$ are called \emph{standard subgroups of $W_S$}. If $P$ is a Coxeter  polytope and $f$ is a proper face of $P$ (\ie $f \neq \varnothing$,  $P$), then we write $S_f= \{ s \in S \mid f \subset s \}$ and $W_f := W_{S_f}$.

\begin{theorem}[Tits, Chapter V of \cite{MR0240238} for Tits simplex, and Vinberg \cite{MR0302779}]\label{theo_vinberg}
Let $P$  be a Coxeter  polytope of $\S(V)$ with Coxeter group $W_P$ and let $\G_P$ be the subgroup of $\mathrm{SL}^{\pm}(V)$ generated by the projective reflections $(\sigma_s)_{s \in S}$. Then:
\begin{enumerate}
\item The homomorphism $\sigma:W_P \rightarrow \Gamma_P$ defined by $\sigma(s) =
\sigma_s$ is an isomorphism.

\item The family of polytopes $\big( \gamma(P) \big)_{\gamma \in \Gamma_P}$ is a tilling of a convex subset $\C_P$ of $\S(V)$.

\item The group $\Gamma_P$ is a discrete subgroup of $\mathrm{SL}^{\pm}(V)$ acting properly discontinuously on the interior $\O_P$ of $\C_P$.

\item An open proper face $f$ of $P$ lies in $\O_P$ if and only if the Coxeter group $W_f$ is finite.
\end{enumerate}
\end{theorem}

We call $\Gamma_P$ the \emph{projective Coxeter group} of $P$. Theorem \ref{theo_vinberg} tells us that $\O_P$ is a convex domain of $\S(V)$ and $\O_P/\Gamma_P$ is a convex real projective Coxeter orbifold. 

\subsection{Tits simplex}\label{sub:tits_simplex}

To a Coxeter group $W$ is associated a Coxeter simplex $\Delta_W$ of $\mathbb{S}(\mathbb{R}^S)$, called the \emph{Tits simplex of $W$}, as follows:
\begin{itemize}
\item For each $s \in S$, we set $\alpha_t = e_t^{*}$, where $(e_t^{*})_{t \in S}$ is the canonical dual basis of $\mathbb{R}^S$.

\item For each $t \in S$, we take the unique vector $b_t \in \mathbb{R}^S$ such that $\alpha_s(b_t) = Cos(W)_{st}$ for all $s \in S$.

\item The \emph{Tits simplex} $\Delta_W$ is the pair of the projective simplex $\cap_{s \in S} \mathbb{S}( \{ v \in \mathbb{R}^S \mid \alpha_s(v) \leqslant 0 \} )$ and the set of reflections $(\sigma_s = \mathrm{Id} - \alpha_s \otimes b_s)_{s \in S}$.
\end{itemize}

If $\Gamma_{W}$ denotes the subgroup of $\mathrm{SL}^{\pm}(\mathbb{R}^S)$ generated by the reflections $(\sigma_s)_{s \in S}$ and $\mathrm{span}(b_S)$ the linear span of $(b_s)_{s \in S}$, then by Theorem 6 of Vinberg \cite{MR0302779}, there exists a $\Gamma_{W}$-invariant scalar product $B_W$ on $\mathrm{span} (b_S)$, called the \emph{Tits bilinear form} of $W$, such that $B_W(b_s,b_t) = Cos(W)_{st}$ for every $s, t \in S$.

\subsection{The link of a Coxeter polytope at a vertex}

Given a mirror polytope $P$ of dimension $d$ in $\S^d$, we can associate to a vertex $v$ of $P$ a new mirror polytope $P_v$ of dimension $d-1$ which is \lq\lq $P$ seen from $v$\rq\rq . We call $P_v$ the \emph{link of $P$ at $v$}.
In order to build $P_v$, firstly, look at the projective sphere $\S_v = \S \Big( \Quotient{\R^{d+1}}{ \langle v \rangle} \Big)$, where $\langle v \rangle$ is the subspace spanned by $v$, and secondly, observe that the reflections $(\sigma_s)_{s \in S_v}$ induce reflections on $\S_v$. Finally, we can construct a projective polytope in $\S_v$:
$$\bigcap_{s \in S_v} \S \Big( \Quotient{ \{ x \in \R^{d+1} \mid \alpha_s(x) \leqslant 0 \}}{\langle v \rangle} \Big).$$

\medskip

This projective polytope together with the reflections induced by $(\sigma_s)_{s \in S_v}$ on $\S_v$ gives us the link $P_v$ of $P$ at $v$, which is a mirror polytope of dimension $d-1$. It is obvious that if $P$ is a Coxeter polytope, then $P_v$ is also a Coxeter polytope. Note that if $W_v$ denotes $W_{P_v}$, then the Coxeter graph $W_v$ is obtained from the Coxeter graph $W_P$ by keeping only the vertices in $S_v$ and the edges connecting two vertices of $S_v$.

\subsection{The Cartan matrix of a Coxeter polytope}\label{subsection:Cartan matrix}

An $m \times m$ matrix $A = (a_{ij})_{i,j = 1, \dotsc, m}$ is a \emph{Cartan matrix} if every diagonal entry of $A$ is $2$, every off-diagonal entry is negative or zero, and for all $i,j= 1, \dotsc, m$, $a_{ij} = 0$ if and only if $a_{ji}=0$.

\medskip

A Cartan matrix $A$ is \emph{irreducible} if there are no simultaneous permutations of the rows and the columns of $A$ such that $A$ is a non-trivial block diagonal matrix, \ie it is not a direct sum of two matrices. Every Cartan matrix $A$ decomposes into a direct sum of irreducible matrices, which are called the \emph{components} of $A$. By the Perron--Frobenius theorem, an irreducible Cartan matrix $A$ has a real eigenvalue. An irreducible Cartan matrix $A$ is of \emph{positive type}, \emph{zero type} or \emph{negative type} when the smallest real eigenvalue of $A$ is positive, zero or negative, respectively. If $A^{+}$ (resp. $A^{0}$, resp. $A^{-}$) denotes the direct sum of the components of positive (resp. zero, resp. negative) type of a Cartan matrix $A$, then $A = A^{+} \oplus A^{0} \oplus A^{-}$, \ie the direct sum of $A^{+}$, $A^{0}$ and $A^{-}$.

\medskip

For each Coxeter polytope $P$, we define the \emph{Cartan matrix $A_P$ of $P$} by $(A_P)_{st} = \alpha_s(b_t)$. Note that $A_P$ is irreducible if and only if $P$ is irreducible. The pairs $(\alpha_s,b_s)_{s \in S}$ that determine the Coxeter polytope $P$ are unique up to the action of $|S|$ positive numbers $(\lambda_s)_{s \in S}$ given by: \begin{equation}\label{DiagonalAction}
(\alpha_s,b_s)_{s \in S} \mapsto (\lambda_s \, \alpha_s, \lambda_s^{-1} \, b_s)_{s \in S}
\end{equation}
This leads to an equivalence relation on Cartan matrices: Two Cartan matrices $A$ and $A'$ are equivalent if there exists a positive diagonal matrix $D$ such that $A'=DAD^{-1}$. We denote by $[A]$ the equivalence class of the Cartan matrix $A$.

\medskip

For every sequence $(s_1, \dotsc, s_k)$ of distinct elements of $S$, the number $A_{s_1 s_2} A_{s_2 s_3} \dotsm A_{s_k s_1}$ is independent of the choice of a Cartan matrix in the class $[A]$. These invariants are called the \emph{cyclic products} of $A$, which determine $[A]$ (see Proposition 16 of  \cite{MR0302779}). The importance of Cartan matrix can be seen from:

\begin{theorem}[Corollary 1 of Vinberg \cite{MR0302779}]\label{theo_vin_reali}
Let $A$ be an irreducible Cartan matrix of negative type and of rank $d+1$. Then there exists a mirror polytope $P$ of dimension $d$ such that $A_P= A$. Moreover, $P$ is unique up to conjugations by $\mathrm{SL}^{\pm}(V)$.
\end{theorem}

\subsection{Elliptic and parabolic Coxeter polytopes}

A Coxeter polytope $P$ is \emph{elliptic} (resp. \emph{parabolic}) if the projective Coxeter group $\Gamma_P$ of $P$ is derived from a discrete cocompact group generated by reflections in the sphere (resp. the Euclidean space). These polytopes were classified by Coxeter:

\begin{theorem}[Coxeter \cite{Coxeter_book} and Propositions 21 and 22 of Vinberg \cite{MR0302779}]\label{thm:classi_spheri}
Let $P$ be a Coxeter polytope with Coxeter group $W$. Then $P$ is elliptic if and only if $A_P=A_P^{+}$. If this is the case, then $P$ is isomorphic to the Tits simplex $\Delta_{W}$.
\end{theorem}

In order to describe parabolic Coxeter polytopes, we need to introduce a notation: If $W$ is an irreducible affine Coxeter group, then $\O_{\Delta_W}$ is an affine chart of dimension $\Rank (W) -1$ on which $W$ acts properly and cocompactly, and $W$ preserves an Euclidean norm on $\O_{\Delta_W}$ induced by the Tits bilinear form $B_W$. Now let $W$ be an affine Coxeter group. If we decompose $W$ into irreducible components $(W_i)_{i=1, \dotsc, r}$ and write $W=W_1 \times \cdots \times W_r$, then the Coxeter polytope
 $$\hat{\Delta}_{W} := \Delta_{W_1} \times \cdots \times \Delta_{W_r} \subset \O_{\Delta_{W_1}} \times \cdots \times \O_{\Delta_{W_r}}$$ is a Coxeter polytope in the Euclidean space $\mathbb{E}^e$ of dimension $e = \Rank(W)-r$, where the symbol $\times$ stands for the usual Cartesian product. The Euclidean norm on the affine space $\O_{\Delta_{W_1}} \times \cdots \times \O_{\Delta_{W_r}}$ is the Euclidean norm induced by the Euclidean norms on the factors, and the action of $W$ on $\mathbb{E}^e$ is again proper and cocompact.

\begin{theorem}[Coxeter \cite{Coxeter_book} and Propositions 21 and 23 of Vinberg \cite{MR0302779}]\label{thm:classi_eucli}
Let $P$ be a Coxeter polytope of dimension $d$ with Coxeter group $W$. Then $P$ is parabolic if and only if $A_P=A_P^{0}$ and $A_P$ is of rank $d$. If this is the case, then $P$ is isomorphic to $\hat{\Delta}_{W}$.
\end{theorem}

\subsection{Hilbert geometry}\label{hilbert_geometry}

A Coxeter  polytope $P$ of $\S^d$ is said to be \emph{loxodromic} if $A_P = A_P^-$ and $A_P$ is of rank $d+1$. Before going any further, we remark that if $P$ is an irreducible loxodromic Coxeter polytope of dimension $d$, then the representation $\sigma:W_P \to \mathrm{SL}^{\pm}_{d+1}(\R)$ in Theorem \ref{theo_vinberg} is irreducible and the convex domain $\O_P$ is properly convex (see Lemma 15 of Vinberg \cite{MR0302779}). This situation leads us to study Hilbert geometry.

\medskip

Given a properly convex domain $\O$, we can use the cross-ratio to define a metric $d_{\O}$ on $\O$: For any two distinct points $x$, $y$ in $\O$, the line $l$ passing through $x$ and $y$ meets the boundary $\partial \O$ of $\O$ in two other points $p$ and $q$. Assume that $p$ and $y$ separate $x$ and $q$ on the line $l$ (see Figure \ref{disttt}). The \emph{Hilbert metric} $d_\O : \O \times \O \rightarrow [0, + \infty)$ is defined by:
$$d_{\O}(x,y) =  \displaystyle \frac{1}{2}\log \left( [p:x:y:q] \right) \quad \textrm{for every } x, y \in \O,$$
where $[p:x:y:q]$ is the cross-ratio of the points $p, x, y, q$. The metric topology of $\O$ is the same as the one inherited from $\S(V)$. The metric space $(\O,d_{\O})$ is complete, the closed balls are compact, the group $\Aut(\O)$ of projective transformations that preserve $\O$ acts by isometries on $\O$, and therefore acts properly.

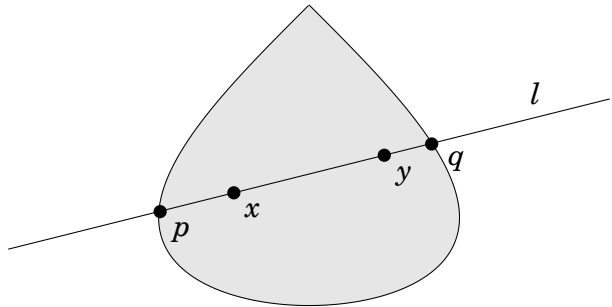
\begin{figure}[ht]
\centering
\begin{tikzpicture}
\filldraw[draw=black,fill=gray!20]
 plot[smooth,samples=200,domain=0:pi] ({4*cos(\x r)*sin(\x r)},{-4*sin(\x r)});
cycle;
\draw (-1,-2.5) node[anchor=north west] {$x$};
\fill [color=black] (-1,-2.5) circle (2.5pt);
\draw (1,-2) node[anchor=north west] {$y$};
\fill [color=black] (1,-2) circle (2.5pt);
\draw [smooth,samples=200,domain=-4:4] plot ({\x},{0.25*\x-2.25});
\draw (-1.97,-2.75) node[anchor=north west] {$p$};
\fill [color=black] (-1.98,-2.75) circle (2.5pt);
\draw (1.7,-1.8) node[anchor=north west] {$q$};
\fill [color=black] (1.63,-1.85) circle (2.5pt);
\draw (2.8,-0.9) node[anchor=north west] {$l$};
\end{tikzpicture}
\caption{The Hilbert metric}
\label{disttt}
\end{figure}

The Hilbert metric is a Finsler metric on $\O$, hence it gives rise to an absolutely continuous measure $\mu_{\O}$ with respect to Lebesgue measure, called the {\em Busemann measure}. We refer to de la Harpe \cite{dlHarpe} or Vernicos \cite{intro_constantin} for more details on Hilbert geometry.

\subsection{Perfect, quasi-perfect and 2-perfect Coxeter polytopes}\label{subsec:def_perfect}
\par{
Vinberg \cite{MR0302779} also shows that if $P$ is a Coxeter polytope, then the following are equivalent:
\begin{itemize}
\item For every vertex $v$ of $P$, the standard subgroup $W_v$ of $W$ is finite.
\item The convex subset $\C_{P}$ of $\mathbb{S}(V)$ is open, \ie $\C_{P}= \O_P$.
\item The action of $\G_{P}$ on $\O_{P}$ is cocompact.
\end{itemize}
 Following him, we call such a Coxeter polytope $P$ \emph{perfect} (see Definition 8 of \cite{MR0302779}). It is known that a perfect Coxeter polytope is either elliptic, parabolic, or irreducible loxodromic (see Proposition 26 of \cite{MR0302779}).

\medskip

A Coxeter polytope $P$ is \emph{2-perfect} if every vertex link of $P$ is perfect, and it is \emph{quasi-perfect} if every vertex link of $P$ is either elliptic or parabolic. It is obvious that a quasi-perfect Coxeter polytope is $2$-perfect. In order to describe the geometric property of irreducible loxodromic $2$-perfect projective Coxeter groups, we introduce the following terminology:
Let $C(\Lambda_P)$ be the convex hull of the limit set\footnote{The limit set of $\G_P$ is the closure of the set of attracting fixed points of loxodromic elements of $\G_P$. It is also the smallest closed $\G_P$-invariant subset of the projective space. See Lemmas 2.9 and 3.3 of Benoist \cite{auto_convex_benoist}.} $\Lambda_P$ of $\G_P$. The action of $\G_P$ on $\O_{P}$ is {\em geometrically finite} if $\mu_{\O_P}(C(\Lambda_P) \cap P) < \infty$, {\em finite covolume} if $\mu_{\O_P}(P) < \infty$, and {\em convex-cocompact} if $C(\Lambda_P) \cap P \subset \O_P$.

\begin{theorem}[Theorem A of Marquis \cite{Marquis:2014aa}]\label{theo_action}
Let $P$ be an irreducible, loxodromic, 2-perfect Coxeter polytope, let $\G_P$ be the projective Coxeter group of $P$, and let $\O_{P}$ be the interior of the $\G_P$-orbit of $P$. Then the action of $\G_P$ on $\O_{P}$ is
\begin{itemize}
\item geometrically finite,
\item of finite covolume if and only if $P$ is quasi-perfect,
\item convex-cocompact if and only if every vertex link of $P$ is elliptic or loxodromic.
\end{itemize}
\end{theorem}

In Remark \ref{rem:compa_geo_hyp}, we explain the version of Theorem \ref{theo_action} in hyperbolic geometry.

\subsection{Deformation spaces of marked Coxeter polytopes}\label{subsec:def_labeled}

Two (convex) polytopes $\GG$ and $\GG'$ are \emph{combinatorially equivalent} if there exists a bijection $\delta$ between the set $\mathcal{F}$ of all faces of $\GG$ and the set $\mathcal{F}'$ of all faces of $\GG'$ such that $\delta$ preserves the inclusion relation, \ie for every $F_1, F_2 \in \mathcal{F}$, $F_1 \subset F_2$ if and only if $\delta(F_1) \subset \delta(F_2)$. We call $\delta$ a \emph{lattice isomorphism} between $\GG$ and $\GG'$. A \emph{combinatorial polytope} is a combinatorial equivalence class of polytopes.
A \emph{labeled polytope} is a pair of a combinatorial polytope $\GG$ and a ridge labeling on $\GG$, which is a function on the set of ridges of $\GG$ to $\{ \frac{\pi}{m} \mid m = 2,3, \dotsc, \infty \}$.

\medskip

Let $\GG$ be a labeled $d$-polytope. A {\em marked Coxeter polytope realizing $\GG$} is a pair of a Coxeter $d$-polytope $P$ of $\S^d$ and a lattice isomorphism $\phi$ between $\GG$ and $P$ such that the label of each ridge $r$ of $\GG$ is the dihedral angle of the ridge $\phi(r)$ of $P$. Two marked Coxeter polytopes $(P,\phi : \GG \rightarrow P)$ and $(P',\phi' : \GG \rightarrow P')$ realizing $\GG$ are {\em isomorphic} if there is a projective automorphism $\psi$ of $\S^d$ such that $\psi(P) = P'$ and $\hat{\psi} \circ \phi = \phi'$, where $\hat{\psi}$ is the lattice isomorphism between $P$ and $P'$ induced by $\psi$. 

\begin{de}
The \emph{deformation space $\B(\GG)$ of a labeled $d$-polytope $\GG$} is the space of isomorphism classes of marked Coxeter $d$-polytopes realizing $\GG$.
\end{de}

A labeled polytope $\GG$ is \emph{rigid} if the deformation space $\B(\GG)$ is the singleton. Otherwise, $\GG$ is \emph{flexible}.

\begin{rem}
A labeled polytope $\GG$ and the deformation space $\B(\GG)$ can be thought of as a Coxeter orbifold $\mathcal{O}$ and the deformation space of marked convex real projective structures on $\mathcal{O}$, respectively (see Thurston \cite{Thurston:2002}).
\end{rem}

In a way similar to Coxeter polytopes, we can associate to a labeled polytope $\GG$ a Coxeter system and therefore a Coxeter group $W_{\GG}$: the generators and the relations of $W_{\GG}$ correspond to the facets and the ridge labeling of $\GG$, respectively. We can also define the \emph{link $\GG_v$ of a labeled polytope $\GG$ at a vertex $v$} using the usual notion of the link of a combinatorial polytope and labeling $\GG_v$ accordingly.

\medskip

A labeled $d$-polytope $\GG$ is {\em spherical} (resp. {\em affine}) if it is the underlying labeled polytope of an elliptic (resp. a parabolic) Coxeter $d$-polytope, and it is {\em Lannér} if the underlying combinatorial polytope of $\GG$ is a $d$-simplex and $W_\GG$ is a Lannér  Coxeter group.  A labeled polytope $\GG$ is {\em irreducible} (resp. {\em large}) if $W_{\GG}$ is irreducible (resp. large), and it is \emph{perfect} (resp. \emph{2-perfect}) if every vertex link of $\GG$ is spherical (resp. perfect). An easy observation is:

\begin{rem}\label{rem:perfect_labeled}
In the case where $\GG$ is a labeled $d$-simplex, the polytope $\GG$ is perfect if and only if $W_{\GG}$ is spherical, irreducible affine or Lannér. If $\GG$ is a labeled $d$-prism with Coxeter group $W_{\GG} = \tilde{A}_1 \times \tilde{A}_{d-1}$ (see Section \ref{polytope_few_facets} for the definition of $d$-prism), then $\GG$ is perfect.
\end{rem}


 Note that we \emph{cannot} define the notion of quasi-perfect labeled polytope because, for example, the labeled triangle $(\frac \pi \infty,\frac \pi \infty,\frac \pi \infty)$ admits a realization for which the link of every vertex is parabolic (e.g. an ideal triangle of $\mathbb{H}^2$) or loxodromic (e.g. a hyperideal triangle of $\mathbb{H}^2$). 

\medskip

To be short, we say that a vertex $v$ of a Coxeter (or labeled) polytope is \emph{$\mathcal{A}$} if the corresponding link is $\mathcal{A}$. Here, the word $\mathcal{A}$ can be: irreducible, spherical, affine, large, Lannér, rigid, flexible, perfect, $2$-perfect, loxodromic and so on.

\subsection{The space of Cartan matrices}

Let $(S,M)$ be a Coxeter system with Coxeter group $W$. A Cartan matrix $A=(a_{st})_{s,t \in S}$ \emph{realizes} $(S,M)$ (or simply $W$) if $a_{st}a_{ts} =  4 \cos^2(\frac{\pi}{M_{st}})$ when $M_{st} \neq  \infty$, and $a_{st}a_{ts} \geqslant 4$ when $M_{st}=\infty$. The \emph{space of Cartan matrices of $W$ of rank $d+1$}, denoted by $\Cm(W,d)$, is the space of equivalence classes of Cartan matrices of rank $d+1$ which realize $W$ (see (\ref{DiagonalAction}) in Section \ref{subsection:Cartan matrix} for the equivalence relation).

\medskip

If $\GG$ is a labeled $d$-polytope with Coxeter group $W_{\GG}$, then 
there is a map
 $$ \Lambda : \B(\GG) \rightarrow \bigcup_{e \leqslant d} \Cm(W_{\GG},e)$$ given by $\Lambda(P) = [A_P]$. 
 We will see that in the cases we are interested in, the image of $\B(\GG)$ under $\Lambda$ is exactly $\Cm(W_{\GG},d)$ and moreover $\B(\GG)$ is homeomorphic to $\Cm(W_{\GG},d)$.

\subsection{Various results of Vinberg}

In order to prove Proposition \ref{single_moduli_finite_case} and others, we need Theorem \ref{combi_structure}, which tells us how the combinatorial structure of a Coxeter polytope $P$ is influenced by the Cartan matrix of $P$. Recall that $S$ is the set of facets of $P$, $A = A_P$ the Cartan matrix of $P$ and for every face $f$ of $P$, $S_f = \{ s \in S \mid f \subset s \}$. Let $\F(P) = \{ S_f \mid f  \textrm{ is a face of } P \} \subset 2^S$. If the dihedral angle between the facets $s$ and $t$ is $\frac\pi 2$, then we express this fact by writing $s\perp t$. For each $T \subset S$,  
 $$ Z(T) = \{ s \in S \mid s\perp t \, \textrm{ for all } t \in T\} \quad \textrm{and} \quad A_T = (A_{st})_{s,t \in T}. $$
We denote by $T^+$ (resp. $T^0$, resp. $T^-$) the subset of $T$ such that $A_{T^+} = A_T^{+}$ (resp. $A_{T^0} = A_T^{0}$, resp. $A_{T^-} = A_T^{-}$).

\begin{theorem}[Theorems 4 and 7 of Vinberg \cite{MR0302779}]\label{combi_structure}
Let $P$ be a Coxeter $d$-polytope with Coxeter group $W$ and let $T$ be a proper subset of $S$.
\begin{enumerate}
\item If $W_T$ is spherical, then $T \in \F(P)$ and $\bigcap_{t \in T} t$ is a face of dimension $d-|T|$.
\item If $A_T = A_T^0$ and $A_{Z(T)}^0 = \emptyset$, then $T \in \F(P)$.
\item If $W_T$ is virtually isomorphic to $\Z^{d-1}$, $A_T = A_T^0$ and $P$ is not parabolic, then $T \in \F(P)$ and $\bigcap_{t \in T} t$ is a vertex of $P$.
\item If $T^{0} \cup T^{-} \in \F(P)$, then $T \in \F(P)$.
\end{enumerate}\end{theorem}

In order to prove Proposition \ref{prop:condition_dehn_filling}, we will need the following lemma:

\begin{lemma}[Lemma 13 and Proposition 21 of Vinberg \cite{MR0302779}]\label{lem:technical1}
Let $W$ be an irreducible affine Coxeter group. If $A_0$ is the Gram matrix of $W$ and $A$ is a Cartan matrix realizing $W$, then the following are true:
\begin{enumerate}
\item $\mathrm{rank}(A) \, \leqslant \, \mathrm{rank}(A_0) +1$.
\item $\mathrm{rank}(A) \, \geqslant \, \mathrm{rank}(A_0)$, and 
the equality holds if and only if $A$ and $A_0$ are equivalent.
\item If $A$ is not of zero type, then $W=\tilde{A}_k$\footnote{See Appendix \ref{classi_diagram} for the notation of the spherical and affine Coxeter groups.} for some $k \geqslant 1$.
\end{enumerate}
\end{lemma}

\begin{theorem}[Theorem 6 of Vinberg \cite{MR0302779}]\label{thm:bilinear}
Let $P$ be a loxodromic Coxeter polytope. Then the group $\G_P$ preserves a bilinear form if and only if $A_P$ is equivalent to a symmetric Cartan matrix.
\end{theorem}

\subsection{Combinatorial $4$-polytopes with 6 facets}\label{polytope_few_facets}

By taking a cone with basis a polytope $Q$ of dimension $d-1$, we obtain a polytope $\Pyr(Q)$ of dimension $d$. Every non-trivial choice of apex of the cone gives the same combinatorial structure for $\Pyr(Q)$. Such a polytope $\Pyr(Q)$ is called the \emph{pyramid over $Q$}. We use the notation $\Delta_d$ to denote the simplex of dimension $d$. A polytope $P$ is \emph{simple} if every vertex link of $P$ is a simplex.

\medskip

There are $4$ different combinatorial polytopes of dimension $4$ with $6$ facets (see e.g. Paragraph 6.5 of Ziegler \cite{MR1311028}). Two of them are simple: $\Delta_1 \times \Delta_3$ and $\Delta_2 \times \Delta_2$, and two others are pyramids: $\Pyr(\Delta_1 \times \Delta_2)$ and $\Pyr(\Pyr(\Delta_1 \times \Delta_1))$.

\medskip

We emphasize the following distinction between these four polytopes.

\begin{itemize}
\item In $\Delta_1 \times \Delta_3$: \textit{there are two disjoint facets}.
\item In $\Delta_2 \times \Delta_2$: \textit{any two facets are adjacent}.
\item In $\Pyr(\Delta_1 \times \Delta_2)$ and $\Pyr(\Pyr(\Delta_1 \times \Delta_1))$: \textit{any two facets intersect}.
\item In $\Pyr(\Delta_1 \times \Delta_2)$: \textit{there is exactly one pair of non-adjacent facets}.
\item In  $\Pyr(\Pyr(\Delta_1 \times \Delta_1))$: \textit{there are exactly two pairs of non-adjacent facets}.
\end{itemize}

A \emph{triangular prism} (or simply \emph{prism}) is the combinatorial polytope $\Delta_1 \times \Delta_2$, and hence $\Pyr(\Delta_1 \times \Delta_2)$ is the $4$-pyramid over the prism. More generally, $\Delta_1 \times \Delta_{d-1}$ is a {\em $d$-prism} (or a \emph{simplicial prism}).

\section{Dehn fillings of Coxeter polytopes}\label{sec:defi_dehn_filling}

Let $P_{\infty}$ be a Coxeter $d$-polytope and let $\V = \{v_1, \dotsc, v_k\}$ be a non-empty subset of the set $\Vp$ of all parabolic vertices of $P_{\infty}$.

\begin{de}\label{def:dehn_filling}
 A \emph{projective generalized Dehn filling} (or simply \emph{Dehn filling}) of $P_{\infty}$ at $\V$ is a Coxeter $d$-polytope $P_{m_{1}, \dotsc, m_{k}}$ with $m_i \in \mathbb{N} \setminus \{0, 1\}$, $i=1, \dotsc, k$, such that the underlying labeled polytope of $P_{\infty}$ is obtained from the underlying labeled polytope of $P_{m_{1}, \dotsc, m_{k}}$ by collapsing a ridge $r_i$ of label $\frac{\pi}{m_i}$ to the vertex $v_i$ for each $i = 1, \dotsc, k$ (see Figure \ref{fig:dehnfilling} for Dehn filling in dimension 3).
\end{de}

\begin{rem}
Since every vertex $v \in \V$ is parabolic, if $P_{\infty}$ is irreducible and loxodromic, then we know from Marquis \cite{Marquis:2014aa} that $v$ is \emph{cusped}, \ie there exists a neighborhood $\U$ of $v$ such that
$$\mu_{\O_{P_{\infty}}}(P_{\infty} \cap \U ) < \infty.$$
\end{rem}

In Section \ref{general_remark}, we show that the existence of Dehn fillings \emph{implies some condition on the cusp}: Let $P_\infty$ be an irreducible Coxeter $d$-polytope with Coxeter group $W_{\infty}$. If there are Dehn fillings of $P_{\infty}$ at $\V \subset \Vp$, then for each $v \in \V$, the Coxeter group $W_{\infty,v}$ of the link of $P_{\infty}$ at $v$ is equal to $\tilde{A}_1 \times \tilde{A}_{d-2}$.

\section{Labeled polytopes $\GG_m^i$}\label{section:labeled}

In Sections \ref{section:labeled}, \ref{sec:geometric_description}, \ref{sec:the_computation} and \ref{section: Limit}, for the sake of demonstrating the construction of the examples in Theorem \ref{MainThm1} without too many technical details, we concentrate our attention on three families of labeled $4$-polytopes $\GG_m^i$ with $i=1,2,3$ for $3 \leqslant m \leq \infty$. The arguments can easily be extended to cover the other families in Table \ref{table:ex1}.

\medskip

The examples we will present fall into three kinds, indexed by superscript $i$ in Table \ref{Cox_gp}. The labeled polytope $\GG^1_m$ is perfect, the deformation space $\B(\GG^1_m)$ with $m > 6$ consists of two points and the deformation space $\B(\GG^1_\infty)$ is a singleton. The labeled polytope $\GG^2_m$ is 2-perfect but not perfect, the deformation space $\B(\GG^2_m)$ with $m > 6$ consists of two points, and the deformation space  $\B(\GG^2_\infty)$ is a singleton. The labeled polytope $\GG^3_m$ is perfect and the deformation spaces $\B(\GG^3_m)$ with $m > 3$ and $\B(\GG^3_\infty)$ are homeomorphic to the disjoint union of two lines and to one line, respectively.

\medskip

Hence, the first case is the simplest way for us to build examples in Theorems \ref{MainThm1} and \ref{whynew}. The second case does not build directly examples in Theorems \ref{MainThm1} and \ref{whynew}: One needs to truncate some vertices first, then one can glue together (see Section \ref{sec:gluing}) the truncated polytopes in order to obtain infinitely many examples in Theorems \ref{MainThm1} and \ref{whynew}. The third case builds directly examples in Theorems \ref{MainThm1} and \ref{whynew}, but this time examples with deformation, so in particular it gives examples with $\O_\infty$ not an ellipsoid. See Section \ref{sec:geometric_description} for more details on the geometric properties of all examples.

\medskip

If $m$ is finite, then the underlying polytope of $\GG_m^i$ is the product of two triangles, and if $m=\infty$, then the underlying polytope of $\GG_{\infty}^i$ is the pyramid over the prism (see Figure \ref{first_schlegel} for the Schlegel diagrams of these polytopes, and refer to Chapter 5 of Ziegler \cite{MR1311028} for more details about Schlegel diagram). We explain below how to give labels on the ridges of $\GG_m^i$ using the Coxeter groups $W^i_m$ in Table \ref{Cox_gp}.

\begin{figure}[ht]
\centering
\begin{tabular}{cc}
\begin{tikzpicture}[line cap=round,line join=round,>=triangle 45,x=1.0cm,y=1.0cm]
\clip(-1.4,-0.33) rectangle (1.4,5.04);
\draw [line width=1.6pt] (-1,4)-- (1,4);
\draw [line width=1.6pt] (1,4)-- (0.2,4.8);
\draw [line width=1.6pt] (0.2,4.8)-- (-1,4);
\draw [line width=1.6pt] (-0.5,2)-- (0.5,2);
\draw [line width=1.6pt] (0.5,2)-- (-0.04,2.5);
\draw [line width=1.6pt] (-0.04,2.5)-- (-0.5,2);
\draw [line width=1.6pt] (-1,0)-- (1,0);
\draw [line width=1.6pt] (1,0)-- (0.2,0.8);
\draw [line width=1.6pt] (0.2,0.8)-- (-1,0);
\draw (-1,0)-- (-1,4);
\draw (0.2,4.8)-- (0.2,0.8);
\draw (1,0)-- (1,4);
\draw [dash pattern=on 7pt off 1pt] (-1,4)-- (-0.5,2);
\draw [dash pattern=on 7pt off 1pt] (-0.5,2)-- (-1,0);
\draw [dash pattern=on 7pt off 1pt] (0.5,2)-- (1,0);
\draw [dash pattern=on 7pt off 1pt] (0.5,2)-- (1,4);
\draw [dash pattern=on 7pt off 1pt] (0.2,4.8)-- (-0.04,2.5);
\draw [dash pattern=on 7pt off 1pt] (-0.04,2.5)-- (0.2,0.8);
\begin{scriptsize}
\fill [color=black] (-1,4) circle (3.0pt);
\fill [color=black] (1,4) circle (3.0pt);
\fill [color=black] (-1,0) circle (3.0pt);
\fill [color=black] (1,0) circle (3.0pt);
\fill [color=black] (-0.5,2) circle (3.0pt);
\fill [color=black] (0.5,2) circle (3.0pt);
\fill [color=black] (0.2,0.8) circle (3.0pt);
\fill [color=black] (0.2,4.8) circle (3.0pt);
\fill [color=black] (-0.04,2.5) circle (3.0pt);
\end{scriptsize}
\end{tikzpicture}
&
\begin{tikzpicture}[line cap=round,line join=round,>=triangle 45,x=1.0cm,y=1.0cm]
\clip(-1.43,-0.29) rectangle (1.32,5.01);
\draw [line width=1.6pt] (-1,4)-- (1,4);
\draw [line width=1.6pt] (1,4)-- (0.2,4.8);
\draw [line width=1.6pt] (0.2,4.8)-- (-1,4);
\draw [line width=1.6pt] (-1,0)-- (1,0);
\draw [line width=1.6pt] (1,0)-- (0.2,0.8);
\draw [line width=1.6pt] (0.2,0.8)-- (-1,0);
\draw (-1,0)-- (-1,4);
\draw (0.2,4.8)-- (0.2,0.8);
\draw (1,0)-- (1,4);
\draw [dash pattern=on 7pt off 1pt] (-1,4)-- (0,2);
\draw [dash pattern=on 7pt off 1pt] (0,2)-- (-1,0);
\draw [dash pattern=on 7pt off 1pt] (0.2,4.8)-- (0,2);
\draw [dash pattern=on 7pt off 1pt] (1,4)-- (0,2);
\draw [dash pattern=on 7pt off 1pt] (0,2)-- (0.2,0.8);
\draw [dash pattern=on 7pt off 1pt] (0,2)-- (1,0);
\begin{scriptsize}
\fill [color=black] (-1,4) circle (3.0pt);
\fill [color=black] (1,4) circle (3.0pt);
\fill [color=black] (-1,0) circle (3.0pt);
\fill [color=black] (1,0) circle (3.0pt);
\fill [color=black] (0,2) circle (3.0pt);
\fill [color=black] (0.2,0.8) circle (3.0pt);
\fill [color=black] (0.2,4.8) circle (3.0pt);
\end{scriptsize}
\end{tikzpicture}
\end{tabular}
\caption{The Schlegel diagrams of the product of two triangles and of a pyramid over a prism.}\label{first_schlegel}
\end{figure}
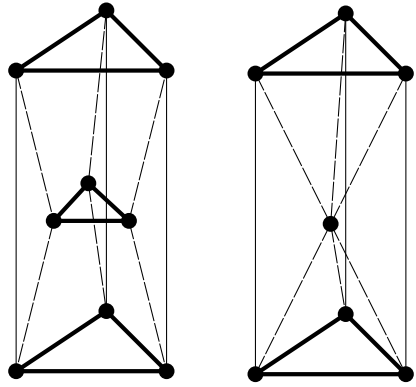


\begin{table}[ht]

\begin{center}
\begin{tabular}{ccccc}
\begin{tikzpicture}[thick,scale=0.6, every node/.style={transform shape}]
\node[draw,circle] (3) at (0,0) {3};
\node[draw,circle] (1) at (-1.5,0.866) {1};
\node[draw,circle] (2) at (-1.5,-0.866) {2};

\node[draw,circle] (4) at (1.732,0) {4};
\node[draw,circle] (5) at (1.732+1.5,0.866) {5};
\node[draw,circle] (6) at (1.732+1.5,-0.866) {6};

\draw (1) -- (2)  node[midway,left] {};
\draw (2) -- (3)  node[above,midway] {};
\draw (3)--(1) node[above,midway] {};
\draw (4) -- (5) node[above,midway] {};
\draw (5) -- (6) node[right,midway] {$m$};
\draw (3) -- (4) node[above,midway] {};
\end{tikzpicture}
&
\begin{tikzpicture}[thick,scale=0.6, every node/.style={transform shape}]
\node[draw,circle] (3) at (0,0) {3};
\node[draw,circle] (1) at (-1.5,0.866) {1};
\node[draw,circle] (2) at (-1.5,-0.866) {2};

\node[draw,circle] (4) at (1.732,0) {4};
\node[draw,circle] (5) at (1.732+1.5,0.866) {5};
\node[draw,circle] (6) at (1.732+1.5,-0.866) {6};

\draw (1) -- (2)  node[midway,left] {};
\draw (2) -- (3)  node[above,midway] {};
\draw (3)--(1) node[above,midway] {};
\draw (4) -- (5) node[above,midway] {};
\draw (5) -- (6) node[right,midway] {$m$};
\draw (3) -- (4) node[above,midway] {$5$};
\end{tikzpicture}
&
\begin{tikzpicture}[thick,scale=0.6, every node/.style={transform shape}]
\node[draw,circle] (3) at (0,0) {3};
\node[draw,circle] (1) at (-1.5,0.866) {1};
\node[draw,circle] (2) at (-1.5,-0.866) {2};

\node[draw,circle] (4) at (1.732,0) {4};
\node[draw,circle] (5) at (1.732+1.5,0.866) {5};
\node[draw,circle] (6) at (1.732+1.5,-0.866) {6};

\draw (1) -- (2)  node[midway,left] {};
\draw (2) -- (3)  node[above,midway] {};
\draw (3)--(1) node[above,midway] {};
\draw (4) -- (5) node[above,midway] {};
\draw (5) -- (6) node[right,midway] {$m$};
\draw (3) -- (4) node[above,midway] {};
\draw (6) -- (4) node[above,midway] {};
\end{tikzpicture}
\\
\subfloat[$W^1_m$.]{\hspace{.333\linewidth}}
&
\subfloat[$W^2_m$.]{\hspace{.333\linewidth}}
&
\subfloat[$W^3_m$.]{\hspace{.333\linewidth}}
\\
  \end{tabular}
\caption{The Coxeter graphs of $W^1_m \, ,\, W^2_m \, ,\, W^3_m$\label{Cox_gp}}

\end{center}
\end{table}


\subsection{The labeled polytopes $\GG_m^i$ for $m$ finite}

The underlying polytope of $\GG_m^i$ is the Cartesian product $\GG$ of two triangles. It is a $4$-dimensional polytope with 6 facets, 15 ridges and 9 vertices.

\medskip

We describe more precisely the combinatorial structure of $\GG$: Let $T_1$, $T_2$ be two triangles and let $\GG = T_1 \times T_2$.  Denote the edges of $T_1$ by $e_1,e_2,e_3$ and the edges of $T_2$ by $e_4,e_5,e_6$. The 6 facets of $\GG$ are $F_k=e_k \times T_2$, $k=1,2,3$, and $F_l=T_1 \times e_l$, $l=4,5,6$, and the 15 ridges of $\GG$ can be divided into 2 families:
\begin{itemize}
\item 9 ridges (quadrilaterals): $e_k \times e_l$ for every $k=1,2,3$ and $l=4,5,6$.
\item 6 ridges (triangles): $T_1\times w$ and $v \times T_2$ for every vertex $w$ (resp. $v$) of $T_2$ (resp. $T_1$).
\end{itemize}

The ridge $F_k \cap F_l$ is $(e_k \cap e_l) \times T_2$ if $(k,l) \in \{ (1,2), (2,3), (3,1) \}$, $T_1 \times (e_k \cap e_l)$ if $(k,l) \in \{ (4,5), (5,6), (6,4) \}$, and $e_k \times e_l$ if $k=1,2,3$ and $l=4,5,6$. Finally, the 9 vertices of $\GG$ are $v_{klst} = F_k \cap F_l \cap F_s \cap F_t$ for $(k,l) \in \{ (1,2), (2,3), (3,1) \}$ and $(s,t) \in \{ (4,5), (5,6), (6,4) \}$.

\medskip

Note that every two distinct facets of $\GG$ are adjacent, and hence we can label the ridges of $\GG$ according to the Coxeter graph $W^i_m$ in Table \ref{Cox_gp} to obtain the labeled polytope $\GG^i_m$.

\begin{propo}
Let $\GG^i_m$ be the labeled polytope with $m$ finite. Then:
\begin{itemize}
\item The labeled polytopes $\GG^1_m$ and $\GG^3_m$ are perfect.

\item The labeled polytope $\GG^2_m$ is $2$-perfect with exactly two Lann\'er vertices $v_{1345}$ and $v_{2345}$.
\end{itemize}
\end{propo}

\begin{proof}
First, remark that $\GG^i_m$ is simple, and so every vertex link is a $3$-simplex. Hence, by Remark \ref{rem:perfect_labeled}, the proof immediately follows from comparing the Coxeter diagrams of the links of $\GG^i_m$ and the Coxeter diagrams in Appendix \ref{classi_diagram}. For example, on the one hand, the Coxeter group of the link of $\GG^1_m$ (or $\GG^3_m$) at the vertex $v_{1345}$ is:
\begin{center}
\begin{tikzpicture}[thick,scale=0.6, every node/.style={transform shape}]
\node[draw,circle] (3) at (0,0) {3};
\node[draw,circle] (1) at (-1.5,0.866) {1};
\node[draw,circle] (4) at (1.732,0) {4};
\node[draw,circle] (5) at (1.732+1.5,0.866) {5};

\draw (3)--(1) node[above,midway] {};
\draw (4) -- (5) node[above,midway] {};
\draw (3) -- (4) node[above,midway] {};
\end{tikzpicture}
\end{center}
which is the spherical Coxeter diagram $A_4$ in Table \ref{spheri_diag}, and on the other hand, the Coxeter group of the link of $\GG^2_m$ at the vertex $v_{1345}$ is:
\begin{center}
\begin{tikzpicture}[thick,scale=0.6, every node/.style={transform shape}]
\node[draw,circle] (3) at (0,0) {3};
\node[draw,circle] (1) at (-1.5,0.866) {1};
\node[draw,circle] (4) at (1.732,0) {4};
\node[draw,circle] (5) at (1.732+1.5,0.866) {5};

\draw (3)--(1) node[above,midway] {};
\draw (4) -- (5) node[above,midway] {};
\draw (3) -- (4) node[above,midway] {5};
\end{tikzpicture}
\end{center}
which is a Lannér Coxeter diagram in Table \ref{table:Lanner_dim3}.
\end{proof}

\subsection{The labeled polytopes $\GG_m^i$ for $m$ infinite}

The underlying polytope of the labeled polytope $\GG_{\infty}^i$, $i=1, 2, 3$, is the pyramid $\GG_{\infty}$ over the prism. It is a $4$-dimensional polytope with 6 facets, 14 ridges and 7 vertices.

\begin{figure}[ht]
\centering
\begin{tikzpicture}[line cap=round,line join=round,>=triangle 45,x=1.0cm,y=1.0cm]
\clip(0,-0.7) rectangle (9.18,3.44);
\draw [line width=1.6pt] (2,2)-- (7,2);
\draw [line width=1.6pt] (7,2)-- (7,0);
\draw [line width=1.6pt] (7,0)-- (2,0);
\draw [line width=1.6pt] (2,0)-- (2,2);
\draw [line width=1.6pt] (2,2)-- (3.42,1);
\draw [line width=1.6pt] (3.42,1)-- (2,0);
\draw [line width=1.6pt] (2,0)-- (2,2);
\draw [line width=1.6pt] (7,2)-- (5.4,1);
\draw [line width=1.6pt] (5.4,1)-- (7,0);
\draw [line width=1.6pt] (7,0)-- (7,2);
\draw [line width=1.6pt] (3.42,1)-- (5.4,1);
\draw (4.5,1.6) node[circle,draw,inner sep=1pt,outer sep=1pt] {$1$};
\draw (4.5,0.5) node[circle,draw,inner sep=1pt,outer sep=1pt] {$2$};
\draw (4.5,-0.4) node[circle,draw,inner sep=1pt,outer sep=1pt] {$3$};
\draw (6.45,1) node[circle,draw,inner sep=1pt,outer sep=1pt] {$6$};
\draw (2.43,1) node[circle,draw,inner sep=1pt,outer sep=1pt] {$5$};
\end{tikzpicture}
\caption{The basis of $\GG_{\infty}$}
\label{prism}
\end{figure}

A more precise description of the combinatorial structure of $\GG_{\infty}$ is as follows: Let $I$ be a segment with two endpoints $a_5$ and $a_6$, and $\Delta$ be a triangle with three edges $b_1$, $b_2$ and $b_3$. The prism $I \times \Delta$ has 5 faces labeled in the obvious way (see Figure \ref{prism}). Let $\GG_{\infty}=\textrm{Pyr}(I \times \Delta)$. Among the 6 facets of $\GG_{\infty}$, the five of them, $F_i$, $i=1,2,3,5,6$, correspond to the faces of $I \times \Delta$, and $F_4$ denotes the remaining one. The 14 ridges of $\GG_{\infty}$ can be divided into 2 families:

\begin{itemize}
\item 5 ridges correspond to the faces of $I \times \Delta$ and these ridges are in the facet $F_4$.
\item 9 ridges correspond to the edges of $I \times \Delta$ and these ridges are $F_1 \cap F_2$, $F_2 \cap F_3$, $F_3 \cap F_1$ and $F_k \cap F_l$ with $k=1,2,3$ and $l=5,6$.
\end{itemize}

We remark that every pair of facets $F_k$ and $F_l$ are adjacent except for $\{k,\, l \} = \{ 5 ,\, 6 \}$ and $F_5 \cap F_6$ is the apex $s$ of the pyramid $\GG_{\infty}$, and therefore we can label the ridges of $\GG_{\infty}$ according to the Coxeter graph $W^i_{\infty}$ in Table \ref{Cox_gp} to obtain $\GG^i_{\infty}$. 

\begin{propo}Let $\GG^i_{\infty}$ be the labeled polytope $\GG^i_m$ with $m$ infinite. Then:
\begin{itemize}
\item For every $i =1,2, 3$, the labeled polytope $\GG^i_{\infty}$ is 2-perfect.

\item For each $i=1,3$, the only non-spherical vertex of $\GG^i_{\infty}$ is the apex $s$, which is affine of type $\tilde{A}_1 \times \tilde{A}_2$.

\item The labeled polytope $\GG^2_{\infty}$ has 3 non-spherical vertices $v_{1345}$, $v_{2345}$ and the apex $s$. Two of them, $v_{1345}$ and $v_{2345}$, are Lann\'er, and $s$ is affine of type $\tilde{A}_1 \times \tilde{A}_2$.
\end{itemize}
\end{propo}

\begin{proof}
Remark that except the apex $s$ of $\GG_{\infty}^i$, every vertex link is a $3$-simplex, and the link of $s$ is a prism. Once again, by Remark \ref{rem:perfect_labeled}, the proof follows from comparing the Coxeter diagrams of the links of $\GG^i_\infty$ and the Coxeter diagrams in Appendix \ref{classi_diagram}. For the apex $s$, the Coxeter group of the link of $\GG_{\infty}^i$ at $s = F_1 \cap F_2 \cap F_3 \cap F_5 \cap F_6$ is:
\begin{center}
\begin{tikzpicture}[thick,scale=0.6, every node/.style={transform shape}]
\node[draw,circle] (3) at (0,0) {3};
\node[draw,circle] (1) at (-1.5,0.866) {1};
\node[draw,circle] (2) at (-1.5,-0.866) {2};

\node[draw,circle] (5) at (1.732+1.5,0.866) {5};
\node[draw,circle] (6) at (1.732+1.5,-0.866) {6};

\draw (1) -- (2)  node[midway,left] {};
\draw (2) -- (3)  node[above,midway] {};
\draw (3)--(1) node[above,midway] {};
\draw (5) -- (6) node[right,midway] {$\infty$};
\end{tikzpicture}
\end{center}
which is the affine Coxeter group $\tilde{A}_2 \times \tilde{A}_1$ (see Table \ref{affi_diag}) virtually isomorphic to $\Z^2 \times \Z^1 = \Z^3$, and hence $s$ is perfect. 
\end{proof}

\begin{rem}\label{rem:dehn_fill_labeled}
We stress that the labeled polytope $\GG_m^i$ is a Dehn filling of $\GG_\infty^i$ at its apex at the level of labeled polytope. The main point to prove Theorem \ref{MainThm1} is now to show that the deformation spaces of $\GG_m^i$ and $\GG_\infty^i$ are not empty. This is the content of Theorems \ref{thm:general_4.1} and \ref{thm:general_4.2}. See Theorems \ref{thm:high_dim_m_finite} and \ref{thm:high_dim_m_infinite} for the higher dimensional cases.
\end{rem}

\section{Prime examples in dimension $4$}\label{sec:more_prime}

In Table \ref{table:ex1}(A), we give a list of Coxeter groups $W_m$, which is not redundant but contains the previous three examples of Table \ref{Cox_gp}. For each of them, we let $m$ be an integer bigger than $6$, and we label the product of two triangles in order to obtain a labeled polytope $\GG_{m}$. The technique we will describe easily extends to show:

\begin{theorem}\label{thm:general_4.1}
Choose one of the 13 families $(W_m)_{m > 6}$ in Table \ref{table:ex1}(A).
If the Coxeter graphs of the family have one loop, then $\B({\GG_m})$ consists of two points.
If the Coxeter graphs of the family have two loops, then $\B({\GG_m})$ consists of two disjoint lines.
\end{theorem}

If $m=\infty$, then we can label a pyramid over a prism in order to obtain a labeled polytope $\GG_{\infty}$. Eventually, we
can show:

\begin{theorem}\label{thm:general_4.2}
Choose one of the 13 families $(W_m)_{m > 6}$ in Table \ref{table:ex1}(A).
\begin{itemize}
\item Assume that the Coxeter graphs of the family have one loop. If $P_m$ is one of the two points of $\B({\GG_m})$, then the sequence $(P_m)_m$ converges to the unique hyperbolic polytope $P_{\infty}$ whose Coxeter graph is $W_{\infty}$.
\item Assume that the Coxeter graphs of the family have two loops. If $\mu > 0$ and $P_m$ is one of the two points of $\B({\GG_m})$ with $\mu(P_m) = \mu$, then $(P_m)_m$ converges to $P_{\infty}$ which is the unique Coxeter $4$-polytope such that $\mu(P_{\infty}) = \mu$ and $W_{P_\infty} = W_{\infty}$.
\end{itemize}
Moreover, the underlying polytope of $P_{\infty}$ is the pyramid over the prism.
\end{theorem}

The definition of $\mu(P_\infty)$ will be given in Section \ref{subsec:prepa_compu}. Theorems \ref{thm:general_4.1} and \ref{thm:general_4.2} will be proved in Sections \ref{sec:the_computation} and \ref{section: Limit} in the case of the examples of Table \ref{Cox_gp}. The general case is analogous.

\section{Geometric description}\label{sec:geometric_description}

In this section, we assume Theorems \ref{thm:general_4.1} and \ref{thm:general_4.2} and we explain the geometric properties of the Coxeter polytopes $P_m$ and of the action of the projective reflection groups $\Gamma_{P_m}$ of $P_m$ on $\Omega_{P_m}$. To do so, we need the following theorem:

\begin{theorem}[Theorem E of Marquis \cite{Marquis:2014aa} using Benoist \cite{divI} and Cooper, Long and Tillmann \cite{clt_cusps}]\label{Benoist}
Let $P$ be a quasi-perfect loxodromic Coxeter polytope. Then the following are equivalent:
\begin{itemize}
\item The group $W_P$ is Gromov-hyperbolic relative to the standard  subgroups of $W_P$ corresponding to the parabolic vertices of $P$.
\item The convex domain $\O_P$ is strictly convex.
\item The boundary $\dO_P$ of $\O_P$ is of class $\C^1$.
\end{itemize}
\end{theorem}

\subsection{Description of the geometry of the first example}\label{ss:first}

Assume $i=1$. The Coxeter polytope $P^1_{\infty}$ is a quasi-perfect hyperbolic polytope of dimension $4$, listed in Tumarkin \cite{MR2086616}. In particular, $P^1_{\infty}$ is a hyperbolic Coxeter polytope of finite volume, \ie the projective Coxeter group of $P^1_{\infty}$ is a non-uniform lattice of $\mathrm{O}^{+}_{d,1}(\R)$.

The Coxeter $4$-polytope $P^1_m$ is perfect. The group $\G^1_m := \G_{P^1_m}$ is not Gromov-hyperbolic since the group generated by the reflections corresponding to $\{1,2,3\}$ is an affine Coxeter group virtually isomorphic to $\mathbb{Z}^2$, and so by Theorem \ref{Benoist}, the convex domain $\O^1_m$ is not strictly convex nor with $\C^1$ boundary.


\subsection{Description of the geometry of the second example}

Assume $i=2$. The Coxeter polytope $P^2_{\infty}$ is a 2-perfect hyperbolic polytope of dimension $4$ with four spherical vertices, one parabolic vertex and two loxodromic vertices, and the Coxeter polytope $P^2_m$ is a 2-perfect Coxeter polytope of dimension $4$ with seven spherical vertices and two loxodromic vertices (see Figure \ref{DiagramsExamples}(B)). We will come back on these examples in Section \ref{sec:gluing}.

\subsection{Description of the geometry of the third example}

Assume $i=3$. The Coxeter polytope $P^3_{\infty}$ is a quasi-perfect polytope of dimension $4$ with six spherical vertices and one parabolic vertex (see Figure \ref{DiagramsExamples}(A)). Moreover, $\G^3_{\infty}$ is a non-uniform lattice of $\mathrm{O}^{+}_{d,1}(\R)$ if and only if $\mu(P^3_{\infty})=1$, by Theorem \ref{thm:bilinear} and Corollary \ref{cor:m_infini}. In other words, if $\mu(P^3_{\infty}) \neq 1$, then $P^3_{\infty}$ is not a hyperbolic Coxeter polytope and the convex domain $\O^3_{\infty} := \O_{P^3_{\infty}}$ is not an ellipsoid, however by Theorem \ref{Benoist}, it is strictly convex with $\C^1$ boundary. 

The Coxeter polytope $P^3_m$ is a perfect Coxeter polytope of dimension $4$. Once again, by Theorem \ref{Benoist}, the convex domain $\O^3_m$ is not strictly convex nor with $\C^1$ boundary. Note that in contrast to Section \ref{ss:first}, in the case when $i=3$, there is a one-parameter family of deformations of the convex domains.

\begin{figure}[ht]
\centering
\begin{tabular}{cc}


\begin{tabular}{cc}
\begin{tikzpicture}[line cap=round,line join=round,>=triangle 45,x=1.0cm,y=1.0cm]
\clip(-1.4,-0.33) rectangle (1.4,5.04);
\draw [line width=1.6pt] (-1,4)-- (1,4);
\draw [line width=1.6pt] (1,4)-- (0.2,4.8);
\draw [line width=1.6pt] (0.2,4.8)-- (-1,4);
\draw [line width=1.6pt] (-0.5,2)-- (0.5,2);
\draw [line width=1.6pt] (0.5,2)-- (-0.04,2.5);
\draw [line width=1.6pt] (-0.04,2.5)-- (-0.5,2);
\draw [line width=1.6pt] (-1,0)-- (1,0);
\draw [line width=1.6pt] (1,0)-- (0.2,0.8);
\draw [line width=1.6pt] (0.2,0.8)-- (-1,0);
\draw (-1,0)-- (-1,4);
\draw (0.2,4.8)-- (0.2,0.8);
\draw (1,0)-- (1,4);
\draw [dash pattern=on 7pt off 1pt] (-1,4)-- (-0.5,2);
\draw [dash pattern=on 7pt off 1pt] (-0.5,2)-- (-1,0);
\draw [dash pattern=on 7pt off 1pt] (0.5,2)-- (1,0);
\draw [dash pattern=on 7pt off 1pt] (0.5,2)-- (1,4);
\draw [dash pattern=on 7pt off 1pt] (0.2,4.8)-- (-0.04,2.5);
\draw [dash pattern=on 7pt off 1pt] (-0.04,2.5)-- (0.2,0.8);
\begin{scriptsize}
\fill [color=black] (-1,4) circle (3.0pt);
\fill [color=black] (1,4) circle (3.0pt);
\fill [color=black] (-1,0) circle (3.0pt);
\fill [color=black] (1,0) circle (3.0pt);
\fill [color=black] (-0.5,2) circle (3.0pt);
\fill [color=black] (0.5,2) circle (3.0pt);
\fill [color=black] (0.2,0.8) circle (3.0pt);
\fill [color=black] (0.2,4.8) circle (3.0pt);
\fill [color=black] (-0.04,2.5) circle (3.0pt);
\end{scriptsize}
\end{tikzpicture}
&
\begin{tikzpicture}[line cap=round,line join=round,>=triangle 45,x=1.0cm,y=1.0cm]
\clip(-1.43,-0.29) rectangle (1.32,5.01);
\draw [line width=1.6pt] (-1,4)-- (1,4);
\draw [line width=1.6pt] (1,4)-- (0.2,4.8);
\draw [line width=1.6pt] (0.2,4.8)-- (-1,4);
\draw [line width=1.6pt] (-1,0)-- (1,0);
\draw [line width=1.6pt] (1,0)-- (0.2,0.8);
\draw [line width=1.6pt] (0.2,0.8)-- (-1,0);
\draw (-1,0)-- (-1,4);
\draw (0.2,4.8)-- (0.2,0.8);
\draw (1,0)-- (1,4);
\draw [dash pattern=on 7pt off 1pt] (-1,4)-- (0,2);
\draw [dash pattern=on 7pt off 1pt] (0,2)-- (-1,0);
\draw [dash pattern=on 7pt off 1pt] (0.2,4.8)-- (0,2);
\draw [dash pattern=on 7pt off 1pt] (1,4)-- (0,2);
\draw [dash pattern=on 7pt off 1pt] (0,2)-- (0.2,0.8);
\draw [dash pattern=on 7pt off 1pt] (0,2)-- (1,0);
\begin{scriptsize}
\fill [color=black] (0.2,4.8) circle (3.0pt);
\fill [color=black] (-1,4) circle (3.0pt);
\fill [color=black] (1,4) circle (3.0pt);

\draw [black,fill=white] (0,2) circle (3.0pt);

\fill [color=black] (0.2,0.8) circle (3.0pt);
\fill [color=black] (1,0) circle (3.0pt);
\fill [color=black] (-1,0) circle (3.0pt);
\end{scriptsize}
\end{tikzpicture}
\end{tabular}
&



\begin{tabular}{cc}
\begin{tikzpicture}[line cap=round,line join=round,>=triangle 45,x=1.0cm,y=1.0cm]
\clip(-1.4,-0.33) rectangle (1.4,5.04);
\draw [line width=1.6pt] (-1,4)-- (1,4);
\draw [line width=1.6pt] (1,4)-- (0.2,4.8);
\draw [line width=1.6pt] (0.2,4.8)-- (-1,4);
\draw [line width=1.6pt] (-0.5,2)-- (0.5,2);
\draw [line width=1.6pt] (0.5,2)-- (-0.04,2.5);
\draw [line width=1.6pt] (-0.04,2.5)-- (-0.5,2);
\draw [line width=1.6pt] (-1,0)-- (1,0);
\draw [line width=1.6pt] (1,0)-- (0.2,0.8);
\draw [line width=1.6pt] (0.2,0.8)-- (-1,0);
\draw (-1,0)-- (-1,4);
\draw (0.2,4.8)-- (0.2,0.8);
\draw (1,0)-- (1,4);
\draw [dash pattern=on 7pt off 1pt] (-1,4)-- (-0.5,2);
\draw [dash pattern=on 7pt off 1pt] (-0.5,2)-- (-1,0);
\draw [dash pattern=on 7pt off 1pt] (0.5,2)-- (1,0);
\draw [dash pattern=on 7pt off 1pt] (0.5,2)-- (1,4);
\draw [dash pattern=on 7pt off 1pt] (0.2,4.8)-- (-0.04,2.5);
\draw [dash pattern=on 7pt off 1pt] (-0.04,2.5)-- (0.2,0.8);
\begin{scriptsize}
\fill [color=black] (0.2,4.8) circle (3.0pt);
\draw [black,fill=gray!40] (-1,4) circle (3.0pt);
\draw [black,fill=gray!40] (1,4) circle (3.0pt);

\fill [color=black] (-0.5,2) circle (3.0pt);
\fill [color=black] (0.5,2) circle (3.0pt);
\fill [color=black] (-0.04,2.5) circle (3.0pt);

\fill [color=black] (-1,0) circle (3.0pt);
\fill [color=black] (1,0) circle (3.0pt);
\fill [color=black] (0.2,0.8) circle (3.0pt);
\end{scriptsize}
\end{tikzpicture}
&
\begin{tikzpicture}[line cap=round,line join=round,>=triangle 45,x=1.0cm,y=1.0cm]
\clip(-1.43,-0.29) rectangle (1.32,5.01);
\draw [line width=1.6pt] (-1,4)-- (1,4);
\draw [line width=1.6pt] (1,4)-- (0.2,4.8);
\draw [line width=1.6pt] (0.2,4.8)-- (-1,4);
\draw [line width=1.6pt] (-1,0)-- (1,0);
\draw [line width=1.6pt] (1,0)-- (0.2,0.8);
\draw [line width=1.6pt] (0.2,0.8)-- (-1,0);
\draw (-1,0)-- (-1,4);
\draw (0.2,4.8)-- (0.2,0.8);
\draw (1,0)-- (1,4);
\draw [dash pattern=on 7pt off 1pt] (-1,4)-- (0,2);
\draw [dash pattern=on 7pt off 1pt] (0,2)-- (-1,0);
\draw [dash pattern=on 7pt off 1pt] (0.2,4.8)-- (0,2);
\draw [dash pattern=on 7pt off 1pt] (1,4)-- (0,2);
\draw [dash pattern=on 7pt off 1pt] (0,2)-- (0.2,0.8);
\draw [dash pattern=on 7pt off 1pt] (0,2)-- (1,0);
\begin{scriptsize}
\fill [color=black] (0.2,4.8) circle (3.0pt);
\draw [black,fill=gray!40] (-1,4) circle (3.0pt);
\draw [black,fill=gray!40] (1,4) circle (3.0pt);

\draw [black,fill=white] (0,2) circle (3.0pt);

\fill [color=black] (0.2,0.8) circle (3.0pt);
\fill [color=black] (1,0) circle (3.0pt);
\fill [color=black] (-1,0) circle (3.0pt);
\end{scriptsize}
\end{tikzpicture}
\end{tabular}
\\
\subfloat[For examples with $i=1$ or $i=3$]{\hspace{0.45\linewidth}}
&

\subfloat[For the example with $i=2$]{\hspace{0.45\linewidth}}
\end{tabular}
\caption{The spherical vertices are in black, the parabolic vertices are in white, and the loxodromic vertices are in gray.}
\label{DiagramsExamples}
\end{figure}

\section{Deformation spaces $\B(\GG^i_m)$}\label{sec:the_computation}

This section is devoted to the proof of the following proposition:

\begin{propo}\label{single_moduli_finite_case} Let $\B(\GG^i_{m})$ be the deformation space of the labeled polytope $\GG^i_m$ for $m$ finite. Then:

\begin{itemize}
\item The deformation space $\B(\GG^i_{m})$, $i=1,2$, consists of two points for every $6 < m < \infty$.

\item The deformation space $\B(\GG^3_{m})$  consists of two disjoint lines for every $3 < m < \infty$.
\end{itemize}
\end{propo}



%

\subsection{Preparation of the computation}\label{subsec:prepa_compu}

If $\iota$ is a pair of integers (resp. an integer), then let $c_{\iota} := \cos(\frac{\pi}{m_\iota})$ (resp. $:=\cos(\frac \pi \iota)$) and $s_{\iota} := \sin(\frac{\pi}{m_\iota})$ (resp. $:=\sin(\frac \pi \iota)$). We introduce the following matrix:

$$
A_{\lambda,\mu}=
\begin{pmatrix}
2             & -1 & -\lambda^{-1} \\
-1  & 2            & -1\\
-\lambda & -1 & 2 & -2c_{34} \\
& & -2c_{34} & 2 & -2c_{45} & -2\mu^{-1} c_{64} \\
& & & -2c_{45} & 2 & -2c_{56}\\
& & & -2\mu c_{64} & - 2c_{56} & 2\\
\end{pmatrix}
$$
with $\lambda,\mu >0$. If $(S = \{1,2,3,4,5,6\}, M^i_m = (m_{st})_{s,t \in S})$ is the Coxeter system associated to $W^i_m$, then the following hold:
\begin{itemize}
\item  If $(m_{34},m_{45},m_{56},m_{64})=(3,3,m,2)$, then $A_{\lambda}:=A_{\lambda,1}$ realizes $W^1_m$.
\item  If $(m_{34},m_{45},m_{56},m_{64})=(5,3,m,2)$, then $A_{\lambda}$ realizes $W^2_m$.
\item  If $(m_{34},m_{45},m_{56},m_{64})=(3,3,m,3)$, then $A_{\lambda,\mu}$ realizes $W^3_m$.
\end{itemize}

If $i=1,2$ (resp. $3$), then the equivalence class of Cartan matrices realizing $W^i_m$ is determined by exactly one (resp. two) cyclic product(s): $(123)$ (resp. $(123)$ and $(456)$). In other words, every Cartan matrix $A$ realizing $W^i_m$ is equivalent to a unique $A_{\lambda_,\mu}$, and so we call $A_{\lambda,\mu}$ the \emph{special form} of $A$.\footnote{More precisely, let $Q = \{ (1,2), (2,3), (3,4), (4,5), (5,6) \}$. If a Cartan matrix $A$ realize $W^i_m$, then there exists a unique positive diagonal matrix $D$ (up to scaling) so that $(DAD^{-1})_{st} = (DAD^{-1})_{ts}$ for every $(s,t) \in Q$ and so $A$ is equivalent to a special form $A_{\lambda,\mu}$. Of course, there are several choices for $Q$ and we make one \emph{special} choice.}

\begin{propo}\label{propo_real}
Suppose $3 \leqslant m < \infty$. Then $\B(\GG^i_{m})$ is homeomorphic to:  
\begin{itemize}
\item $\{ \lambda \in \mathbb{R} \mid \lambda > 0 \textrm{ and } \det(A_{\lambda})= 0\}$ for $i=1, 2$.
\item $\{ (\lambda,\mu) \in \mathbb{R}^2 \mid \lambda, \mu > 0 \textrm{ and } \det(A_{\lambda,\mu})= 0\}$ for $i=3$.
\end{itemize}
\end{propo}

\begin{proof}
We begin with the fact that the rank of $A_{\lambda,\mu}$ is always greater than or equal to $5$ by computing the $(6,6)$-minor\footnote{In fact, except the case $i=m=3$, the $(3,3)$-minor of $A_{\lambda,\mu}$ is always non-zero.} of $A_{\lambda,\mu}$:
$$ -24 c_{34}^2 - 4 s_{45}^2 (\sqrt{\lambda} - \frac{1}{\sqrt{\lambda}})^2 \, <\, 0$$ and therefore the image of the map
$$  \Lambda : \B(\GG_m^{i}) \rightarrow \bigcup_{e \leqslant 4} \Cm(W^i_m,e)$$ 
given by $P \mapsto [A_P]$ lies in $\Cm(W^i_m,4)$. Moreover, we know from Theorem \ref{theo_vin_reali} that $\Lambda$ is injective. Now let us prove that the image of $\B(\GG_m^{i})$ under $ \Lambda$ is exactly $\Cm(W^i_m,4)$: Once again, by Theorem \ref{theo_vin_reali}, every Cartan matrix $A$ in $\Cm(W^i_m,4)$ is realized by a Coxeter 4-polytope $P$. We need to check that the underlying combinatorial polytope of $P$ is the product of two triangles.

For every two distinct facets $s, t$ of $P$, the standard subgroup generated by $s, t$ is a finite dihedral group, which is a spherical Coxeter group, and hence by the first item of Theorem \ref{combi_structure}, the facets $s, t$ are adjacent. Finally, we know from Section \ref{polytope_few_facets} that the product of two triangles is the only $4$-polytope with 6 facets such that every two distint facets are adjacent, which completes the proof.
\end{proof}

\subsection{The computation}

For $\alpha,\beta,\gamma \in \R$, let:
$$
\Psi_{\alpha,\beta,\gamma}=1-\cos(\alpha)^2-\cos(\beta)^2-\cos(\gamma)^2 -2\cos(\alpha)\cos(\beta)\cos(\gamma)
$$

The following equality is remarkable (see e.g. (3.1) of Roeder, Hubbard and Dunbar \cite{MR2336832} or (1) of Section I.2 of Fenchel \cite{werner}):
\begin{equation*}
\Psi_{\alpha,\beta,\gamma}=
-4\cos\bigg(\frac{\alpha+\beta+\gamma}{2}\bigg)\cos\bigg(\frac{-\alpha+\beta+\gamma}{2}\bigg)\cos\bigg(\frac{\alpha-\beta+\gamma}{2}\bigg)\cos\bigg(\frac{\alpha+\beta-\gamma}{2}\bigg)
\end{equation*}

If $\alpha,\beta,\gamma \in (0,\frac{\pi}{2}]$, then $\Psi_{\alpha,\beta,\gamma} < 0$ (resp. $=0$, resp. $>0$) if and only if $\alpha +\beta + \gamma < \pi$ (resp. $=\pi$, resp. $> \pi$). We introduce the following two matrices:
\begin{center}
$
\begin{array}{ccc}
A^1_{\lambda}=
\begin{pmatrix}
1             & -c_{12} & -\lambda^{-1}c_{31} \\
-c_{12}  & 1            & -c_{23}\\
-\lambda c_{31} & -c_{23} & 1
\end{pmatrix}
& \textrm{and} &
A^2_{\mu}=
\begin{pmatrix}
1 & -c_{45} & -\mu^{-1}c_{64}\\
-c_{45} & 1 & -c_{56}\\
-\mu c_{64} & - c_{56} & 1\\
\end{pmatrix}
\end{array}
$
\end{center}
Implicitly, we assume that $m_{12}=m_{23}=m_{13}=m_{45}=3$, but we keep the notation of double indices in order to the use the remarkable equality. By a direct computation, we have:
$$
\frac{1}{64} \det(A_{\lambda,\mu})=  \det(A^1_{\lambda})\det(A^2_{\mu})-s_{12}^2 c_{34}^2 s_{56}^2
$$
Another direct computation gives us:
\begin{align*}
\det(A^1_{\lambda}) & =1-(c_{12}^2+c_{23}^2+c_{31}^2)-c_{12}c_{23}c_{31}(\lambda+\lambda^{-1}) = \Psi_{123} - c_{12}c_{23}c_{31} (\lambda+\lambda^{-1}-2)
\\
\det(A^2_{\mu})& =1-(c_{45}^2+c_{56}^2+c_{64}^2)-c_{45}c_{56}c_{64}(\mu+\mu^{-1}) = \Psi_{456} - c_{45}c_{56}c_{64} (\mu+\mu^{-1}-2) 
\end{align*}
where $\Psi_{123}$ (resp. $\Psi_{456}$) is $\Psi_{\alpha,\beta,\gamma}$ with $(\alpha,\beta,\gamma) = (\tfrac{\pi}{m_{12}},\tfrac{\pi}{m_{23}},\tfrac{\pi}{m_{31}})$ (resp. $(\tfrac{\pi}{m_{45}},\tfrac{\pi}{m_{56}},\tfrac{\pi}{m_{64}})$).

\begin{rem}
The quantity $\Psi_{123}$ is negative (resp. zero, resp. positive) if and only if the triangle $(\tfrac{\pi}{m_{12}},\tfrac{\pi}{m_{23}},\tfrac{\pi}{m_{31}})$ is hyperbolic (resp. affine, resp. spherical). Of course, the analogous statement is true for the quantity $\Psi_{456}$ and the triangle $(\tfrac{\pi}{m_{45}},\tfrac{\pi}{m_{56}},\tfrac{\pi}{m_{64}})$.
\end{rem}

\subsection{Proof of Proposition \ref{single_moduli_finite_case}}

\begin{proof}[Proof of Proposition \ref{single_moduli_finite_case}]
Assume that $i=1,2$. By Proposition \ref{propo_real}, we know that $\B(\GG^i_{m})$ is homeomorphic to:
$$\{ \lambda \in \mathbb{R} \mid \lambda > 0 \textrm{ and } \det(A_{\lambda})= 0\}$$
Moreover, since $c_{64}=0$ and the triangle $(\frac \pi 3, \frac \pi 3,\frac \pi 3)$ is affine, \ie $\Psi_{123}=0$, we have:
\begin{align}
\frac{1}{64}\det(A_{\lambda}) = \frac{1}{8} \Big( 2-(\lambda + \lambda^{-1})  \Big) \Psi_{456}-s_{12}^2 c_{34}^2 s_{56}^2
\label{eq:E1}
\end{align}
Observe that $m >6$ if and only if the triangle $(\frac \pi 3,\frac \pi m,\frac \pi 2)$ is hyperbolic, \ie $\Psi_{456} < 0$. Let $x=\lambda+\lambda^{-1}$. The function $[2, +\infty) \ni x \mapsto \det(A_{\lambda}) \in \R$ is an unbounded strictly increasing function.   If $x=2$, \ie $\lambda=1$, then $\det(A_{\lambda}) = - 64 s_{12}^2 c_{34}^2 s_{56}^2 <0$. Hence, there is a unique $x > 2$ and exactly two $\lambda \in \R_+^* \smallsetminus \{ 1 \}$ such that $\det(A_{\lambda})=0$.

\medskip

Now assume that $i=3$. Once again, by Proposition \ref{propo_real}, we need to solve the equation:
\begin{align}
\frac{1}{64} \det(A_{\lambda,\mu})= \frac{1}{8}(\lambda+\lambda^{-1}-2)(c_{45}c_{56}c_{64}(\mu+\mu^{-1}-2) - \Psi_{456}) - s_{12}^2 c_{34}^2 s_{56}^2 =0
\label{eq:E2}
\end{align}
on the set $ \{ (\lambda,\mu) \in \mathbb{R}^2 \mid  \lambda, \mu > 0\}$. Let $x=\lambda+\lambda^{-1}$ and $y=\mu+\mu^{-1}$. The equation is of the form:
$$
(x-2)(y-A)=B
$$
with $B >0$ and $A < 2$, since $m >3$ and so $\Psi_{456} < 0$. Thus, the solution space of the equation on $ \{ (x,y) \in \mathbb{R}^2 \mid x,y \geqslant 2 \} $ is homeomorphic to a closed half-line whose extremity $E$ is a point $(x_0,2)$ with $x_0 >2$ (see the blue curve in Figure \ref{Moduli_Lines}(A)). It follows that the set of solutions in $(\lambda,\mu)$ is homeomorphic to two disjoint lines (see the two blue curves in Figure \ref{Moduli_Lines}(B)), which completes the proof.
\end{proof}

\begin{figure}[ht]
\centering
\labellist
\tiny\hair 2pt
\pinlabel $E=(x_0,2)$ at 310 240
\endlabellist
\subfloat[$xy$-coordinate with \quad $x=\lambda+\lambda^{-1}$ and $y=\mu+\mu^{-1}$]{\includegraphics[scale=.35]{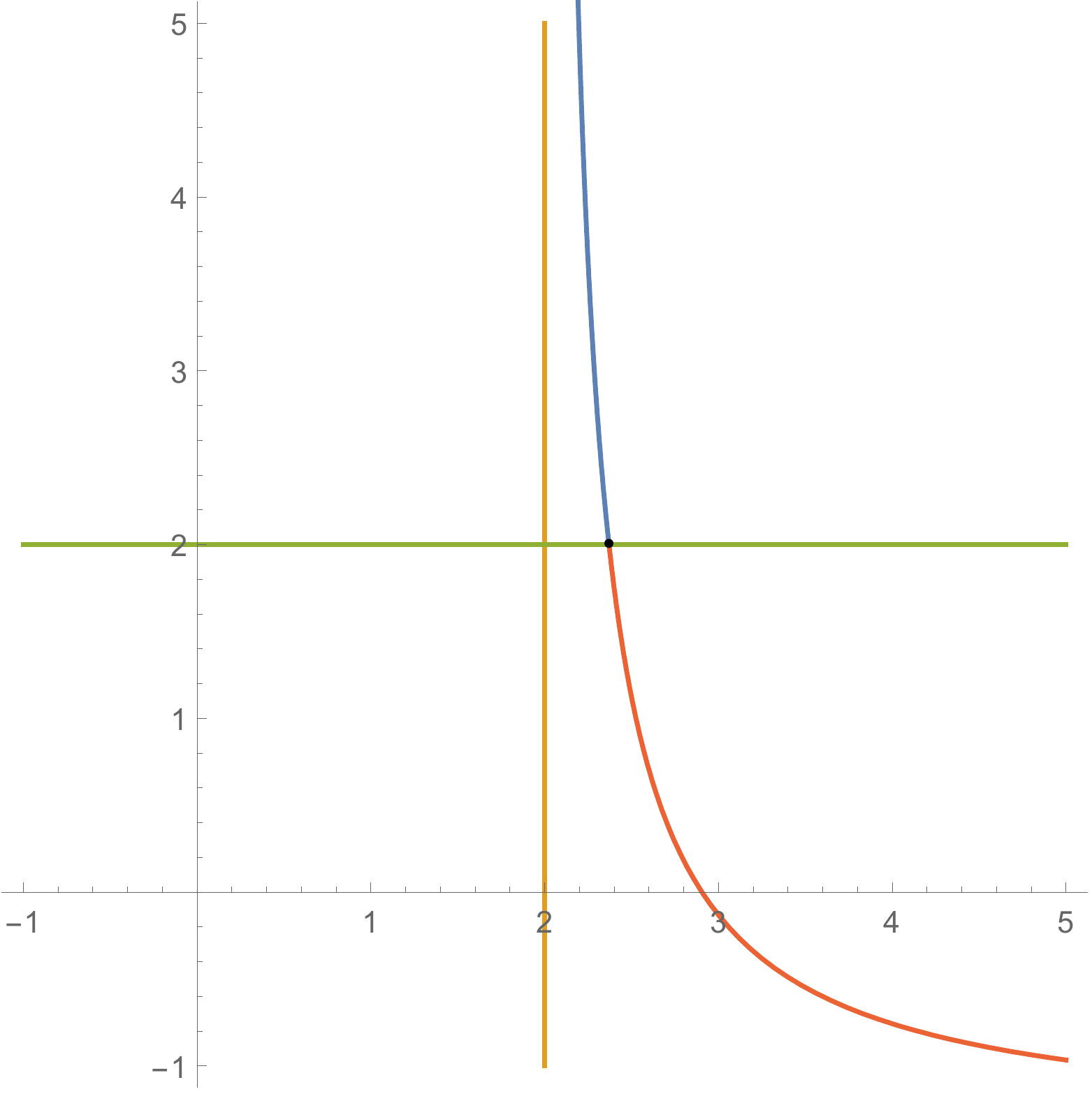}}
\quad\quad
\subfloat[$ab$-coordinate with \quad $a= \log (\lambda)$ and $b = \log (\mu)$]{\includegraphics[scale=.35]{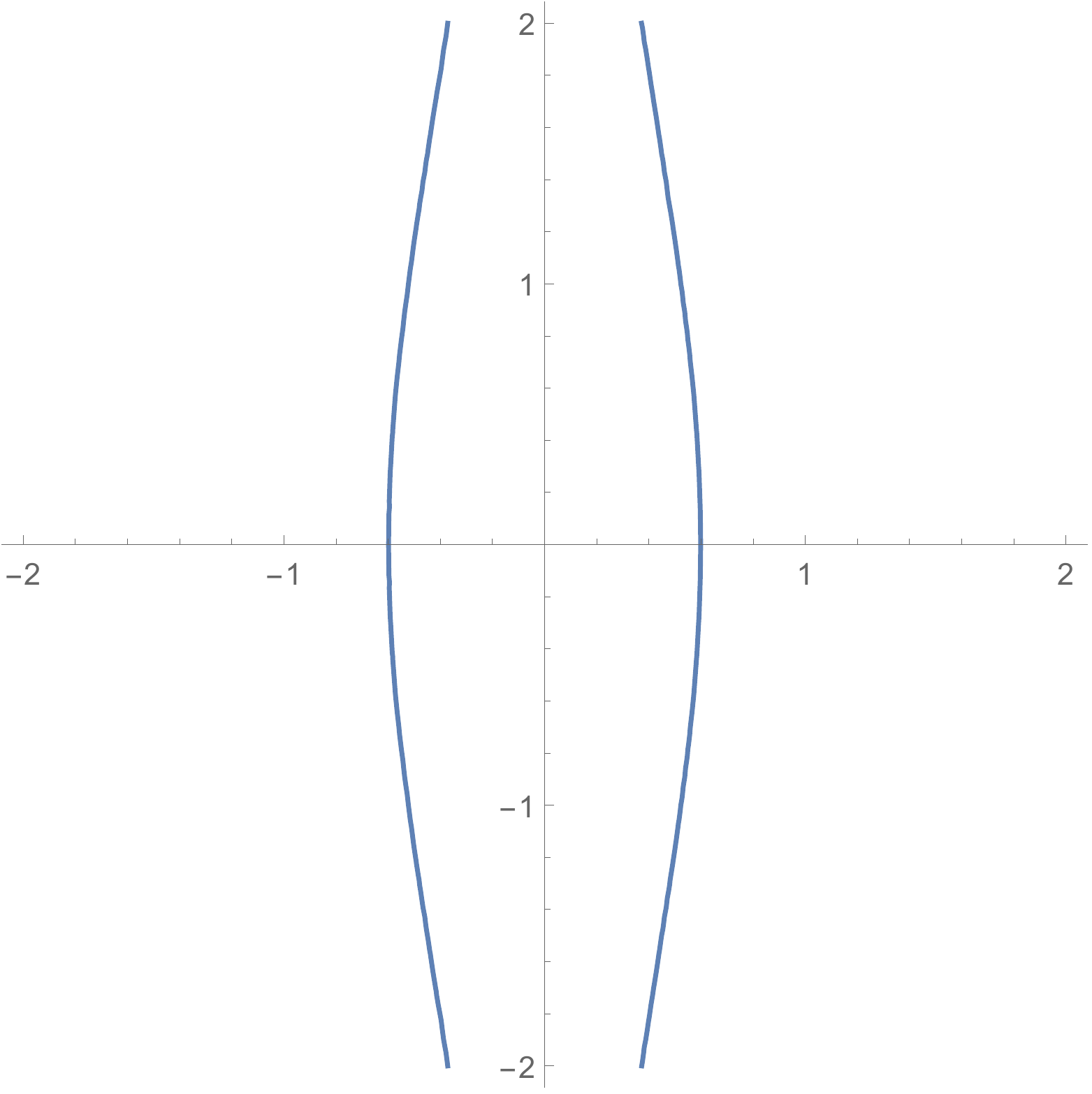}}
\caption{The deformation space $\B(\GG^3_{m})$ with $m=7$}\label{Moduli_Lines}
\end{figure}

\begin{rem}\label{rem:empty_set}
If $i \in \{1,2\}$ and $3 \leqslant m \leqslant 6$, then $\B(\GG^i_m)=\varnothing$.
\end{rem}

\begin{rem}\label{rem:mixed}
If $i=3$ and $m=3$, then the equation to solve is of the form:
$$
(x-2)(y-2)=B
$$
with $B >0$, and thus the set of solutions in $(\lambda,\mu)$ is homeomorphic to four disjoint lines.
\end{rem}

\begin{de}
For each Coxeter polytope $P \in \B({\GG^3_{m}})$, the Cartan matrix $A_P$ is equivalent to the unique special form $A_{\lambda,\mu}$, and thus the real number $\mu$ is an invariant of $P$. It is called \emph{the $\mu$-invariant of $P$} and denoted by $\mu(P)$.
\end{de}

By the parametrization of $\B({\GG^3_{m}})$, the map $\B({\GG^3_{m}}) \rightarrow \R_+^*$ given by $P \mapsto \mu(P)$ is a two-to-one covering map. More precisely, for any $\mu > 0$, there exist exactly two Coxeter polytopes $P_m$ and $P'_m$ in $\B(\GG^3_m)$ such that $A_{P_m}  = A_{\lambda,\mu}$ and $A_{P'_m}  = A_{\lambda^{-1},\mu}$ with some $\lambda > 0$.

\begin{rem}[Duality]
If $i=1,2$, then $\B(\GG_m ^i) = \{ P,P^\sharp \}$ and $A_{P^\sharp}$ is the transpose of $A_P$, and so $P$ and $P^\sharp$ are \emph{dual} in a sense that there is an outer automorphism of $\mathrm{SL}^{\pm}_{5}(\R)$ which exchanges $\G_P$ and $\G_{P^\sharp}$. If $i=3$, then the analogous statement is true for each pair $ \{ P,P^\sharp \}$ of Coxeter polytopes in $\B(\GG_m ^i)$ such that $A_P = A_{\lambda,\mu}$ and $A_{P^\sharp} = A_{\lambda^{-1},\mu^{-1}}$ (up to equivalence).
\end{rem}

\section{Limits of $P_m$}\label{section: Limit}

In this section, we study the limit of the sequence $(P_m)_m$.

\begin{propo}\label{lim}
$\,$
\begin{itemize}
\item Let $i=1, 2$. Assume $P^i_m$ is one of the two points of $\B(\GG^i_m)$. If $A_{P^i_m}$ has the special form, then as $m$ goes to infinity,  $A_{P^i_m}$ converges to the symmetric matrix:
$$
A^i_{\infty}=
\begin{pmatrix}
2 & -1 & -1 \\
 -1 &  2 & -1 \\
 -1 & -1 & 2 & -2c_{34} \\
 & & -2c_{34} & 2 & -1\\
 & & & -1 & 2 & -2 \\
 & & & & -2 & 2
\end{pmatrix}
$$
\item  Let $i=3$ and $\mu >0$. Assume $P^i_m$ is one of the two points of $\B(\GG^3_m)$ such that $\mu(P^i_m) = \mu$. If $A_{P^i_m}$ has the special form, then as $m$ goes to infinity, $A_{P^i_m}$ converges to the matrix:
$$
A^3_{\infty,\mu}=
\begin{pmatrix}
2 & -1 & -1 \\
 -1 &  2 & -1 \\
 -1 & -1 & 2 & -1 \\
 & & -1 & 2 &-1 & -\mu^{-1}\\
 & & & -1 & 2 & -2 \\
 & & & -\mu & -2 & 2
\end{pmatrix}
$$
\end{itemize}
\end{propo}

\begin{proof}
For the case $i=1,2$, recall that $m_{45}=3$, $m_{56}=m$ and $m_{64}=2$ (see Table \ref{Cox_gp}). Since $s_{56} \to 0$, we obtain $\Psi_{456}$ converges to a negative number $\Psi_{\frac{\pi}{3},0,\frac{\pi}{2}}$  as  $m \to \infty$. Remember also Equation (\ref{eq:E1}) claiming that
$$
0 = \frac{1}{64}\det(A_{P^i_m}) = \frac{1}{8} \Big( 2-(\lambda + \lambda^{-1})  \Big) \Psi_{456}-s_{12}^2 c_{34}^2 s_{56}^2
$$
So $\lambda \to 1$ as  $m \to \infty$. The case $i=3$ is similar, using Equation (\ref{eq:E2}) instead.
\end{proof}

For each $m \in \mathbb{N} \cup \{\infty\} \setminus \{0,1\}$, let $P_{m}$ be a Coxeter polytope with Coxeter group $W_{m}$ and Coxeter system $(S_{m},M_{m})$, and let $\rho_{m} : W_{m} \rightarrow \mathrm{SL}^{\pm}(V)$ be the corresponding representation. A sequence of Coxeter polytopes $(P_m)_m$ \emph{converges} to $P_{\infty}$ as $m \rightarrow \infty$ if:
\begin{enumerate}
\item There is a bijection $\hat{\phi} : S_{\infty} \rightarrow S_{m}$ which induces the homomorphism $\phi : W_{\infty} \rightarrow W_{m}$.
\item The representations $(\rho_{m} \circ \phi)_m$ converge algebraically to $\rho_{\infty} \in \mathrm{Hom}(W_{\infty}, \mathrm{SL}^{\pm}(V))$.
\end{enumerate}

Note that $\mathrm{Hom}(W_{\infty}, \mathrm{SL}^{\pm}(V))$ is the space of all homomorphisms of $W_{\infty}$ into $\mathrm{SL}^{\pm}(V)$ with the topology of pointwise convergence.

\begin{cor}\label{cor:m_infini}
Let $i=1,2$ or $3$ and $\mu > 0$. Let $P^i_m$ be one of the two points of $\B(\GG^i_m)$ (such that $\mu(P^i_m)=\mu$ if $i=3$). As $m$ goes to infinity, the sequence $(P^i_m)_m$ converges to a Coxeter polytope $P^i_{\infty}$ whose Cartan matrix is $A_{\infty}^i$ if $i=1,2$ and $A^3_{\infty,\mu}$ if $i=3$. Moreover, the underlying polytope of $P^i_{\infty}$ is the pyramid over the prism.
\end{cor}

\begin{proof}
We ignore the superscript $i$ and the index $\mu$ because the proof does not depend on $i$ and $\mu$. 

The Cartan matrix $A_{\infty}$ is irreducible, of negative type and of rank $d+1$. By Theorem \ref{theo_vin_reali}, there exists a unique Coxeter polytope $P_\infty$ realizing $A_{\infty}$. The obvious bijection between $S_{\infty}$ and $S_{m}$ and Proposition \ref{lim} immediately tell us that the sequence $(P_m)_m$ converges to a Coxeter polytope $P_{\infty}$. The only remaining part is to prove that the underlying polytope of $P_{\infty}$ is the pyramid over the prism. Let $T=\{1,2,3,5,6 \}$. The Cartan matrix $A_T = A_T^0$ and $W_T = \tilde{A}_1 \times \tilde{A}_2$ is virtually isomorphic to $\Z^3$, and thus the third part of Theorem \ref{combi_structure} shows that $T \in \sigma(P_{\infty})$ and the intersection of the five facets corresponding to $T$ is a vertex $v$ of $P_{\infty}$. By Theorem \ref{thm:classi_eucli}, the link of $P_{\infty}$ at $v$ is a parabolic Coxeter $3$-polytope whose underlying polytope is a prism, and so $P_{\infty}$ is indeed the pyramid over the prism.
\end{proof}

\section{A necessary condition for Dehn fillings}\label{general_remark}

Let $P_{\infty}$ be a Coxeter polytope with Coxeter group $W_{\infty}$ and for some $m_i \in \mathbb{N} \setminus \{0,1\}$, let $P_{m_{1}, \dotsc, m_{k}}$ be a Dehn filling of $P_{\infty}$ at a non-empty set $\V = \{v_1, \dotsc, v_k\}$ of some parabolic vertices of $P_{\infty}$, \ie the underlying labeled polytope of $P_{\infty}$ is obtained from $P_{m_{1}, \dotsc, m_{k}}$ by collapsing a ridge $r_i$ of label $\frac{\pi}{m_i}$ to the vertex $v_i$ for every $v_i \in \V$. If $A_S$ is a Cartan matrix and $T \subset S$, then $A_T$ denotes the restriction of $A_S$ to $T \times T$, and $T^0$ the subset of $T$ such that $A_{T^0} = A_T^0$.

\begin{propo}\label{prop:condition_dehn_filling}
Let $d \geqslant 3$, and let $P_{\infty}$ be an irreducible Coxeter polytope of dimension $d$. If there exists a Dehn filling $P_{m_{1}, \dotsc, m_{k}}$ of $P_{\infty}$ at $\V$ with $m_i \geqslant 3$, $i=1, \dotsc, k$, then for every $v \in \V$, the link $P_{\infty,v}$ of $P_{\infty}$ at $v$ is a Coxeter $(d-1)$-prism with Coxeter group $W_{\infty,v} = \tilde{A}_1 \times \tilde{A}_{d-2}$.
\end{propo}

\begin{proof}
For the simplicity of the notation, we assume that $|\V| = 1$ and denote by $P_m$ the $m$-Dehn filling of $P_{\infty}$ at $\V$. The proof of the general case is similar.

\medskip

Let $S_v$ be the set of facets of $P_{\infty}$ containing $v \in \V$. First, look at the $v$-link Coxeter group $W_{\infty,v}$ of $W_{\infty}$, \ie the Coxeter group of the link $P_{\infty,v}$ of $P_{\infty}$ at $v$. Since $v$ is a parabolic vertex, each component of the decomposition of the Coxeter group $W_{\infty,v}$:
$$ W_{\infty,v} = W_{\infty,v}^1 \times \cdots \times W_{\infty,v}^l \quad \textrm{for some } l \in \mathbb{N} \setminus \{0\}$$
is an irreducible affine Coxeter group (see Table \ref{affi_diag}) and $W_{\infty,v}$ is virtually isomorphic to $\Z^{d-1}$.

\medskip

The facets of $P_{\infty}$ are in correspondence with the facets of $P_m$, and so we can consider the Coxeter subgroup $W_{m,v}$ of $W_m$ corresponding to the set $S_v$. Moreover, two adjacent facets of $S_v$ in $P_{\infty}$ are also adjacent in $P_m$ with the same dihedral angle. Remark that the converse statement is false. Also, the decomposition of $W_{\infty,v}$ gives a partition of the set $ S_v = S_v^1 \sqcup \dotsm \sqcup S_v^l$, where $W_{\infty,v}^i$ corresponds to $S_v^i$ for each $i=1, \dotsc, l$, with the geometric interpretation that if $s\in S_v^i$ and $t \in S_v^j$ with $i\neq j$, then $s$ and $t$ are adjacent in $P_{\infty}$ and the dihedral angle is $\frac{\pi}{2}$. As a consequence, the decomposition of $W_{\infty,v}$ induces a decomposition of the Coxeter group $W_{m,v}$:
$$ W_{m,v} = W_{m,v}^1 \times \cdots \times W_{m,v}^l$$
At this moment, it is not obvious that each factor $W_{m,v}^i$ of this decomposition is irreducible; however, we will see in the next paragraph that this is the case.

\medskip

Let $r$ be the ridge of $P_m$ that collapses to the vertex $v$ to obtain $P_\infty$, and let $r = s \cap t$ with two facets $s,t \in S_v$. Then $W_{\{s,t\}} = I_2(m)$ in $W_m$ and $W_{\{s,t\}} = \tilde{A}_1$ in $W_{\infty}$, and so we can assume that $W_{m,v}^1 = I_2(m)$ and $W_{\infty,v}^1 = \tilde{A}_1$. Finally, collapsing the ridge $r$ is the only change of the adjacency rule, and hence for all $i=2, \dotsc, l$, we have $W_{m,v}^i=W_{\infty,v}^i$.

\medskip

Let $A_{\infty,v}$ (resp. $A_{m,v}$) be the Cartan submatrix of $A_{\infty} := A_{P_{\infty}}$ (resp. $A_m := A_{P_{m}}$) corresponding to $S_v$. According to the decomposition of $W_{\infty,v}$, we can decompose $A_{\infty,v}$ and $A_{m,v}$ as the direct sums:
$$A_{\infty,v} = A_{\infty,v}^1 \oplus \cdots \oplus A_{\infty,v}^l \quad \mathrm{and} \quad A_{m,v} = A_{m,v}^1 \oplus \cdots \oplus A_{m,v}^l$$
where each $A_{\infty,v}^i$ and $A_{m,v}^i$, $i=1, \dotsc, l$, correspond to $S_v^i$, respectively, \ie each $A_{\infty,v}^i$ (resp. $A_{m,v}^i$) is a component of $A_{\infty,v}$ (resp. $A_{m,v}$). Observe that $l \geqslant 2$: Otherwise,
$$ d-1 = \mathrm{rank}(A_{\infty,v}) = \mathrm{rank}(A_{\infty,v}^1) = 1,$$
which contradicts to the fact that $d \geqslant 3$. 

\medskip

Now, we have the following equalities or inequalities for rank:
\begin{enumerate}
\item $\mathrm{rank}(A_{m,v})  \leqslant \mathrm{rank}(A_{m}) \leqslant  d+1$,
\item $\mathrm{rank}(A_{\infty,v}) = d-1$,
\item $\mathrm{rank}(A_{m,v}^1) = \mathrm{rank}(A_{\infty,v}^1) + 1$,
\item $\mathrm{rank}(A_{m,v}^i) \geqslant \mathrm{rank}(A_{\infty,v}^i)$ \, for every $i \in  \{2, \dotsc, l\}$,
\item $\mathrm{rank}(A_{m,v}^i) = \mathrm{rank}(A_{\infty,v}^i) + 1$ \,for at least one $i \in  \{2, \dotsc, l\}$.
\end{enumerate}

The first item is trivial. The second item follows from the fact that $v$ is a parabolic vertex of $P_{\infty}$. The third item is also trivial. The fourth item is a consequence of the second item of Lemma \ref{lem:technical1} and the fact that $W_{m,v}^i=W_{\infty,v}^i$ for every $i \in  \{2, \dotsc, l\}$. 

\medskip

In order to prove the fifth item, we assume on the contrary that for every $i \in  \{2, \dotsc, l\}$, $ \mathrm{rank}(A_{m,v}^i) = \mathrm{rank}(A_{\infty,v}^i)$. Let $T = S_v^2 \sqcup \dotsm \sqcup S_v^l$. First, by the second item of Lemma \ref{lem:technical1}, for each $i \in  \{2, \dotsc, l\}$, the Cartan matrix $A_{m,v}^i $ is of zero type, and hence $A_{m,T} = A_{m,T}^0$. 

\medskip

We claim that $Z(T)^0 = \varnothing$. Indeed, suppose by contradiction that $Z(T)^0 \neq \varnothing$. Let $U = S_v^1 \sqcup Z(T)^0$. Note that $\mathrm{rank}(A_{m,v}^1) = 2$ and $\mathrm{rank}(A_{m,U}) \neq 2$, hence $\mathrm{rank}(A_{m,U}) \geqslant 3$. Moreover, 
\begin{align*}
\mathrm{rank}(A_{m,T}) & = \sum_{i=2}^{l} \mathrm{rank}(A_{m,v}^i) = \sum_{i=2}^{l} \mathrm{rank}(A_{\infty,v}^{i}) = \mathrm{rank}(A_{\infty,T}) \quad \textrm{and} \\
\mathrm{rank}{(A_{\infty,T})} & = \mathrm{rank}{(A_{\infty,v})} - \mathrm{rank}{(A_{\infty,v}^1)} = d-2.
\end{align*}
It therefore follows that:
$$ d+1 \,\leqslant\, \mathrm{rank}(A_{m,U}) +  \mathrm{rank}(A_{m,T}) = \textrm{ rank}\, (A_{m, U \sqcup T}) \leqslant \textrm{ rank}\,(A_{m}) \,\leqslant\, d+1,$$
and so $\mathrm{rank}(A_m) = \mathrm{rank}(A_{m, U \sqcup T}) = d+1$. 
Now, consider the Cartan matrix $A_m$:
\begin{equation*}
A_m
=\begin{pmatrix}
A_{m,U}    & 0          & C_{U} \\
0          & A_{m,T}    & C_{T} \\
D_{U}      & D_{T}      & B
\end{pmatrix}
\end{equation*}
where $C_{U}$, $D_{U}$, $C_{T}$, $D_{T}$ and $B$ are some matrices. First, all entries of $C_{U}$, $D_{U}$, $C_{T}$ and $D_{T}$ are negative or null. Second, since $A_{m, T}$ is of zero type, there is a vector $X_T >0$ such that $A_{m, T} X_T = 0$ (see e.g. Theorem 3 of Vinberg \cite{MR0302779}). Recall that $X > 0$ if every entry of $X$ is positive. Moreover, since $\mathrm{rank}(A_{m}) = \mathrm{rank}(A_{m, U \sqcup T})$, we also have that
$$\mathrm{rank}\,
\begin{pmatrix}
A_{m,T} \\
D_{T}
\end{pmatrix}
 = \mathrm{rank}(A_{m,T})$$
and so the vector $D_{T} X_{T}$ should be zero. Therefore, $D_{T}$ is a zero matrix, and so is $C_{T}$. However, it is impossible since $A_{m}$ and $A_\infty$ are irreducible.

\medskip

As a consequence, $T \in \sigma(P_m)$, \ie $T = \{s \in S \mid f \subset s \}$ for some face $f$ of $P_m$, by the second item of Theorem \ref{combi_structure}, and so $S_v = S_{v}^{1} \sqcup T \in \sigma(P_m)$ by the fourth item of Theorem \ref{combi_structure}. Since the face $\cap_{s \in S_v} s$ is of dimension $d - \mathrm{dim}(\mathrm{span}(\alpha_{s})_{s \in S_v})$ and $ \mathrm{dim}(\mathrm{span}(\alpha_{s})_{s \in S_v}) \geqslant \mathrm{rank}(A_{m,v}) = d$, the face
$\cap_{s \in S_v} s$ is in fact a vertex of $P_m$, which contradicts the combinatorial structure of the Dehn filling $P_m$.

\medskip

From the fifth item, we can assume without loss of generality that:
$$\mathrm{rank}(A_{m,v}^2) = \mathrm{rank}(A_{\infty,v}^2) + 1.$$
Now we have to show that $l=2$. Suppose for contradiction that $l\geqslant 3$. Since
$$ \mathrm{rank}(A_{\infty,v}) = \sum_{i=1}^{l} \mathrm{rank}(A_{\infty,v}^{i}) \quad \textrm{and} \quad \mathrm{rank}(A_{m,v}) = \sum_{i=1}^{l} \, \textrm{ rank}\, (A_{m,v}^i),$$
the previous five equalities and inequalities imply that:
$$
\mathrm{rank}(A_{m}) = \mathrm{rank}(A_{m,v}) = d+1 \quad \textrm{and} \quad \mathrm{rank}(A_{m,v}^i) = \mathrm{rank}(A_{\infty,v}^i) \textrm{ for every } i \in  \{3, \dotsc, l\}
$$
and hence by the second item of Lemma \ref{lem:technical1}, we conclude that for every $i=3, \dotsc, l$, $A_{m,v}^i$ is equivalent to $A_{\infty,v}^i$, which is of zero type.

\medskip

Now look at the Cartan matrix $A_m$:
\begin{equation*}
A_m
=\begin{pmatrix}
A_{m,v}^1  & 0      & 0      & \ldots  & 0     & C_{1} \\
0      & A_{m,v}^2  & 0      & \ldots  & 0     & C_{2} \\
0      & 0      & A_{m,v}^3  & \ldots  & 0     & C_{3}\\
\ldots & \ldots & \ldots & \ldots  & 0     & \ldots \\
0      & 0      & 0      & 0       & A_{m,v}^l & C_{l}\\
D_{1}  & D_{2}  & D_{3}  & \ldots  & D_{l} & F
\end{pmatrix}
\end{equation*}
where $C_{i}$, $D_{i}$, $i=1, \dotsc, l$, and $F$ are some matrices. As in the proof of the previous claim, we can show that for every $i = 3, \dotsc, l$, the matrices $D_{i}$ and $C_{i}$ are zero, which contradicts the fact that $A_{m}$ and $A_\infty$ are irreducible. Finally, we prove that $l=2$.

\medskip

In the previous paragraph, we show that $A_{m,v}^2$ cannot be of zero type, and therefore the Coxeter group $W_{\infty,v}^{2}=W_{m,v}^{2}$ must be an irreducible affine Coxeter group which can be realized by a Cartan matrix not of zero type. By the third item of Lemma \ref{lem:technical1}, only one family of irreducible affine Coxeter group satisfies this property: $\tilde{A}_k$ (see Table \ref{affi_diag}). After all, $W_{\infty,v}= \tilde{A}_1 \times \tilde{A}_k$ virtually isomorphic to $\Z^{k+1}$, and therefore $k=d-2$.

\medskip

Finally, the Coxeter polytope $P_{\infty,v}$ satisfies the hypothesis of Theorem \ref{thm:classi_eucli}, and so $P_{\infty,v}$ is isomorphic to $\hat{\Delta}_{\tilde{A}_1 \times \tilde{A}_{d-2}}$, which is a $(d-1)$-prism.
\end{proof}

\section{How to make more examples}\label{sec:gluing}

\subsection{Some definitions}

\subsubsection*{Truncation}

Let $\GG$ be a combinatorial polytope, and let $v$ be a vertex of $\GG$. A (combinatorial) truncation of $\GG$ at $v$ is the operation that cuts the vertex $v$ (without touching the facets not containing $v$) and that creates a new facet in place of $v$ (see Figure \ref{fig:truncation}). The truncated polytope of $\GG$ at $v$ is denoted by $\GG^{\dagger v}$. If $\GG$ is a labeled polytope, then we attach the labels $\frac \pi 2$ to all new ridges of $\GG^{\dagger v}$, and we denote this new labeled polytope again by $\GG^{\dagger v}$. 
Given a set $\V$ of vertices of $\GG$, in a similar way, we define the truncated polytope $\GG^{\dagger \V}$ of $\GG$ at $\V$.

\begin{center}
\begin{figure}[ht]
\begin{tabular}{>{\centering\arraybackslash}m{.4\textwidth} c >{\centering\arraybackslash}m{.4\textwidth}}
\begin{tikzpicture}[line cap=round,line join=round,>=triangle 45,x=1.0cm,y=1.0cm]
\clip(-2.5,-4) rectangle (4.5,4);
\draw [line width=1.6pt] (-1,1.73)-- (0,0);
\draw [line width=1.6pt] (0,0)-- (-1,-1.73);
\draw [line width=1.6pt] (0,0)-- (2,0);
\draw [dotted] (-1.5,1.73)-- (-1,1.73);
\draw [dotted] (-1,1.73)-- (-0.75,2.17);
\draw [dotted] (2,0)-- (2.25,0.43);
\draw [dotted] (2,0)-- (2.25,-0.43);
\draw [dotted] (-1,-1.73)-- (-1.5,-1.73);
\draw [dotted] (-1,-1.73)-- (-0.75,-2.17);
\draw [dotted] (-1.5,1.73)-- (-1.25,2.17);
\draw [dotted] (-1.25,2.17)-- (-0.75,2.17);
\draw [dotted] (2.25,0.43)-- (2.5,0);
\draw [dotted] (2.5,0)-- (2.25,-0.43);
\draw [dotted] (-1.5,-1.73)-- (-1.25,-2.17);
\draw [dotted] (-1.25,-2.17)-- (-0.75,-2.17);
\draw [dotted] (-1.5,1.73)-- (-1.5,-1.73);
\draw [dotted] (-0.75,-2.17)-- (2.25,-0.43);
\draw [dotted] (2.25,0.43)-- (-0.75,2.17);
\draw [dotted] (-2,3.46)-- (-1.25,2.17);
\draw [dotted] (2.5,0)-- (4,0);
\draw [dotted] (-1.25,-2.17)-- (-2,-3.46);
\draw [dotted] (-2,3.46)-- (-2,-3.46);
\draw [dotted] (-2,-3.46)-- (4,0);
\draw [dotted] (4,0)-- (-2,3.46);
\draw (0,0.5) node[anchor=north west] {$v$};

\fill [color=black] (0,0) circle (3pt);
\end{tikzpicture}
&
$\quad \rightsquigarrow \quad$
&
\definecolor{ffqqqq}{rgb}{1,0,0}
\begin{tikzpicture}[line cap=round,line join=round,>=triangle 45,x=1.0cm,y=1.0cm]
\clip(-2.49,-3.99) rectangle (4.5,4.01);
\draw [dotted] (-1.5,1.73)-- (-1,1.73);
\draw [dotted] (-1,1.73)-- (-0.75,2.17);
\draw [dotted] (2,0)-- (2.25,0.43);
\draw [dotted] (2,0)-- (2.25,-0.43);
\draw [dotted] (-1,-1.73)-- (-1.5,-1.73);
\draw [dotted] (-1,-1.73)-- (-0.75,-2.17);
\draw [dotted] (-1.5,1.73)-- (-1.25,2.17);
\draw [dotted] (-1.25,2.17)-- (-0.75,2.17);
\draw [dotted] (2.25,0.43)-- (2.5,0);
\draw [dotted] (2.5,0)-- (2.25,-0.43);
\draw [dotted] (-1.5,-1.73)-- (-1.25,-2.17);
\draw [dotted] (-1.25,-2.17)-- (-0.75,-2.17);
\draw [dotted] (-1.5,1.73)-- (-1.5,-1.73);
\draw [dotted] (-0.75,-2.17)-- (2.25,-0.43);
\draw [dotted] (2.25,0.43)-- (-0.75,2.17);
\draw [dotted] (-2,3.46)-- (-1.25,2.17);
\draw [dotted] (2.5,0)-- (4,0);
\draw [dotted] (-1.25,-2.17)-- (-2,-3.46);
\draw [dotted] (-2,3.46)-- (-2,-3.46);
\draw [dotted] (-2,-3.46)-- (4,0);
\draw [dotted] (4,0)-- (-2,3.46);
\draw [line width=1.6pt] (-1,1.73)-- (-0.25,0.43);
\draw [line width=1.6pt,color=ffqqqq] (-0.25,0.43)-- (-0.25,-0.43);
\draw [line width=1.6pt] (-0.25,-0.43)-- (-1,-1.73);
\draw [line width=1.6pt,color=ffqqqq] (-0.25,-0.43)-- (0.5,0);
\draw [line width=1.6pt,color=ffqqqq] (-0.25,0.43)-- (0.5,0);
\draw [line width=1.6pt] (0.5,0)-- (2,0);

\draw (0,1) node[anchor=north west] {$\frac\pi 2$};
\draw (-1,0.3) node[anchor=north west] {$\frac\pi 2$};
\draw (0,-0.4) node[anchor=north west] {$\frac\pi 2$};
\end{tikzpicture}
\end{tabular}
\caption{An example of truncation: $\GG \, \rightsquigarrow \, \GG^{\dagger v}$}
\label{fig:truncation}
\end{figure}
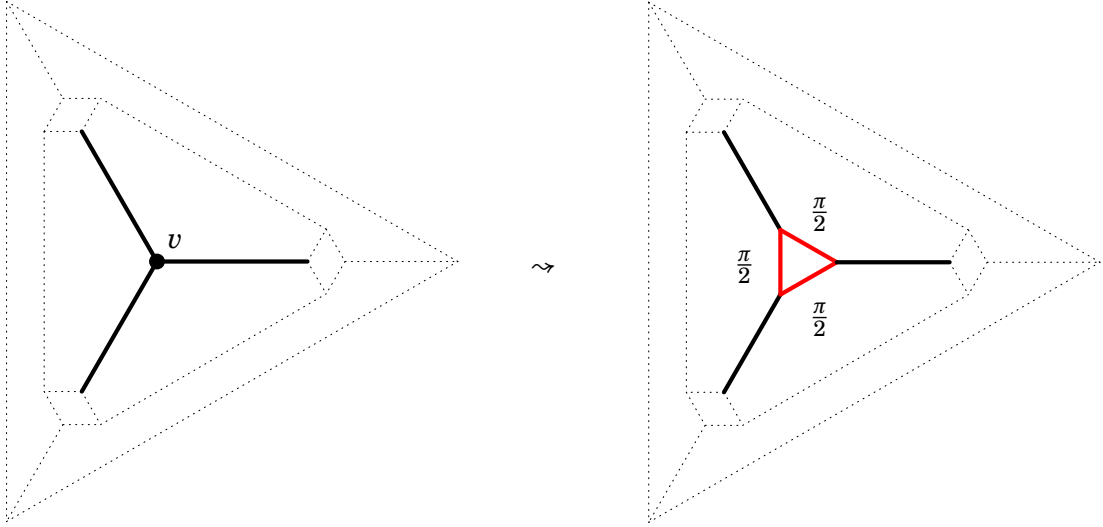
\end{center}

If $P$ is a Coxeter polytope of $\S^d$ and $v$ is a vertex of $P$, then we can also define a (geometric) truncation applied to $P$ in the case when:

\begin{de}
The vertex $v$ of $P$ is \emph{truncatable} if the projective subspace $\Pi_v$ spanned by the poles $[b_s]$ for all facets $s$ of $P$ containing $v$ is a hyperplane such that
for every (closed) edge $e$ containing $v$, the set $\Pi_v \cap e$ consists of exactly one point in the interior of $e$.
\end{de}

Assume that a vertex $v$ of $P$ is truncatable. Then a new Coxeter polytope $P^{\dagger v}$ can be obtained from $P$ by the following operation of \emph{truncation}: 
Let $\Pi_v^-$ (resp. $\Pi_v^+$) be the connected component of $\S^d \setminus \Pi_v$ which does not contain $v$ (resp. which contains $v$), and let $\overline{\Pi_v^-}$ be the closure of $\Pi_v^-$. The underlying polytope of $P^{\dagger v}$ is $\overline{\Pi_v^-} \cap P$, which has one \emph{new facet} given by the hyperplane $\Pi_v$ and the \emph{old facets} given by $P$. The reflections about the old facets of $P^{\dagger v}$ are unchanged, and the reflection about the new facet is determined by the support $\Pi_v$ and the pole $v$. It is easy to see that the dihedral angles of the ridges in the new facet of $P^{\dagger v}$ are $\frac{\pi}{2}$. 

\begin{rem}
The hyperplane $\Pi_v$ is invariant by the reflections about the facets of $P$ containing $v$, and thus the intersection $\Pi_v \cap P$ is a Coxeter polytope of $\Pi_v$ isomorphic to $P_v$. Also, note that the polytope $\overline{\Pi_v^+} \cap P$ is the pyramid over $\Pi_v \cap P$.
\end{rem}

In a similar way, if $\V$ is a subset of the set of truncatable vertices of $P$, then we may define the truncated Coxeter polytope $P^{\dagger \V}$ of $P$ at $\V$. The following lemma is an easy corollary of Proposition 4.14 and Lemma 4.17 of Marquis \cite{Marquis:2014aa}.

\begin{lemma}\label{lemma:homeo}
Let $\GG$ be a labeled polytope with Coxeter group $W$, and let $\V$ be a subset of the set of simple vertices of $\GG$. Assume that $W$ is irreducible and for every $v \in \V$, the $v$-link Coxeter group $W_v$ is Lann{\'e}r. Then the map $\chi: \B(\GG) \to \B(\GG^{\dagger \V}) $ given by $ \chi(P) = P^{\dagger \V}$ is well-defined and $ \chi$ is a homeomorphism.
\end{lemma}

\begin{rem}
If $\GG$ is a labeled polytope with only spherical and Lannér vertices and $\V$ is the set of Lannér vertices, then $\GG^{\dagger  \V}$ is perfect.
\end{rem}

\subsubsection*{Gluing}

Let $\GG_1$ and $\GG_2$ be two labeled $d$-polytopes, and let $v_i$, $i=1,2$, be a vertex of $\GG_i$. If there is an isomorphism $f : (\GG_1)_{v_1} \rightarrow (\GG_2)_{v_2}$ as labeled polytope, then we can construct a new labeled $d$-polytope, denoted by $\GG_1\,{}_{v_1}\!\!\sharp_{v_2} \GG_2$  (or simply $\GG_1\sharp \GG_2$), by gluing together two truncated polytopes $\GG_1^{\dagger  v_1}$ and $\GG_2^{\dagger  v_2}$ via $f$ (see Figure \ref{fig:gluing}).

\begin{figure}[ht]
\labellist
\small\hair 2pt
\pinlabel $\sharp$ at 310 155
\pinlabel $v_1$ at 285 150
\pinlabel $v_2$ at 340 150
\pinlabel $\GG_1$ at 125 150
\pinlabel $\GG_2$ at 500 150
\pinlabel $\GG_1$ at 875 95
\pinlabel $\sharp$ at 905 100
\pinlabel $\GG_2$ at 935 95
\pinlabel $=$ at 660 150
\endlabellist
\centering
\includegraphics[scale=0.35]{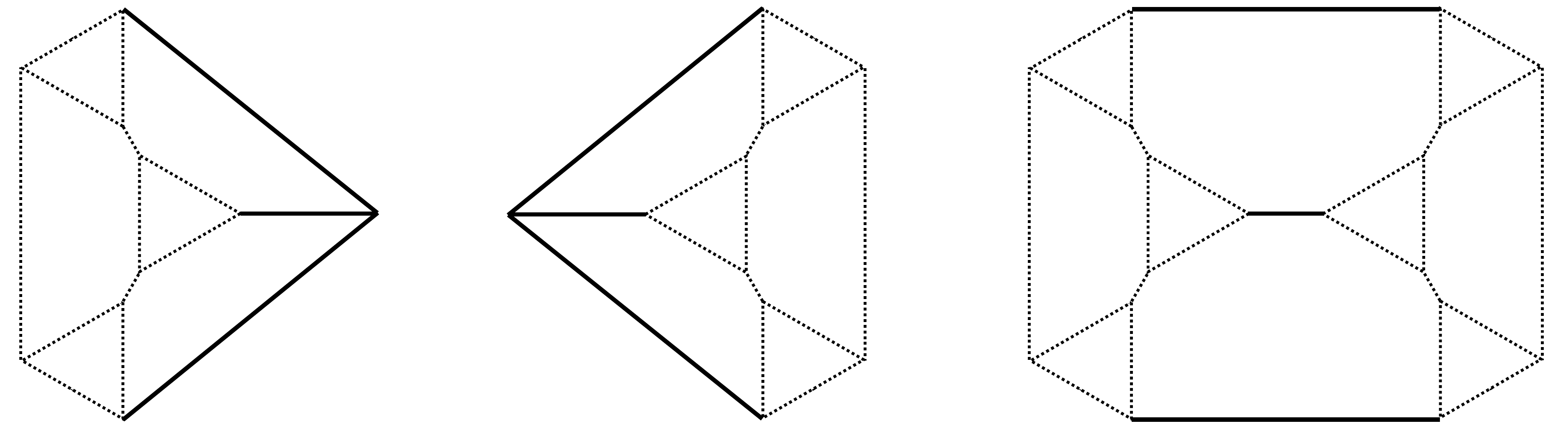}
\caption{Gluing two $3$-polytopes $\GG_1$ and $\GG_2$ along $v_1$ and $v_2$}
\label{fig:gluing}
\end{figure}

Similarly, we can also define a (geometric) gluing applied to Coxeter polytopes: Let $P_1$ and $P_2$ be Coxeter polytopes of $\S^d$ and let $v_i$, $i=1,2$, be a truncatable vertex of $P_i$. If there is an isomorphism $g : (P_1)_{v_1} \rightarrow (P_2)_{v_2} $ as Coxeter polytopes, then the Coxeter polytope $P_1 \sharp P_2$ is obtained by gluing two truncated polytopes $P_1^{\dagger v_1}$ and $P_2^{\dagger v_2}$ via $g$. For more details about truncation, we refer to \cite{CLM_ecima}.

\subsection{Examples of truncatable polytopes}

In order to check whether the Coxeter group is relatively Gromov-hyperbolic or not, we need the following theorem. Recall that if $(S,M)$ is a Coxeter system, then a subset $T$ of $S$ gives us a Coxeter subsystem $(T,M_T)$, where $M_T$ is the restrictioin of $M$ to $T \times T$. Here, we abbreviate $(T,M_T)$ to $T$. If $s, t \in S$ are commuting elements of $W_S$, then we express this fact by writing $s \perp t$. In a natural way, we may understand the following notations: $S_1 \perp S_2$ for $S_1, S_2 \subset S$, and $U^\perp$ for $U \subset S$.  

\begin{theorem}[Moussong \cite{moussong} and Caprace \cite{caprace_cox_rel-hyp,Caprace:fk}]\label{moussong_caprace}
Let $(S,M)$ be a Coxeter system, and let $\T$ be a collection of subsets of $S$. Then the group $W_S$ is Gromov-hyperbolic relative to $\{ W_T \mid T \in \T \}$ if and only if the following hold:
\begin{enumerate}
\item If $U$ is an affine subsystem of $S$ of rank $\geqslant 3$, then $U \subset T$ for some $T \in \T$.
\item If $S_1, S_2$ are irreducible non-spherical subsystems such that $S_1 \perp S_2$, then $S_1 \cup S_2 \subset T$ for some $T \in \T$.
\item If $T, T'$ are two distinct elements of $\mathcal{T}$, then $T \cap T'$ is a spherical subsystem of $S$.
\item If $T \in \T$ and $U$ is an irreducible non-spherical subsystem of $T$, then $U^{\perp} \subset T$.
\end{enumerate}
\end{theorem}

\subsubsection{First example}

Let $P \in \B(\GG_m^2)$ and let $\V$ be the set of the two Lannér vertices of $P$. By Lemma \ref{lemma:homeo}, every vertex $v \in \V$ is truncatable, and hence we obtain a perfect polytope $P^{\dagger \V}$.

\subsubsection{Second example}

Take the Coxeter group $U$ in the left diagram of Table \ref{diag_UV}.
As in Section \ref{section:labeled}, we can label the product of two triangles according to the Coxeter diagram $U$ to obtain a labeled $4$-polytope $\GG_{U}$, which has $5$ spherical vertices and $4$ Lann{\'e}r vertices. By Moussong's criterion (see Theorem \ref{moussong_caprace}), the Coxeter group $U$ is Gromov-hyperbolic; however, we will see that $U$ is not hyperbolizable, \ie $U$ cannot be the Coxeter group of a hyperbolic Coxeter $4$-polytope, by computing the deformation space of $\GG_{U}$.

\begin{table}[ht]
\centering
\begin{tabular}{ccccc}
\begin{tikzpicture}[thick,scale=0.6, every node/.style={transform shape}]
\node[draw,circle] (3) at (0,0) {3};
\node[draw,circle] (1) at (-1.5,0.866) {1};
\node[draw,circle] (2) at (-1.5,-0.866) {2};

\node[draw,circle] (4) at (1.732,0) {4};
\node[draw,circle] (5) at (1.732+1.5,0.866) {5};
\node[draw,circle] (6) at (1.732+1.5,-0.866) {6};

\draw (1) -- (2)  node[midway,left] {$4$};
\draw (2) -- (3)  node[above,midway] {};
\draw (3)--(1) node[above,midway] {};
\draw (4) -- (5) node[above,midway] {};
\draw (5) -- (6) node[right,midway] {$4$};
\draw (3) -- (4) node[above,midway] {$5$};
\draw (6) -- (4) node[above,midway] {};
\end{tikzpicture}
&
&
&
&
\begin{tikzpicture}[thick,scale=0.6, every node/.style={transform shape}]
\node[draw,circle] (3) at (0,0) {3};
\node[draw,circle] (1) at (-1.5,0.866) {1};
\node[draw,circle] (2) at (-1.5,-0.866) {2};

\node[draw,circle] (4) at (1.732,0) {4};
\node[draw,circle] (5) at (1.732+1.5,0.866) {5};
\node[draw,circle] (6) at (1.732+1.5,-0.866) {6};

\draw (1) -- (2)  node[midway,left] {};
\draw (2) -- (3)  node[below,midway] {$4$};
\draw (3)--(1) node[above,midway] {};
\draw (4) -- (5) node[above,midway] {$4$};
\draw (5) -- (6) node[right,midway] {};
\draw (3) -- (4) node[above,midway] {};
\draw (6) -- (4) node[above,midway] {};
\end{tikzpicture}
\end{tabular}
\caption{The Coxeter diagrams of $U$ and $V$}
\label{diag_UV}
\end{table}

We will compute the deformation space $\B(\GG_{U})$ of $\GG_{U}$ and see that $\B(\GG_{U})$ satisfies the surprising property of being homeomorphic to a circle. In a way similar to Section \ref{sec:the_computation}, the deformation space $\B(\GG_{U})$ is homeomorphic to $\{ (\lambda,\mu)  \in \mathbb{R}^2 \mid  \lambda,\mu > 0 \textrm{ and } \det(A_{\lambda,\mu}) = 0\}$, where:
$$
A_{\lambda,\mu}=
\begin{pmatrix}
2             & -\sqrt{2} & -\lambda^{-1} \\
-\sqrt{2}  & 2            & -1\\
-\lambda & -1 & 2 & -2c_5 \\
& & -2c_5 & 2 & -1 & -\mu^{-1}\\
& & & -1 & 2 & -\sqrt{2}\\
& & & -\mu & - \sqrt{2} & 2\\
\end{pmatrix}
$$

A straightforward computation gives us:
\begin{eqnarray*}
\frac{1}{64}\det(A_{\lambda,\mu})  & = & \bigg(1-\frac 14 -\frac 14 - \frac 12-\frac 12 \frac 12 \frac{\sqrt{2}}{2} (\lambda+\lambda^{-1})\bigg)\bigg(  -\frac{\sqrt{2}}{8} (\mu+\mu^{-1})\bigg) - c_5^2 s_4^2 s_4^2\\
\frac{1}{2} \det(A_{\lambda,\mu})   & = &  (\lambda+\lambda^{-1})( \mu+\mu^{-1})  - 8 c_5^2
\end{eqnarray*}
In other words, the deformation space $\B(\GG_{U})$ is homeomorphic to:
$$
\{ (\lambda,\mu)  \in \mathbb{R}^2 \mid  \lambda,\mu > 0 \textrm{ and } (\lambda+\lambda^{-1})( \mu+\mu^{-1})  - 8 c_5^2 = 0\}
$$
which is homeomorphic to a circle (see Figure \ref{Moduli_circle}). 
\begin{figure}[ht]
\centering
\includegraphics[scale=.35]{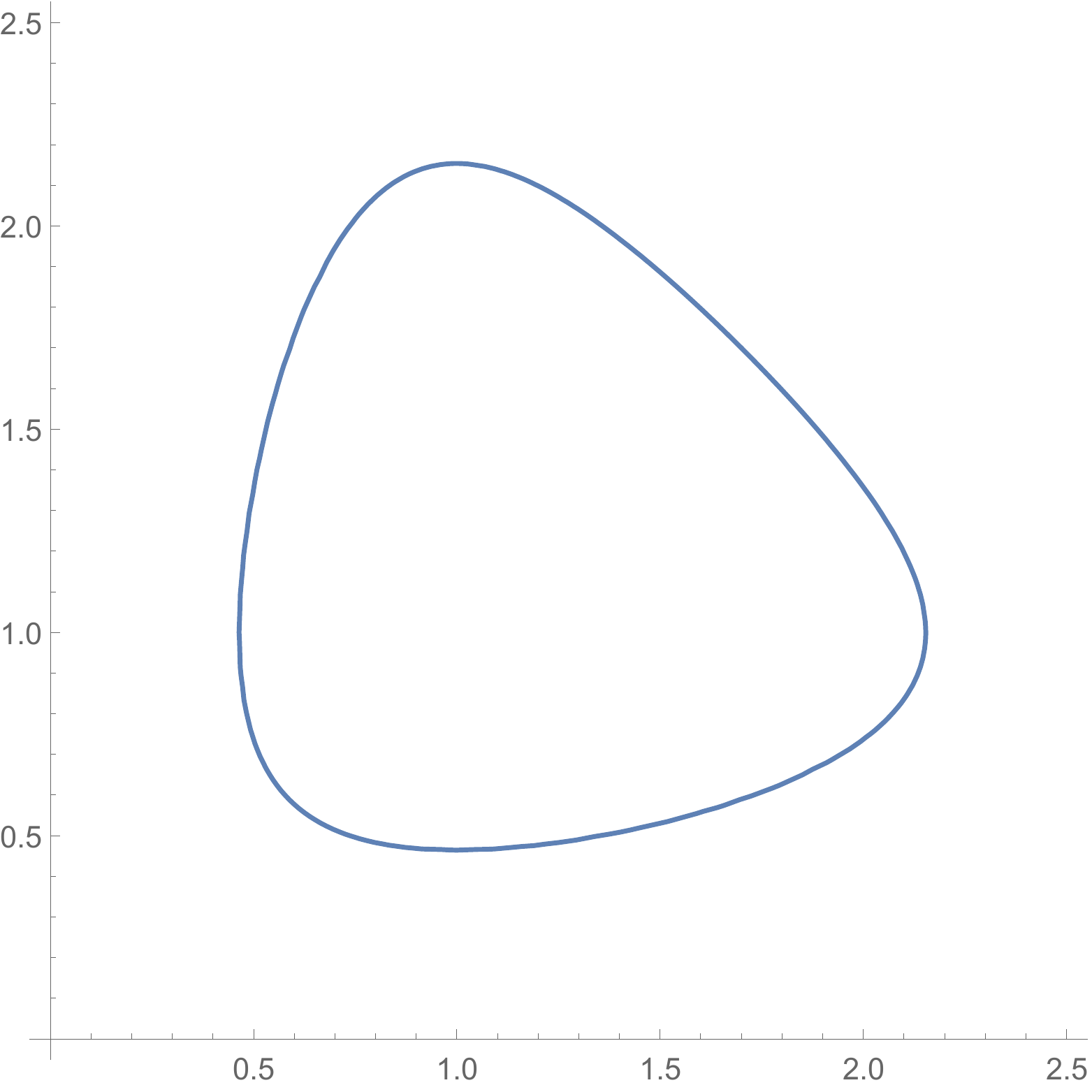}
\caption{The deformation space of $\GG_{U}$}\label{Moduli_circle}
\end{figure}

As a consequence, first $\GG_{U}$ is not hyperbolizable by Theorem \ref{thm:bilinear} since $\lambda = \mu =1$ is not in $\B(\GG_{U})$, and second we can answer the following question asked by Benoist if we weaken slightly the assumption:

\begin{qu}
Is there a perfect labeled polytope $\GG$ such that $\GG$ can be realized as a hyperbolic Coxeter polytope and the deformation space $\B(\GG)$ is compact and is of dimension $ \geqslant 1$?
\end{qu}

\begin{theorem}\label{thm:circle_guy}
There exists a perfect labeled $4$-polytope $\GG$ such that $W_{\GG}$ is Gromov-hyperbolic and $\B(\GG)$ is homeomorphic to the circle.
\end{theorem}

\begin{proof}
Let $\V$ be the set of all Lann{\'e}r vertices of $\GG_{U}$. Then the truncated polytope $\GG_{U}^{\dagger \V}$ is perfect, and by Lemma \ref{lemma:homeo},  $\B(\GG_{U}^{\dagger \V}) = \B(\GG_{U}) $, which is homeomorphic to the circle.
\end{proof}

Remark that if we take the Coxeter group $V$ in the right diagram of Table \ref{diag_UV}, then $\B(\GG_{V})$ is also homeomorphic to a circle and every Coxeter polytope $P \in \B(\GG_{V})$ is quasi-perfect and of finite volume. Note that the link of $\GG_{V}$ at the vertex $v_{2345}$ has the Coxeter group:

\begin{center}
\begin{tikzpicture}[thick,scale=0.6, every node/.style={transform shape}]
\node[draw,circle] (3) at (0,0) {3};
\node[draw,circle] (2) at (-1.5,-0.866) {2};
\node[draw,circle] (4) at (1.732,0) {4};
\node[draw,circle] (5) at (1.732+1.5,0.866) {5};

\draw (2) -- (3)  node[below,midway] {$4$};
\draw (4) -- (5) node[above,midway] {$4$};
\draw (3) -- (4) node[above,midway] {};
\end{tikzpicture}
\end{center}
which is the affine Coxeter group $\tilde{C}_{3}$ virtually isomorphic to $\mathbb{Z}^3$ (see Table \ref{affi_diag}).

\subsection{Playing Lego}

\subsubsection*{First game}

Let $N$ be an integer $\geqslant 2$, let $T$ be a linear tree with $N$ vertices, \ie each vertex of $T$ is of valence 1 or 2, and enumerate the vertices $t_1, t_2, \dotsc, t_N$ of $T$. 
\begin{figure}[ht]
\centering
\begin{tikzpicture}[thick, scale=0.6, every node/.style={transform shape}]
\node[draw,circle] (1) at (0,0) {};
\node[draw,circle] (2) at (2,0) {};
\node[draw,circle] (N-1) at (4,0) {};
\node[draw,circle] (N) at (6,0) {};

\draw (1) -- (2) node[above,midway] {};
\draw[dashed] (2) -- (N-1) node[above,midway] {};
\draw (N-1) -- (N) node[above,midway] {};
\end{tikzpicture}
\put(-110,10){$t_1$}
\put(-75,10){$t_2$}
\put(-45,10){$t_{N-1}$}
\put(-5,10){$t_{N}$}
\caption{The first game}
\end{figure}
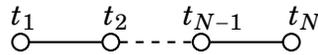

Choose $N$ integers $m_1, \dotsc ,m_N > 6$, and for each $i = 1, \dotsc, N$, select a Coxeter polytope $P_i$ from the two polytopes in $\B(\GG_{m_i}^2)$  (see Section \ref{sec:the_computation}). Each $P_i$ has two truncatable vertices whose links are isomorphic as Coxeter polytope because every Lann{\'e}r vertex of $\GG_{m_i}^2$ is rigid.   Let $\V_i$ be the set of truncatable vertices of $P_i$. If each vertex $t_i$ of $T$ corresponds to $P_i$, then the tree $T$ naturally gives us a pattern to glue the Coxeter polytopes $P_i^{\dagger \V_i}$, $i=1, \dotsc, N$, along the new facets of $P_i^{\dagger \V_i}$ in order to obtain a perfect Coxeter $4$-polytope $P_{m_1,\dotsc,m_N}$. Remark that for each $t_i$, there may be several possible positions of $P_i$.

\medskip

If we make every $m_i$ go to infinity, then the limit $P_{\infty, \dotsc ,\infty}$ is a finite volume hyperbolic Coxeter $4$-polytope with $N$ cusps. In other words, the Coxeter polytopes $P_{m_1, \dotsc ,m_N}$ are Dehn fillings of $P_{\infty, \dotsc ,\infty}$. An important remark is that since the choice of the integers $m_1, \dotsc ,m_N > 6$ are independent each other, we can apply the Dehn filling operation to each cusp of $P_{\infty, \dotsc ,\infty}$ \emph{independently}.




\subsubsection*{Second game}

Choose two integers $N, N'  \geqslant 1$ and a tree $T$ with $N+N'$ vertices such that $N$ vertices, denoted by $\{\tau_1, \dotsc, \tau_{N}\} =: \tau$, are of valence $1$ or $2$ and $N'$ vertices, denoted by $\{\tau'_1, \dotsc, \tau'_{N'}\} =: \tau'$, are of valence $1, 2, 3$ or $4$. Select $N$ integers $m_1, \dotsc, m_N > 6$. 
\begin{figure}[ht]
\centering
\begin{tikzpicture}[thick, scale=0.6, every node/.style={transform shape}]
\node[draw,circle] (t2) at (0,0) {};
\node[draw,circle] (t1) at (2,2) {};
\node[draw,circle] (u1) at (2,0) {};
\node[draw,circle] (t3) at (2,-2) {};
\node[draw,circle] (u2) at (4,0) {};
\node[draw,circle] (t4) at (4,-2) {};
\node[draw,circle] (u3) at (6,0) {};
\node[draw,circle] (t5) at (7.4,1.4) {};
\node[draw,circle] (u4) at (7.4,-1.4) {};
\node[draw,circle] (u5) at (9.4,1.4) {};

\draw (u1) -- (t2) node[above,midway] {};
\draw (t1) -- (u1) node[above,midway] {};
\draw (u1) -- (t3) node[above,midway] {};
\draw (u1) -- (u2) node[above,midway] {};
\draw (u2) -- (t4) node[above,midway] {};
\draw (u2) -- (u3) node[above,midway] {};
\draw (u3) -- (t5) node[above,midway] {};
\draw (t5) -- (u5) node[above,midway] {};
\draw (u3) -- (u4) node[above,midway] {};
\end{tikzpicture}
\put(-165,43){$\tau_2$}
\put(-125,70){$\tau_1$}
\put(-125,43){$\tau'_1$}
\put(-125,1){$\tau_3$}
\put(-94,43){$\tau'_2$}
\put(-91,1){$\tau_4$}
\put(-56,36){$\tau'_3$}
\put(-36,66){$\tau_5$}
\put(-3,66){$\tau'_5$}
\put(-36,4){$\tau'_4$}
\caption{The second game with $N = N'=5$}
\end{figure}
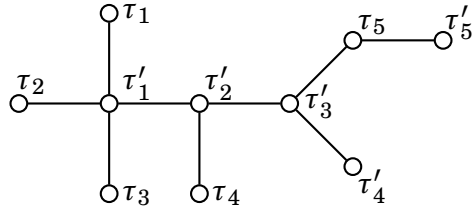

For each vertex $t \in\tau$ (resp. $\tau'$), we put a Coxeter polytope $P_i$ (resp. $Q_j$) among the polytopes in $\B(\GG_{m_i}^2)$ (resp. $\B(\GG_U)$). Each Coxeter polytope $P_i$ (resp. $Q_j$) has two (resp. four) truncatable rigid vertices, and all the links of $P_i$ and $Q_j$ at the truncatable vertices are isomorphic to each other. Let $\V_i$ (resp. $\V'_j$) be the set of truncatable vertices of $P_i$ (resp. $Q_j$). Then the tree $T$ gives us a pattern to glue the Coxeter polytopes $P_i^{\dagger \V_i}$, $i=1, \dotsc, N$, and $Q_j^{\dagger \V'_j}$, $j=1, \dotsc, N'$, along their new facets in order to obtain a perfect Coxeter $4$-polytope $P'_{m_1,\dotsc,m_N}$.

Once again, if every $m_i$ goes to infinity, then by Theorem \ref{theo_action}, the limit polytope $P'_{\infty, \dotsc ,\infty}$ is of finite volume in $\Omega_{P'_{\infty, \dotsc, \infty}}$. Remark that by Theorem \ref{moussong_caprace} of Caprace, the Coxeter group of $P'_{\infty, \dotsc ,\infty}$ is Gromov-hyperbolic relative to the standard subgroups corresponding to the parabolic vertices of $P'_{\infty, \dotsc, \infty}$, however $P'_{\infty, \dotsc ,\infty}$ \emph{cannot} be a hyperbolic Coxeter polytope since $\GG_U$ is not hyperbolizable.



\section{Higher dimensional examples}\label{section:Higher dimensional}

\subsection{The deformation spaces and the limits}

As in Section \ref{section:labeled}, we can label the product of a $(d-2)$-simplex and a triangle according to each of the 13 families of Coxeter groups $(W_m)_{m > 6}$ in Tables \ref{table:ex1}(B), \ref{table:ex1}(C), \ref{table:ex1}(D) in order to obtain a labeled polytope $\GG_{m}$.

\medskip

First, let $a_1, \dotsc ,a_{d-1}$ be the vertices in the loop $\tilde{A}_{d-2}$ of $W_m$, and $b_1,b_2,b_3$ the remaining vertices of $W_m$. Next, denote the facets of the $(d-2)$-simplex by the same symbols $a_1, \dotsc, a_{d-1}$ and the edges of the triangle by the symbols $b_1,b_2,b_3$. Then the symbols $a_1, \dotsc, a_{d-1},b_1,b_2,b_3$  are naturally assigned to the facets of $\Delta_{d-2} \times \Delta_2$. Finally, we obtain the labeled polytopes $\GG_m$ with a ridge labeling on $\Delta_{d-2} \times \Delta_2$ according to $W_m$. We extend the technique in Sections \ref{sec:the_computation} and \ref{section: Limit} to show:

\begin{theorem}\label{thm:high_dim_m_finite}
Choose one of the 13 families $(W_m)_{m > 6}$ in Tables \ref{table:ex1}(B), \ref{table:ex1}(C), \ref{table:ex1}(D).
If the Coxeter graphs of the family have one loop, then $\B({\GG_m})$ consists of two points. If the Coxeter graphs of the family have two loops, then $\B({\GG_m})$ consists of two disjoint lines.
\end{theorem}

If $m=\infty$, then we can label a pyramid over a $(d-1)$-prism according to $W_{\infty}$ in order to obtain a labeled polytope $\GG_{\infty}$. Eventually, we
can show:

\begin{theorem}\label{thm:high_dim_m_infinite}
Choose one of the 13 families $(W_m)_{m > 6}$ in Tables \ref{table:ex1}(B), \ref{table:ex1}(C), \ref{table:ex1}(D).
\begin{itemize}
\item Assume that the Coxeter graphs of the family have one loop. If $P_m$ is one of the two points of  $\B({\GG_m})$, then the sequence $(P_m)_m$ converges to the unique hyperbolic polytope $P_{\infty}$ whose Coxeter graph is $W_{\infty}$.
\item Assume that the Coxeter graphs of the family have two loops. If $\mu > 0$ and $P_m$ is one of the two points of $\B({\GG_m})$ such that $\mu(P_m) = \mu$, then $(P_m)_m$ converges to $P_{\infty}$ which is the unique Coxeter polytope such that $\mu(P_{\infty}) = \mu$ and whose Coxeter graph is $W_{\infty}$.
\end{itemize}
In any case, the underlying polytopes of $P_m$, for $m < \infty$, are $\Delta_2 \times \Delta_{d-2}$.
\end{theorem}

\begin{proof}[Proof of Theorems \ref{thm:high_dim_m_finite} and \ref{thm:high_dim_m_infinite}]
We first need to compute $\B(\GG_m)$. The proof is identical to the proof of Proposition \ref{single_moduli_finite_case}, so we leave it to the reader. To help the reader do the actual computation, we remark that if $c_1, \dots, c_n$ are real numbers and $D_\lambda$ is the determinant of the following matrix $M_\lambda$:
$$
M_\lambda=
\begin{pmatrix}
1 & -c_1 &  &  &  & -c_n\lambda^{-1}\\
-c_1 & 1 & -c_2 & \\
 & -c_2 & 1 & -c_3 \\
 & & -c_3 & 1 & \ddots \\
 &&&\ddots  & 1 &  -c_{n-1}\\
-c_n \lambda &  &  &  & -c_{n-1} & 1\\ 
\end{pmatrix}
$$
then we have the following equality:
$$
D_\lambda = D_1-c_1c_2 \cdots c_n \left( \lambda + \lambda^{-1} -2 \right)
$$
Now we need to check that $\GG_m$ is the underlying labeled polytope of any Coxeter polytopes realizing $W_m$. In the case when $m=\infty$, there is a standard subgroup of $W_{\infty}$ that is an affine Coxeter group $\tilde{A}_1 \times \tilde{A}_{d-2}$, and hence as in the proof of Corollary \ref{cor:m_infini}, $P_{\infty}$ is a pyramid over a prism (see also Proposition \ref{prop:condition_dehn_filling}).

\medskip

In the case when $m< \infty$, we need to be more careful, and we will separate the case of dimension $5$ from the case of dimensions $6$ and $7$. If we denote by $\mathrm{Pyr}^g(Q)$ the polytope obtained from $Q$ by iterating the construction $R \mapsto \mathrm{Pyr}(R)$ $g$ times, then every $d$-polytope with $d+2$ facets is of the form $\mathrm{Pyr}^g(\Delta_e \times \Delta_f)$ with $e+f+g=d$ (see Section 6.5 of Ziegler \cite{MR1311028}). It is easy to show that $\mathrm{Pyr}^g(\Delta_e \times \Delta_f)$ is a polytope with $(e+1)(f+1)+g$ vertices and among them exactly $(e+1)(f+1)$ vertices are simple. Moreover, the number of edges of $\mathrm{Pyr}^g(\Delta_e \times \Delta_f)$ is $(e+1)(f+1)\big(g+\frac{e+f}{2} \big)+\frac{g(g-1)}{2}$.

\medskip

First in dimension $d =6$ or $7$, the $3(d-1)$ standard subgroups of $W_m$ given by deleting one node of the $\tilde{A}_{d-2}$ loop and one node of the remaining part are spherical Coxeter groups of rank $d$ (or irreducible affine Coxeter groups $\tilde{E}_6$ of rank $d$ if $d=7$). It follows that the Coxeter group $W_m$ contains at least $3(d-1)$ standard subgroups of rank $d$ which are spherical (or irreducible affine of type $\tilde{E}_6$ if $d=7$), and so by the first and third items of Theorem \ref{combi_structure}, $P_m$ has at least $3(d-1)$ simple vertices.

\medskip

Now, a careful inspection shows that: \textit{if $d=6$ or $7$, then the only $d$-polytopes with at least $3(d-1)$ simple vertices are: $\Delta_2 \times \Delta_4$, $\Delta_3 \times \Delta_3$, $\Delta_2 \times \Delta_5$ and $\Delta_3 \times \Delta_4$.} In particular, $P_m$ must be simple, and we just have to eliminate the possibility of $\Delta_3 \times \Delta_3$ and $\Delta_3 \times \Delta_4$.

\medskip

In order to do so, we need to introduce the following: Let $S$ be the set of facets of a polytope $P$. A subset $A$ of $(k+1)$ elements of $S$ is \emph{prismatic} if the intersection of the facets in $A$ is empty, and for every subset $B$ of $k$ elements of $A$, the intersection of the facets in $B$ is a $(d-k)$-dimensional face of $P$. Notice that for the polytope $\Delta_e \times \Delta_f$ with $e \leqslant f$, there exists a prismatic subset of $(f+1)$ facets but no prismatic subset of $(f+2)$ facets.

\medskip
We denote by $T$ the set of nodes of the $\tilde{A}_{d-2}$ loop of $W_m$, and by $U$ the set of two nodes of $W_m$ connected by the edge of label $m$. For every subset $T'$ of $(d-2)$ elements of $T$, the intersection of the facets of $P_m$ in $T'$ is a $2$-dimensional face of $P_m$ by the first item of Theorem \ref{combi_structure}. We now claim that $T$ is a prismatic subset of $(d-1)$ facets of $P_m$, \ie the intersection $q$ of the facets in $T$ is empty. Indeed, suppose by contradiction that $q \neq \varnothing$. Then $T \in \sigma(P_m)$, \ie $T = \{ s \in S \mid q \subset s \}$, because $P_m$ is a simple polytope. Since (\emph{i}) $T \perp U$, (\emph{ii}) $U = U^+$, and (\emph{iii}) $T = T^0$ or $T = T^-$, we have that $T = (T \sqcup U)^0$ or $T = (T \sqcup U)^-$, and so $(T \sqcup U)^0 \cup (T \sqcup U)^- = T \in \sigma(P_m)$. Now, by the fourth item of Theorem \ref{combi_structure}, $T \sqcup U \in \sigma(P_m)$, which is impossible, because $P_m$ is a simple $d$-polytope and $|T \sqcup U| = d+1$. As a consequence, $P_m$ must be $\Delta_2 \times \Delta_{d-2}$.

\medskip

Second, in dimension $d = 5$, we should remark that the $30$ standard subgroups of $W_m$ given by deleting one node (resp. two nodes) of the $\tilde{A}_3$ loop and two nodes (resp. one node) of the remaining part are spherical of rank $d-1$. So the Coxeter group $W_m$ contains at least $30$ standard spherical subgroups of rank $d-1$. By Theorem \ref{combi_structure}, this implies that $P_m$ has at least $30$ edges. Finally, a careful inspection shows that the only $5$-polytope with $7$ facets and with at least $30$ edges is $\Delta_2 \times \Delta_3$.
\end{proof}

\begin{rem}
The following is the reason why we separate the arguments for dimension 5 and dimensions 6 and 7: On the one hand, the right Coxeter group in Table \ref{table:ex1}(B) with $p=q=5$ has only $6$ standard spherical subgroups of rank $5$. On the other hand, the polytope $\mathrm{Pyr}(\Delta_3 \times \Delta_3)$ has $64$ edges, however $\Delta_2 \times \Delta_5$ has only $63$ edges. 
\end{rem}

\subsection{Nine examples in dimension $5$}\label{subsec:case_p=4}

See the left diagram in Table \ref{table:ex1}(B). If $p=3$, then the Coxeter polytopes $P_m$ with $m$ finite are perfect and $\O_{P_m}$ is not strictly convex nor with $\C^1$ boundary. The polytope $P_{\infty}$ is a quasi-perfect hyperbolic Coxeter polytope with one parabolic vertex.

If $p=4$, then the polytopes $P_m$ with $m$ finite are quasi-perfect with one parabolic vertex and $\O_{P_m}$ is not strictly convex nor with $\C^1$ boundary. The polytope $P_{\infty}$ is a quasi-perfect hyperbolic Coxeter polytope with two parabolic vertices.

If $p=5$, then the polytopes $P_m$ with $m$ finite are 2-perfect with three rigid loxodromic vertices and $\O_{P_m}$ is not strictly convex. The polytope $P_{\infty}$ is a 2-perfect hyperbolic Coxeter polytope with one parabolic vertex and three rigid loxodromic vertices. 

\medskip

Now look at the right diagram in Table \ref{table:ex1}(B). Here we will only mention the properties of the Coxeter polytope $P_m$ when $m$ is finite. When $(p,q)=(3,3)$, the Coxeter polytopes $P_m$ are perfect. When $(p,q)=(4,3)$, the Coxeter polytopes $P_m$ are quasi-perfect with one parabolic vertex. When $(p,q)=(4,4)$, the Coxeter polytopes $P_m$ are quasi-perfect with two parabolic vertices. When $(p,q)=(5,3)$, the Coxeter polytopes $P_m$ are 2-perfect with three rigid loxodromic vertices. When $(p,q)=(5,4)$, the Coxeter polytopes $P_m$ are 2-perfect with one parabolic vertex and three rigid loxodromic vertices. Finally, when $(p,q)=(5,5)$, the Coxeter polytopes $P_m$ are 2-perfect with six rigid loxodromic vertices.

\medskip

An important remark is that when $p=5$, we can apply the operation of truncation to $P_m$ and of gluing in order to build infinitely many non-prime examples. 

\subsection{Two examples in dimension $6$}

The two Coxeter polytopes we find in dimension $6$ are perfect when $m$ is finite (see Table\ref{table:ex1}(C)). The following remark explains why we \emph{cannot}} contruct infinitely many non-prime examples in dimension $\geqslant 6$. In other words, it is hard to find truncatable vertices of polytopes of dimension $\geqslant 6$. 

\begin{rem}\label{rem:up_to_five}
By the work of Lannér, we know that large perfect Coxeter simplex, which corresponds to Lann{\'e}r Coxeter group, exists only up to dimension $4$. However, Lemma \ref{lemma:homeo} which is the starting point for truncation requires a truncatable vertex to be simple, and hence the operation of truncation and gluing easily works only up to dimension $5$.
\end{rem}

\begin{rem}
There exists a Coxeter polytope with a \emph{non-truncatable non-simple} loxodromic perfect vertex (See Appendix \ref{example_non-truncatable} for more details). 
\end{rem}

\subsection{Two examples in dimension $7$}

The two Coxeter polytopes we find in dimension $7$ are quasi-perfect when $m$ is finite (see Table \ref{table:ex1}(D)). Note $P_m$ is \emph{not} perfect.

\section{Proof of the main theorems}\label{sec:main_theo}

\subsection{Proof of Theorem \ref{MainThm1}}

\begin{proof}[Proof of Theorem \ref{MainThm1}]
As mentioned in Remark \ref{rem:dehn_fill_labeled}, Theorems \ref{thm:general_4.1}, \ref{thm:general_4.2}, \ref{thm:high_dim_m_finite} and \ref{thm:high_dim_m_infinite} provide finite volume hyperbolic Coxeter polytopes $P_{\infty}$ of dimension $d=4, 5, 6, 7$ that admit Dehn fillings $(P_m)_{m > 6}$. 

\medskip

For each $m > 6$ or $m=\infty$, let $(S_{m},M_{m})$ be the Coxeter system of $P_{m}$, let $W_m := W_{P_m}$, and let $\sigma_m : W_m \rightarrow \mathrm{SL}_{d+1}^{\pm}(\mathbb{R})$ be the faithful representation of $W_m$ obtained from $P_m$ as in Tits-Vinberg's Theorem \ref{theo_vinberg}. The image $\Gamma_m := \Gamma_{P_m} = \sigma_m(W_m)$ is a discrete subgroup of $\mathrm{SL}_{d+1}^{\pm}(\mathbb{R})$ acting properly discontinuously on the properly convex domain $\O_m$ of $\mathbb{S}^d$. We denote by $\phi : W_{\infty} \rightarrow W_m$ the natural surjective homomorphism induced by the obvious bijection between $S_{\infty}$ and $S_{m}$, and set $ \rho_m := \sigma_m \circ \phi$.

\medskip

Since the Coxeter polytopes $(P_m)_{m > 6}$ and $P_{\infty}$ are quasi-perfect, the action of $\G_m$ on $\O_m$ and of $\G_\infty$ on $\O_\infty$ is of finite covolume by Theorem \ref{theo_action}. In fact, in the case where $P_m$ is perfect, the action is cocompact by definition of the perfectness (see Section \ref{subsec:def_perfect}).

\medskip

We denote by $\mathcal{O}_m$ the underlying smooth orbifold of the convex real projective $d$-orbifold $\Omega_m / \Gamma_m$ for $m > 6$ or $m = \infty$. Then the induced representation $W_\infty / \ker(\rho_m) \rightarrow \mathrm{SL}^{\pm}_{d+1}(\mathbb{R})$ of $\rho_m$, which may be identified with $\sigma_m$, is the holonomy representation of the properly convex real projective structure on the $m$-Dehn filling $\mathcal{O}_m$ of $\mathcal{O}_{\infty}$, by the first paragraph of this proof.

\medskip

By Corollary \ref{cor:m_infini}, the sequence of Coxeter polytopes $(P_m)_{m > 6}$ converges to $P_\infty$, hence the sequence of representations $(\rho_m)_{m > 6}$ converges algebraically to $\rho_\infty$.

\medskip

Since the Coxeter polytopes $(P_m)_{m > 6}$ and $P_{\infty}$ are quasi-perfect, the convex domains $(\Omega_m)_{m > 6}$ (resp. $\Omega_\infty$) that are different from $\mathbb{S}^d$ and that are invariant by $\rho_{m}(W_{\infty})$ (resp. $\rho_{\infty}(W_{\infty})$) are unique by Theorem 8.2 of \cite{Marquis:2014aa}.

\medskip

Eventually, the properly convex closed subsets $(\overline{\Omega_m})_{m > 6}$ of $\mathbb{S}^d$ converges to a convex closed set $K$ in some affine chart of $\mathbb{S}^d$ (up to subsequence), and so the uniqueness of the $\rho_{\infty}(W_{\infty})$-invariant convex domain $\Omega_{\infty}$ implies that $K = \overline{\Omega_\infty}$. As a consequence, the sequence of convex sets $(\overline{\Omega_m})_{m > 6} $ converges to $\overline{\Omega_\infty}$ in the Hausdorff topology. 
\end{proof}

\subsection{Proof of Theorem \ref{MainThm2}}

\begin{proof}[Proof of Theorem \ref{MainThm2}]
Let $W_{\infty}$ (resp. $W_{m}$) be the orbifold fundamental group of $\mathcal{O}_{\infty}$ (resp. $\mathcal{O}_{m}$). By (a refined version of) Selberg's Lemma (see for example Theorem 3.1 of McReynolds, Reid and Stover \cite{MR3118413}), there is a torsion-free finite-index normal subgroup $\G_{\infty}$ of $W_{\infty}$ such that the quotient $M_{\infty}= \O_{\infty}/\Gamma_{\infty}$ is a manifold \emph{with toral boundary}. Let $\G_m$ be the projection of $\Gamma_{\infty}$ by the surjective homomorphism $W_{\infty} \to W_m$. The subgroup $\G_m$ of $W_m$ is of finite index bounded above by the index of $\G_{\infty}$ in $W_{\infty}$, and the closed $d$-orbifolds $M_{m} = \O_m/\G_m$ are Dehn fillings of $M_{\infty}$.
\end{proof}

\begin{rem}\label{rem:filling_manifold}
In the proof of Theorem \ref{MainThm2}, if $m$ is sufficiently big, then the orbifold $M_m$ is not a manifold.
\end{rem}

\begin{proof}[Proof of Remark \ref{rem:filling_manifold}]
Assume for contradiction that $\G_m$ is torsion-free. Since $\Gamma_m$ is torsion-free, the restriction of the map $W_m \to W_m/\Gamma_m$ to $D_m$ is injective. This implies that the index of $\Gamma_m$ inside $W_m$ is unbounded, which is a contradiction.
\end{proof}

\subsection{Proof of Theorems \ref{whynew}, \ref{fewnew} and \ref{thm:mixed}}

The main tools to prove Theorem \ref{whynew} are Theorem \ref{moussong_caprace} of Moussong and Caprace and Theorem \ref{drutu_sapir} of Dru{\c{t}}u and Sapir \cite{drutu_sapir}. We state a simplifed version of the main theorem in \cite{drutu_sapir}. Recall the definition of quasi-isometry:
Let $(X,d_X)$ and $(Y,d_Y)$ be metric spaces. A map $f :  X \to Y$ is called a \emph{$(c,D)$-quasi-isometric embedding} if there exist constants $c \geq 1 $ and $D \geq 0$ such that for all $x, x' \in X$,
$$
\frac{1}{c} d_X(x,x') - D \leq d_Y(f(x),f(x')) \leq c d_X(x,x') + D.
$$
Moreover, a quasi-isometric embedding $f :  X \to Y$ is a \emph{quasi-isometry} if there is $E >0$ such that the $E$-neighborhood of $f(X)$ is $Y$. 

\begin{theorem}[Dru{\c{t}}u and Sapir \cite{drutu_sapir}]\label{drutu_sapir}
Let $\G$ be a finitely generated Gromov hyperbolic group relative to a family $A_1, \dotsc, A_r$ of virtually abelian subgroups of rank at least two. Let $\Lambda$ be a finitely generated group and $B$ a virtually abelian group of rank at least two.
\begin{itemize}
\item For every quasi-isometric embedding of $B$ into $\G$, the image of $B$ is at bounded distance from a coset $\gamma A_i$ for some $\gamma \in \Gamma$ and some $i \in \{1, \dotsc, r\}$.

\item If $\Lambda$ is quasi-isometric to $\G$, then $\Lambda$ is Gromov-hyperbolic relative to a family $B_1, \dotsc, B_s$ of subgroups, each of which is quasi-isometric to $A_i$ for some $i \in \{1, \dotsc, r \}$.
\end{itemize}
\end{theorem}

\begin{proof}[Proof of Theorem \ref{whynew}]
The fact that the Coxeter group $\G$ is Gromov-hyperbolic relative to a family of standard subgroups $H_1, \dotsc, H_k$ isomorphic to $\tilde{A}_{d-2}$ is an easy consequence of Theorem \ref{moussong_caprace} of Caprace. Recall that the Coxeter group $\tilde{A}_{d-2}$ is a virtually abelian group of rank $d-2$.

\medskip

By Lemma 8.19 of Marquis \cite{Marquis:2014aa}, the convex domain $\O$ contains a properly embedded $(d-2)$-simplex. Now, we show that the convex domain $\O$ does not contain a properly embedded $(d-1)$-simplex.

\medskip

First, note that a $(d-1)$-simplex equipped with Hilbert metric is bilipshitz equivalent to the Euclidean space $\R^{d-1}$ (see de la Harpe \cite{dlHarpe}). Second, since $\G$ divides $\O$, by \u{S}varc-Milnor lemma, the group $\G$ is quasi-isometric to the metric space $(\O,d_{\O})$. Hence, any properly embedded $(d-1)$-simplex induces a quasi-isometric embedding $f$ of $\R^{d-1}$ into $\G$, which is Gromov-hyperbolic relative to a family of subgroups $H_1, \dotsc, H_k$ virtually isomorphic to $\Z^{d-2}$. By the first item of Theorem \ref{drutu_sapir}, the embedding $f$ induces a quasi-isometric embedding of $\R^{d-1}$ into a coset $\gamma H_i$ for some $\gamma \in \Gamma$ and some $i \in \{1, \dotsc, k\}$, which is quasi-isometric to $\mathbb{Z}^{d-2}$, but there is no quasi-isometric embedding of $\R^{d-1}$ into $\Z^{d-2}$.

\medskip

Finally, the quotient orbifold $\Omega / \Gamma$ does not decompose into hyperbolic pieces, \ie it is not homeomorphic to a union along the boundaries of finitely many $d$-orbifolds each of which admits a finite volume hyperbolic structure on its interior, because otherwise $\G$ would be Gromov-hyperbolic relative to a finite collection of subgroups virtually isomorphic to $\Z^{d-1}$, and $\G$ is Gromov-hyperbolic relative to a finite collection of subgroups virtually isomorphic to $\Z^{d-2}$, which would contradict the item 2 of Theorem \ref{drutu_sapir}.
\end{proof}

\begin{proof}[Proof of Theorem \ref{fewnew}]
See the Coxeter graphs of Table \ref{table:ex2}. Each of those Coxeter graph $W$ is obtained by linking a Coxeter subgraph $W_{left}$ which is an $\tilde{A}_2,\, \tilde{A}_3$ or $\tilde{A}_4$ Coxeter graph to a disjoint Coxeter subgraph $W_{right}$ which is Lannér, by an edge with no label.

\medskip

One can label the product $\GG$ of two simplex $\Delta_e \times \Delta_f$ with $e$ (resp. $f$) being the rank of $W_{left}$ (resp. $W_{right}$) minus one using $W$ as in Section \ref{section:labeled}. After that, the same computation as the one in Section \ref{sec:the_computation} shows that $\B(\GG)$ consists of two points if $W$ has one loop, and two disjoint lines if $W$ has two loops.

\medskip

In the case when $W$ is the left Coxeter diagram with $k=4$ in Table \ref{table:ex2}(A) or a Coxeter diagram in Tables \ref{table:ex2}(B) or \ref{table:ex2}(C), the labeled polytope $\GG$ is perfect: It is a consequence of the fact that for every node $s$ of $W_{left}$, and every node $t$ of $W_{right}$, the Coxeter group obtained by deleting the nodes $s$ and $t$ of $W$ is spherical. However, in the case when $W$ is the left Coxeter diagram with $k=5$, or the middle or the right Coxeter diagram in Table \ref{table:ex2}(A), the labeled polytope $\GG$ is \emph{not} perfect, but $2$-perfect with two Lannér vertices $v$ and $w$. The truncated polytope $\GG^{\dagger \V}$ with $\V = \{v,w\}$ is perfect and $\B(\GG^{\dagger \V}) = \B(\GG)$. In particular, $\B(\GG^{\dagger \V})$ is non-empty. Let $P$ be a Coxeter polytope in $\B(\GG)$ (in the first case) or in $\B(\GG^{\dagger \V})$ (in the second case), let $\Gamma := \Gamma_{P}$, and let $\O := \O_{P}$.

\medskip

The same arguments as for Theorem \ref{whynew} show that the group $W$, which is isomorphic to $\Gamma$, is Gromov-hyperbolic relative to the virtually abelian subgroup $W_{left}$, that $\O$ contains only tight properly embedded simplex of the expected dimension and that the quotient orbifold $\Omega / \Gamma$ does not decompose into hyperbolic pieces.
\end{proof}

\begin{proof}[Proof of Theorem \ref{thm:mixed}]
See the Coxeter graphs of Table \ref{ex:mix_examples}. The only difference with Theorem \ref{fewnew} is that the computation for the existence of the Coxeter polytopes $P$ which give the discrete subgroups $\Gamma := \Gamma_P$ of $\mathrm{SL}_{d+1}^{\pm}(\mathbb{R})$ acting cocompactly on convex domains $\O := \O_P$ of $\mathbb{S}^d$ is this time an extension of the computation of Remark \ref{rem:mixed}.
\end{proof}


\clearpage

\appendix

\section{A Coxeter polytope with a non-truncatable vertex}\label{example_non-truncatable}

In this appendix, we will construct $2$-perfect Coxeter polytopes of dimension $3$ with a \emph{non-truncatable} loxodromic perfect vertex (see also Proposition 4 of Choi \cite{Choi2006}).

\medskip

Let $P_{\lambda}$ be a Coxeter $3$-polytope of $\mathbb{S}^3$ defined by (see Figure \ref{fig:non-truncatable}):
\begin{displaymath}
\begin{array}{lll}
\alpha_1 = (1,0,0,0)                                            & \quad\quad\quad &  b_1 = (2,-1,0,0)^T \\
\alpha_2 = (0,1,0,0)                                            & \quad\quad\quad & b_2 =(-1,2,-2\lambda,0)^T \\
\alpha_3 = (0,0,1,0)                                            & \quad\quad\quad & b_3 =(0,-2\lambda,2,-1)^T \\
\alpha_4 = (0,0,0,1)                                            & \quad\quad\quad & b_4 =(0,0,-1,2)^T \\
\alpha_5 = \left( 0,1,\frac{1}{\lambda},-\frac{2(\lambda^2-1)}{\lambda} \right)   & \quad\quad\quad & b_5 =(-1,0,0,-\frac{\lambda}{\lambda^2-1})^T \\
\end{array}
\end{displaymath}
where $\lambda > 1$ and row vectors $v$ (resp. column vectors $v^{T}$) mean linear forms of $\mathbb{R}^4$ (resp. vectors of $\mathbb{R}^4$). Then the Cartan matrix $A_{\lambda}$ of $P_{\lambda}$ is:
\begin{displaymath}
A_{\lambda} = \left( \begin{array}{ccccc}
2  & -1        & 0         & 0                          & -1 \\
-1 & 2         & -2\lambda & 0                          & 0 \\
0  & -2\lambda & 2         & -1                         & 0 \\
0  & 0         & -1        & 2                          & -\frac{\lambda}{\lambda^2-1} \\
-1 & 0         & 0         & \frac{3}{\lambda}-4\lambda & 2
\end{array}
\right)
\end{displaymath}

\newcommand{\scalee}{1.1}
\begin{figure}[ht!]
\subfloat{
\begin{tikzpicture}[thick,scale=\scalee, every node/.style={transform shape}]
\draw (0,0) -- (3,0) -- (3,3) -- (0,3) -- (0,0);
\draw (0,0) -- (3,3);
\draw (0,3) -- (3,0);
\node[draw,circle, inner sep=2pt, minimum size=2pt] (1) at (3.5,1.5) {$1$};
\node[draw,circle, inner sep=2pt, minimum size=2pt] (2) at (2.4,1.5) {$2$};
\node[draw,circle, inner sep=2pt, minimum size=2pt] (3) at (0.7,1.5) {$3$};
\node[draw,circle, inner sep=2pt, minimum size=2pt] (4) at (1.5,0.7) {$4$};
\node[draw,circle, inner sep=2pt, minimum size=2pt] (5) at (1.5,2.3) {$5$};
\node[fill=white, inner sep=2pt, minimum size=2pt] (6) at (1.5,0) {$\tfrac{\pi}{2}$};
\node[fill=white, inner sep=2pt, minimum size=2pt] (7) at (1.5,3) {$\tfrac{\pi}{3}$};
\node[fill=white, inner sep=2pt, minimum size=2pt] (8) at (3,1.5) {$\tfrac{\pi}{3}$};
\node[fill=white, inner sep=2pt, minimum size=2pt] (9) at (0,1.5) {$\tfrac{\pi}{2}$};
\node[fill=white, inner sep=2pt, minimum size=2pt] (10) at (0.7,0.7) {$\tfrac{\pi}{3}$};
\node[fill=white, inner sep=2pt, minimum size=2pt] (11) at (0.7,2.3) {$\tfrac{\pi}{2}$};
\node[fill=white, inner sep=2pt, minimum size=2pt] (11) at (2.3,2.3) {$\tfrac{\pi}{2}$};
\node[fill=white, inner sep=2pt, minimum size=2pt] (11) at (2.3,0.7) {$\tfrac{\pi}{2}$};
\end{tikzpicture}
}
\quad\quad\quad
\subfloat{
\begin{tikzpicture}[thick,scale=\scalee, every node/.style={transform shape}]
\node[draw,circle, inner sep=2pt, minimum size=2pt] (1) at (36-18:1.0514) {$5$};
\node[draw,circle, inner sep=2pt, minimum size=2pt] (2) at (108-18:1.0514){$1$};
\node[draw,circle, inner sep=2pt, minimum size=2pt] (3) at (180-18:1.0514){$2$};
\node[draw,circle, inner sep=2pt, minimum size=2pt] (4) at (252-18:1.0514){$3$};
\node[draw,circle, inner sep=2pt, minimum size=2pt] (5) at (324-18:1.0514){$4$};
\draw (1)--(2)--(3)--(4)--(5)--(1);
\draw (3)--(4) node[left, midway] {$\infty$};
\draw (1)--(5) node[right, midway] {$\infty$};
\draw (0,-1.5) node[]{} ;
\end{tikzpicture}
}
\caption{A Coxeter $3$-pyramid $P_{\lambda}$ over a quadrilateral}\label{fig:non-truncatable}
\end{figure}
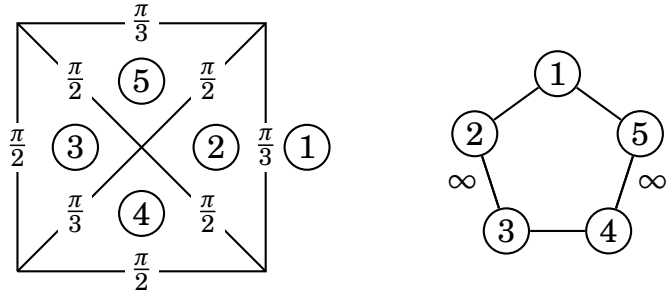

It is easy to check that $P_{\lambda}$ is a $2$-perfect Coxeter $3$-pyramid over a quadrilateral such that the only non-spherical vertex of $P_{\lambda}$ is the apex of $P_{\lambda}$, which is loxodromic and perfect. Moreover, a computation shows that:
 $$\det(b_2,b_3,b_4,b_5) = 0 \quad \Leftrightarrow \quad \lambda = \sqrt{\frac{3}{2}} \quad \Leftrightarrow \quad A_{\lambda} \textrm{ is symmetric} \quad \Leftrightarrow \quad P_{\lambda} \textrm{ is hyperbolic.}$$  
It therefore follows that if $P_{\lambda}$ is \emph{not} a hyperbolic Coxeter polytope, then the apex of $P_{\lambda}$ is a \emph{non-truncatable} loxodromic perfect vertex. 

\clearpage

\section{Useful Coxeter diagrams}\label{classi_diagram}

\subsection{Spherical and affine Coxeter diagrams}

The irreducible spherical and irreducible affine Coxeter groups were classified by Coxeter \cite{classi_spherical}. We reproduce the list of those Coxeter diagrams in Figures \ref{spheri_diag} and \ref{affi_diag}. We use the usual convention that an edge that should be labeled by $3$ has in fact no label.

\begin{table}[ht]
\centering
\begin{minipage}[b]{7.5cm}
\centering
\begin{tikzpicture}[thick,scale=0.6, every node/.style={transform shape}]
\node[draw,circle] (A1) at (0,0) {};
\node[draw,circle,right=.8cm of A1] (A2) {};
\node[draw,circle,right=.8cm of A2] (A3) {};
\node[draw,circle,right=1cm of A3] (A4) {};
\node[draw,circle,right=.8cm of A4] (A5) {};
\node[left=.8cm of A1] {$A_n\,(n\geq 1)$};

\draw (A1) -- (A2)  node[above,midway] {};
\draw (A2) -- (A3)  node[above,midway] {};
\draw[loosely dotted,thick] (A3) -- (A4) node[] {};
\draw (A4) -- (A5) node[above,midway] {};


\node[draw,circle,below=1.2cm of A1] (B1) {};
\node[draw,circle,right=.8cm of B1] (B2) {};
\node[draw,circle,right=.8cm of B2] (B3) {};
\node[draw,circle,right=1cm of B3] (B4) {};
\node[draw,circle,right=.8cm of B4] (B5) {};
\node[left=.8cm of B1] {$B_n\,(n\geq 2)$};

\draw (B1) -- (B2)  node[above,midway] {$4$};
\draw (B2) -- (B3)  node[above,midway] {};
\draw[loosely dotted,thick] (B3) -- (B4) node[] {};
\draw (B4) -- (B5) node[above,midway] {};


\node[draw,circle,below=1.5cm of B1] (D1) {};
\node[draw,circle,right=.8cm of D1] (D2) {};
\node[draw,circle,right=1cm of D2] (D3) {};
\node[draw,circle,right=.8cm of D3] (D4) {};
\node[draw,circle, above right=.8cm of D4] (D5) {};
\node[draw,circle,below right=.8cm of D4] (D6) {};
\node[left=.8cm of D1] {$D_n\,(n\geq 4)$};

\draw (D1) -- (D2)  node[above,midway] {};
\draw[loosely dotted] (D2) -- (D3);
\draw (D3) -- (D4) node[above,midway] {};
\draw (D4) -- (D5) node[above,midway] {};
\draw (D4) -- (D6) node[below,midway] {};


\node[draw,circle,below=1.2cm of D1] (I1) {};
\node[draw,circle,right=.8cm of I1] (I2) {};
\node[left=.8cm of I1] {$I_2(p)$};

\draw (I1) -- (I2)  node[above,midway] {$p$};


\node[draw,circle,below=1.2cm of I1] (H1) {};
\node[draw,circle,right=.8cm of H1] (H2) {};
\node[draw,circle,right=.8cm of H2] (H3) {};
\node[left=.8cm of H1] {$H_3$};

\draw (H1) -- (H2)  node[above,midway] {$5$};
\draw (H2) -- (H3)  node[above,midway] {};


\node[draw,circle,below=1.2cm of H1] (HH1) {};
\node[draw,circle,right=.8cm of HH1] (HH2) {};
\node[draw,circle,right=.8cm of HH2] (HH3) {};
\node[draw,circle,right=.8cm of HH3] (HH4) {};
\node[left=.8cm of HH1] {$H_4$};

\draw (HH1) -- (HH2)  node[above,midway] {$5$};
\draw (HH2) -- (HH3)  node[above,midway] {};
\draw (HH3) -- (HH4)  node[above,midway] {};


\node[draw,circle,below=1.2cm of HH1] (F1) {};
\node[draw,circle,right=.8cm of F1] (F2) {};
\node[draw,circle,right=.8cm of F2] (F3) {};
\node[draw,circle,right=.8cm of F3] (F4) {};
\node[left=.8cm of F1] {$F_4$};

\draw (F1) -- (F2)  node[above,midway] {};
\draw (F2) -- (F3)  node[above,midway] {$4$};
\draw (F3) -- (F4)  node[above,midway] {};


\node[draw,circle,below=1.2cm of F1] (E1) {};
\node[draw,circle,right=.8cm of E1] (E2) {};
\node[draw,circle,right=.8cm of E2] (E3) {};
\node[draw,circle,right=.8cm of E3] (E4) {};
\node[draw,circle,right=.8cm of E4] (E5) {};
\node[draw,circle,below=.8cm of E3] (EA) {};
\node[left=.8cm of E1] {$E_6$};

\draw (E1) -- (E2)  node[above,midway] {};
\draw (E2) -- (E3)  node[above,midway] {};
\draw (E3) -- (E4)  node[above,midway] {};
\draw (E4) -- (E5)  node[above,midway] {};
\draw (E3) -- (EA)  node[left,midway] {};


\node[draw,circle,below=1.8cm of E1] (EE1) {};
\node[draw,circle,right=.8cm of EE1] (EE2) {};
\node[draw,circle,right=.8cm of EE2] (EE3) {};
\node[draw,circle,right=.8cm of EE3] (EE4) {};
\node[draw,circle,right=.8cm of EE4] (EE5) {};
\node[draw,circle,right=.8cm of EE5] (EE6) {};
\node[draw,circle,below=.8cm of EE3] (EEA) {};
\node[left=.8cm of EE1] {$E_7$};

\draw (EE1) -- (EE2)  node[above,midway] {};
\draw (EE2) -- (EE3)  node[above,midway] {};
\draw (EE3) -- (EE4)  node[above,midway] {};
\draw (EE4) -- (EE5)  node[above,midway] {};
\draw (EE5) -- (EE6)  node[above,midway] {};
\draw (EE3) -- (EEA)  node[left,midway] {};


\node[draw,circle,below=1.8cm of EE1] (EEE1) {};
\node[draw,circle,right=.8cm of EEE1] (EEE2) {};
\node[draw,circle,right=.8cm of EEE2] (EEE3) {};
\node[draw,circle,right=.8cm of EEE3] (EEE4) {};
\node[draw,circle,right=.8cm of EEE4] (EEE5) {};
\node[draw,circle,right=.8cm of EEE5] (EEE6) {};
\node[draw,circle,right=.8cm of EEE6] (EEE7) {};
\node[draw,circle,below=.8cm of EEE3] (EEEA) {};
\node[left=.8cm of EEE1] {$E_8$};

\draw (EEE1) -- (EEE2)  node[above,midway] {};
\draw (EEE2) -- (EEE3)  node[above,midway] {};
\draw (EEE3) -- (EEE4)  node[above,midway] {};
\draw (EEE4) -- (EEE5)  node[above,midway] {};
\draw (EEE5) -- (EEE6)  node[above,midway] {};
\draw (EEE6) -- (EEE7)  node[above,midway] {};
\draw (EEE3) -- (EEEA)  node[left,midway] {};

\end{tikzpicture}
\caption{The irreducible spherical Coxeter diagrams}
\label{spheri_diag}
\end{minipage}
\begin{minipage}[t]{7.5cm}
\centering
\begin{tikzpicture}[thick,scale=0.6, every node/.style={transform shape}]
\node[draw,circle] (A1) at (0,0) {};
\node[draw,circle,above right=.8cm of A1] (A2) {};
\node[draw,circle,right=.8cm of A2] (A3) {};
\node[draw,circle,right=.8cm of A3] (A4) {};
\node[draw,circle,right=.8cm of A4] (A5) {};
\node[draw,circle,below right=.8cm of A5] (A6) {};
\node[draw,circle,below left=.8cm of A6] (A7) {};
\node[draw,circle,left=.8cm of A7] (A8) {};
\node[draw,circle,left=.8cm of A8] (A9) {};
\node[draw,circle,left=.8cm of A9] (A10) {};

\node[left=.8cm of A1] {$\tilde{A}_n\,(n\geq 2)$};

\draw (A1) -- (A2)  node[above,midway] {};
\draw (A2) -- (A3)  node[above,midway] {};
\draw (A3) -- (A4) node[] {};
\draw (A4) -- (A5) node[above,midway] {};
\draw (A5) -- (A6) node[] {};
\draw (A6) -- (A7) node[] {};
\draw (A7) -- (A8) node[] {};
\draw[loosely dotted,thick] (A8) -- (A9) node[] {};
\draw (A9) -- (A10) node[] {};
\draw (A10) -- (A1) node[] {};


\node[draw,circle,below=1.7cm of A1] (B1) {};
\node[draw,circle,right=.8cm of B1] (B2) {};
\node[draw,circle,right=.8cm of B2] (B3) {};
\node[draw,circle,right=1cm of B3] (B4) {};
\node[draw,circle,right=.8cm of B4] (B5) {};
\node[draw,circle,above right=.8cm of B5] (B6) {};
\node[draw,circle,below right=.8cm of B5] (B7) {};
\node[left=.8cm of B1] {$\tilde{B}_n\,(n\geq 3)$};

\draw (B1) -- (B2)  node[above,midway] {$4$};
\draw (B2) -- (B3)  node[above,midway] {};
\draw[loosely dotted,thick] (B3) -- (B4) node[] {};
\draw (B4) -- (B5) node[above,midway] {};
\draw (B5) -- (B6) node[above,midway] {};
\draw (B5) -- (B7) node[above,midway] {};


\node[draw,circle,below=1.5cm of B1] (C1) {};
\node[draw,circle,right=.8cm of C1] (C2) {};
\node[draw,circle,right=.8cm of C2] (C3) {};
\node[draw,circle,right=1cm of C3] (C4) {};
\node[draw,circle,right=.8cm of C4] (C5) {};
\node[left=.8cm of C1] (CCC) {$\tilde{C}_n\,(n\geq 3)$};

\draw (C1) -- (C2)  node[above,midway] {$4$};
\draw (C2) -- (C3)  node[above,midway] {};
\draw[loosely dotted,thick] (C3) -- (C4) node[] {};
\draw (C4) -- (C5) node[above,midway] {$4$};


\node[draw,circle,below=1cm of C1] (D1) {};
\node[draw,circle,below right=0.8cm of D1] (D3) {};
\node[draw,circle,below left=0.8cm of D3] (D2) {};
\node[draw,circle,right=.8cm of D3] (DA) {};
\node[draw,circle,right=1cm of DA] (DB) {};
\node[draw,circle,right=.8cm of DB] (D4) {};
\node[draw,circle, above right=.8cm of D4] (D5) {};
\node[draw,circle,below right=.8cm of D4] (D6) {};
\node[below=1.5cm of CCC] {$\tilde{D}_n\,(n\geq 4)$};

\draw (D1) -- (D3)  node[above,midway] {};
\draw (D2) -- (D3)  node[above,midway] {};
\draw (D3) -- (DA) node[above,midway] {};
\draw[loosely dotted] (DA) -- (DB);
\draw (D4) -- (DB) node[above,midway] {};
\draw (D4) -- (D5) node[above,midway] {};
\draw (D4) -- (D6) node[below,midway] {};


\node[draw,circle,below=2.5cm of D1] (I1) {};
\node[draw,circle,right=.8cm of I1] (I2) {};
\node[left=.8cm of I1] {$\tilde{A}_1$};

\draw (I1) -- (I2)  node[above,midway] {$\infty$};


\node[draw,circle,below=1.2cm of I1] (H1) {};
\node[draw,circle,right=.8cm of H1] (H2) {};
\node[draw,circle,right=.8cm of H2] (H3) {};
\node[left=.8cm of H1] {$\tilde{B}_2=\tilde{C}_2$};

\draw (H1) -- (H2)  node[above,midway] {$4$};
\draw (H2) -- (H3)  node[above,midway] {$4$};


\node[draw,circle,below=1.2cm of H1] (HH1) {};
\node[draw,circle,right=.8cm of HH1] (HH2) {};
\node[draw,circle,right=.8cm of HH2] (HH3) {};
\node[left=.8cm of HH1] {$\tilde{G}_2$};

\draw (HH1) -- (HH2)  node[above,midway] {$6$};
\draw (HH2) -- (HH3)  node[above,midway] {};


\node[draw,circle,below=1.2cm of HH1] (F1) {};
\node[draw,circle,right=.8cm of F1] (F2) {};
\node[draw,circle,right=.8cm of F2] (F3) {};
\node[draw,circle,right=.8cm of F3] (F4) {};
\node[draw,circle,right=.8cm of F4] (F5) {};
\node[left=.8cm of F1] {$\tilde{F}_4$};

\draw (F1) -- (F2)  node[above,midway] {};
\draw (F2) -- (F3)  node[above,midway] {$4$};
\draw (F3) -- (F4)  node[above,midway] {};
\draw (F4) -- (F5)  node[above,midway] {};


\node[draw,circle,below=1.2cm of F1] (E1) {};
\node[draw,circle,right=.8cm of E1] (E2) {};
\node[draw,circle,right=.8cm of E2] (E3) {};
\node[draw,circle,right=.8cm of E3] (E4) {};
\node[draw,circle,right=.8cm of E4] (E5) {};
\node[draw,circle,below=.8cm of E3] (EA) {};
\node[draw,circle,below=.8cm of EA] (EB) {};
\node[left=.8cm of E1] {$\tilde{E}_6$};

\draw (E1) -- (E2)  node[above,midway] {};
\draw (E2) -- (E3)  node[above,midway] {};
\draw (E3) -- (E4)  node[above,midway] {};
\draw (E4) -- (E5)  node[above,midway] {};
\draw (E3) -- (EA)  node[left,midway] {};
\draw (EB) -- (EA)  node[left,midway] {};


\node[draw,circle,below=3cm of E1] (EE1) {};
\node[draw,circle,right=.8cm of EE1] (EEB) {};
\node[draw,circle,right=.8cm of EEB] (EE2) {};
\node[draw,circle,right=.8cm of EE2] (EE3) {};
\node[draw,circle,right=.8cm of EE3] (EE4) {};
\node[draw,circle,right=.8cm of EE4] (EE5) {};
\node[draw,circle,right=.8cm of EE5] (EE6) {};
\node[draw,circle,below=.8cm of EE3] (EEA) {};
\node[left=.8cm of EE1] {$\tilde{E}_7$};

\draw (EE1) -- (EEB)  node[above,midway] {};
\draw (EE2) -- (EEB)  node[above,midway] {};
\draw (EE2) -- (EE3)  node[above,midway] {};
\draw (EE3) -- (EE4)  node[above,midway] {};
\draw (EE4) -- (EE5)  node[above,midway] {};
\draw (EE5) -- (EE6)  node[above,midway] {};
\draw (EE3) -- (EEA)  node[left,midway] {};


\node[draw,circle,below=1.8cm of EE1] (EEE1) {};
\node[draw,circle,right=.8cm of EEE1] (EEE2) {};
\node[draw,circle,right=.8cm of EEE2] (EEE3) {};
\node[draw,circle,right=.8cm of EEE3] (EEE4) {};
\node[draw,circle,right=.8cm of EEE4] (EEE5) {};
\node[draw,circle,right=.8cm of EEE5] (EEE6) {};
\node[draw,circle,right=.8cm of EEE6] (EEE7) {};
\node[draw,circle,right=.8cm of EEE7] (EEE8) {};
\node[draw,circle,below=.8cm of EEE3] (EEEA) {};
\node[left=.8cm of EEE1] {$\tilde{E}_8$};

\draw (EEE1) -- (EEE2)  node[above,midway] {};
\draw (EEE2) -- (EEE3)  node[above,midway] {};
\draw (EEE3) -- (EEE4)  node[above,midway] {};
\draw (EEE4) -- (EEE5)  node[above,midway] {};
\draw (EEE5) -- (EEE6)  node[above,midway] {};
\draw (EEE6) -- (EEE7)  node[above,midway] {};
\draw (EEE8) -- (EEE7)  node[above,midway] {};
\draw (EEE3) -- (EEEA)  node[left,midway] {};
\end{tikzpicture}
\caption{The irreducible affine Coxeter diagrams}
\label{affi_diag}
\end{minipage}
\end{table}

\clearpage

\subsection{Lann{\'e}r Coxeter groups}

A Coxeter group associated to a large irreducible perfect simplex is called a Lann{\'e}r Coxeter group, after Lann{\'e}r \cite{MR0042129} first enumerated them in 1950. These groups are very limited, for example they exist only in dimension $2,3,4$, and there exist infinitely many in dimension $2$, $9$ in dimension $3$, and $5$ in dimension $4$.

%
\begin{table}[ht]
\begin{tabular}{ccc}
\begin{minipage}[t]{7.5cm}
\centering
\begin{tabular}{cccc}
\begin{tikzpicture}[thick,scale=0.6, every node/.style={transform shape}]
\node[draw,circle] (A1) at (0,0) {};
\node[draw,circle] (A2) at (1.5cm,0) {};
\node[draw,circle] (A3) at (60:1.5cm) {};

\draw (A3) -- (A1) node[left,midway] {$p$};
\draw (A2) -- (A3)  node[right,midway] {$q$};
\draw (A1) -- (A2)  node[below,midway] {$r$};
\end{tikzpicture}
&
&
\begin{tikzpicture}[thick,scale=0.6, every node/.style={transform shape}]
\node[draw,circle] (A1) at (0,0) {};
\node[draw,circle] (A2) at (1cm,0) {};
\node[draw,circle] (A3) at (2cm,0) {};

\draw (A1) -- (A2)  node[above,midway] {$p$};
\draw (A2) -- (A3)  node[above,midway] {$q$};
\end{tikzpicture}
\\
\\
$\frac{1}{p}+\frac{1}{q}+\frac{1}{r}<1$
&
&
$\frac{1}{p}+\frac{1}{q}<\frac 1 2$
\\
\\
\end{tabular}
\caption{The Lannér Coxeter groups of the Lannér polygons}
\end{minipage}

&

\begin{minipage}[t]{7.5cm}
\centering
\begin{tabular}{ccc}
\begin{tikzpicture}[thick,scale=0.6, every node/.style={transform shape}]
\node[draw,circle] (C1) at (0,0) {};
 \node[draw,circle,right=.8cm of C1] (C2) {};
\node[draw,circle,right=.8cm of C2] (C3) {};
\node[draw,circle,right=.8cm of C3] (C4) {};

\draw (C1) -- (C2)  node[above,midway] {};
\draw (C2) -- (C3)  node[above,midway] {$5$};
\draw (C3) -- (C4) node[] {};


\node[draw,circle,below=1.5cm of C1] (CC1) {};
\node[draw,circle,right=.8cm of CC1] (CC2) {};
\node[draw,circle,right=.8cm of CC2] (CC3) {};
\node[draw,circle,right=1cm of CC3] (CC4) {};

\draw (CC1) -- (CC2)  node[above,midway] {$5$};
\draw (CC2) -- (CC3)  node[above,midway] {};
\draw (CC3) -- (CC4) node[above,midway] {$4$};


\node[draw,circle,below=1.5cm of CC1] (D1) {};
\node[draw,circle,right=.8cm of D1] (D2) {};
\node[draw,circle,right=.8cm of D2] (D3) {};
\node[draw,circle,right=1cm of D3] (D4) {};

\draw (D1) -- (D2)  node[above,midway] {$5$};
\draw (D2) -- (D3)  node[above,midway] {};
\draw (D3) -- (D4) node[above,midway] {$5$};



\node[draw,circle,below=1.7cm of D1] (B4) {};
\node[draw,circle,right=.8cm of B4] (B5) {};
\node[draw,circle,above right=.8cm of B5] (B6) {};
\node[draw,circle,below right=.8cm of B5] (B7) {};

\draw (B4) -- (B5) node[above,midway] {$5$};
\draw (B5) -- (B6) node[above,midway] {};
\draw (B5) -- (B7) node[above,midway] {};

\end{tikzpicture}
&
&
\begin{tikzpicture}[thick,scale=0.6, every node/.style={transform shape}]
\node[draw,circle] (C1) at (0,0) {};
\node[draw,circle,right=.8cm of C1] (C2) {};
\node[draw,circle,below=.8cm of C2] (C3) {};
\node[draw,circle,left=.8cm of C3] (C4) {};

\draw (C1) -- (C2)  node[above,midway] {$4$};
\draw (C2) -- (C3)  node[above,midway] {};
\draw (C3) -- (C4) node[above,midway] {};
\draw (C4) -- (C1) node[above,midway] {};


\node[draw,circle,right=2.5cm of C1] (D1) {};
\node[draw,circle,right=.8cm of D1] (D2) {};
\node[draw,circle,below=.8cm of D2] (D3) {};
\node[draw,circle,left=.8cm of D3] (D4) {};

\draw (D1) -- (D2)  node[above,midway] {$5$};
\draw (D2) -- (D3)  node[above,midway] {};
\draw (D3) -- (D4) node[above,midway] {};
\draw (D4) -- (D1) node[above,midway] {};


\node[draw,circle,below=2.5cm of C1] (E1) {};
\node[draw,circle,right=.8cm of E1] (E2) {};
\node[draw,circle,below=.8cm of E2] (E3) {};
\node[draw,circle,left=.8cm of E3] (E4) {};

\draw (E1) -- (E2)  node[above,midway] {$4$};
\draw (E2) -- (E3)  node[above,midway] {};
\draw (E3) -- (E4) node[above,midway] {$4$};
\draw (E4) -- (E1) node[above,midway] {};


 \node[draw,circle,right=2.5cm of E1] (F1) {};
\node[draw,circle,right=.8cm of F1] (F2) {};
\node[draw,circle,below=.8cm of F2] (F3) {};
\node[draw,circle,left=.8cm of F3] (F4) {};

\draw (F1) -- (F2)  node[above,midway] {$5$};
\draw (F2) -- (F3)  node[above,midway] {};
\draw (F3) -- (F4) node[above,midway] {$4$};
\draw (F4) -- (F1) node[above,midway] {};


\node[draw,circle,below=2.5cm of E1] (G1) {};
\node[draw,circle,right=.8cm of G1] (G2) {};
\node[draw,circle,below=.8cm of G2] (G3) {};
\node[draw,circle,left=.8cm of G3] (G4) {};

\draw (G1) -- (G2)  node[above,midway] {$5$};
\draw (G2) -- (G3)  node[above,midway] {};
\draw (G3) -- (G4) node[above,midway] {$5$};
\draw (G4) -- (G1) node[above,midway] {};

\end{tikzpicture}
\\
\\
\end{tabular}
\caption{The Lannér Coxeter groups of the Lannér polyhedra}\label{table:Lanner_dim3}
\end{minipage}
\end{tabular}

\centering
\begin{tabular}{ccccc}
\begin{tikzpicture}[thick,scale=0.6, every node/.style={transform shape}]
\node[draw,circle] (A) at (0,0){};

\node[draw,circle] (B) at (1,0){};
\node[draw,circle] (C) at (-36:1.5){};
\node[draw,circle,] (D) at (0.5,-1.5){};
\node[draw,circle,] (E) at (252:1){};

\draw (A) -- (B) node[above,midway] {$4$};
\draw (B) -- (C) node[above,midway] {};
\draw (C) -- (D) node[above,midway] {};
\draw (D) -- (E) node[above,midway] {};
\draw (E) -- (A) node[above,midway] {};
\end{tikzpicture}
&
\begin{tikzpicture}[thick,scale=0.6, every node/.style={transform shape}]
\node[draw,circle] (B3) at (0,0){};
\node[draw,circle,right=.8cm of B3] (B4) {};
\node[draw,circle,right=.8cm of B4] (B5) {};
\node[draw,circle,above right=.8cm of B5] (B6) {};
\node[draw,circle,below right=.8cm of B5] (B7) {};

\draw (B3) -- (B4) node[above,midway] {$5$};
\draw (B4) -- (B5) node[above,midway] {};
\draw (B5) -- (B6) node[above,midway] {};
\draw (B5) -- (B7) node[above,midway] {};
\end{tikzpicture}
\\
\begin{tikzpicture}[thick,scale=0.6, every node/.style={transform shape}]
\node[draw,circle] (C1) at (0,0){};
\node[draw,circle,right=.8cm of C1] (C2) {};
\node[draw,circle,right=.8cm of C2] (C3) {};
\node[draw,circle,right=.8cm of C3] (C4) {};
\node[draw,circle,right=.8cm of C4] (C5) {};

\draw (C1) -- (C2)  node[above,midway] {$5$};
\draw (C2) -- (C3)  node[above,midway] {};
\draw (C3) -- (C4) node[] {};
\draw (C4) -- (C5) node[above,midway] {};
\end{tikzpicture}
&
\begin{tikzpicture}[thick,scale=0.6, every node/.style={transform shape}]
\node[draw,circle] (CC1) at (0,0){};
\node[draw,circle,right=.8cm of CC1] (CC2) {};
\node[draw,circle,right=.8cm of CC2] (CC3) {};
\node[draw,circle,right=.8cm of CC3] (CC4) {};
\node[draw,circle,right=.8cm of CC4] (CC5) {};

\draw (CC1) -- (CC2)  node[above,midway] {$5$};
\draw (CC2) -- (CC3)  node[above,midway] {};
\draw (CC3) -- (CC4) node[] {};
\draw (CC4) -- (CC5) node[above,midway] {$5$};
\end{tikzpicture}
&
\begin{tikzpicture}[thick,scale=0.6, every node/.style={transform shape}]
\node[draw,circle] (D1) at (0,0){};
\node[draw,circle,right=.8cm of D1] (D2) {};
\node[draw,circle,right=.8cm of D2] (D3) {};
\node[draw,circle,right=.8cm of D3] (D4) {};
\node[draw,circle,right=.8cm of D4] (D5) {};

\draw (D1) -- (D2)  node[above,midway] {$5$};
\draw (D2) -- (D3)  node[above,midway] {};
\draw (D3) -- (D4) node[] {};
\draw (D4) -- (D5) node[above,midway] {$4$};
\end{tikzpicture}
\\
\\
\end{tabular}
\caption{The Lannér Coxeter groups of the Lannér $4$-polytopes}
\end{table}



\bibliographystyle{alpha}

\end{document}